\documentclass[a4paper, 11pt]{memoir} 

\usepackage[UKenglish]{babel}
\usepackage[latin1]{inputenc}
\usepackage[T1]{fontenc}
\usepackage{bbm}
\usepackage{amsmath}
\usepackage{amssymb}
\usepackage{amsthm}
\usepackage{tikz}
\usetikzlibrary{positioning, graphs, trees}
\usepackage{flafter}
\usepackage{hyperref}

\hypersetup{hidelinks} 
\hypersetup{linktocpage} 
\hypersetup{bookmarksnumbered} 

\setlrmarginsandblock{3.5 cm}{2.5 cm}{*} 
\setulmarginsandblock{2.5cm}{2.5cm}{*} 
\setheadfoot{\baselineskip}{6ex} 
\setheaderspaces{*}{\baselineskip}{*} 
\setmarginnotes{1pt}{1pt}{0pt} 
\checkandfixthelayout[nearest]

\nonzeroparskip 

\setfloatlocations{figure}{!ht} 

\makepagestyle{mystyle} 
\makeevenhead{mystyle}{\thepage}{\MakeUppercase{\theauthor}}{} 
\makeoddhead{mystyle}{}{\MakeUppercase{\thetitle}}{\thepage}
\aliaspagestyle{chapter}{mystyle} 

\chapterstyle{article} 
\setsecnumdepth{subsection} 
\counterwithout{section}{chapter} 
\setparaheadstyle{\bfseries \addperiod} 

\abstractrunin 
\AtBeginDocument{\setlength{\abstitleskip}{-\absparindent}}  
\abslabeldelim{. } 

\settocdepth{subsection} 

\addto\captionsenglish{} 
\nobibintoc 

\captionnamefont{\bfseries} 
\captiontitlefont{\itshape} 
\hangcaption 
\newsubfloat{figure} 
\counterwithout{figure}{section} 

\newtheoremstyle{theorem}{}{}{\itshape}{}{\bfseries}{}{\newline}{\thmname{#1}\thmnumber{ #2}\thmnote{ (#3)}}
\newtheoremstyle{proposition}{}{}{\itshape}{}{\bfseries}{}{1em}{}
\newtheoremstyle{standard}{}{}{}{}{\bfseries}{}{1em}{}
\newtheoremstyle{unnumbered}{}{}{}{}{\bfseries}{}{1em}{\thmname{#1}}

\theoremstyle{theorem} 
\newtheorem{thm}{Theorem}[section] 

\theoremstyle{proposition} 
\newtheorem{prop}[thm]{Proposition} 
\newtheorem{coro}[thm]{Corollary}

\theoremstyle{standard} 

\theoremstyle{unnumbered} 
\newtheorem{rem}{Remark} 
\newtheorem{exmpl}{Example}

\DeclareMathOperator{\Shift}{T} 
\DeclareMathOperator{\Sub}{Sub} 
\DeclareMathOperator{\Jac}{H} 
\DeclareMathOperator{\GL}{GL} 

\newcommand{\alphab}{\mathcal{A}} 
\newcommand{\alphabEv}{\widetilde{\mathcal{A}}} 
\newcommand{\eventNr}{\widetilde{K}} 
\newcommand{\length}[1]{|#1|} 
\newcommand{\cylin}[2]{C_{#2}(#1)} 
\newcommand{\block}[1]{p^{(#1)}} 
\newcommand{\subshift}{\Omega} 
\newcommand{\langu}[1]{\Sub(#1)} 
\newcommand{\comp}{\mathcal{C}} 
\newcommand{\growth}{\mathcal{G}} 
\newcommand{\pali}{\mathcal{P}} 
\newcommand{\repe}{\mathcal{R}} 
\newcommand{\alleB}{F} 
\newcommand{\card}[1]{\# #1} 
\newcommand{\cmplmt}[1]{#1^{\mathsf{c}}} 
\newcommand{\rev}[1]{\left. #1 \right.^{\mathsf{R}}} 
\newcommand{\debruijn}[1]{\mathcal{G}_{#1}} 
\newcommand{\vertices}[1]{\mathcal{V}_{#1}} 
\newcommand{\edges}[1]{\mathcal{E}_{#1}} 

\title{Combinatorics of One-Dimensional Simple Toeplitz Subshifts} 
\posttitle{\par\end{center}} 
\author{Daniel Sell\thanks{\textsc{Friedrich-Schiller-Universität Jena, Institut für Mathematik, 07743 Jena, Germany} \newline \textit{E-mail address:} \texttt{daniel.sell@uni-jena.de}}} 
\predate{} 
\date{} 
\postdate{\vspace{-1ex}} 
\hypersetup{pdftitle=\thetitle, pdfauthor=\theauthor} 

\begin{document}
\allowdisplaybreaks 

\maketitle

\begin{abstract}
This paper provides a systematic study of fundamental combinatorial properties of one-dimensional, two-sided infinite simple Toeplitz subshifts. Explicit formulas for the complexity function, the palindrome complexity function and the repetitivity function are proven. Moreover, a complete description of the de Bruijn graphs of the subshifts is given. Finally, the Boshernitzan condition is characterised in terms of combinatorial quantities, based on a recent result of Liu and Qu (\cite{LiuQu_Simple}). Particular simple characterisations are provided for simple Toeplitz subshifts that correspond to the orbital Schreier graphs of the family of Grigorchuk's groups, a class of subshifts that serves as main example throughout the paper.
\end{abstract}

\tableofcontents* 
\thispagestyle{plain} 


\section{Introduction}

The investigation of Toeplitz subshifts has a long history. In fact, Toeplitz subshifts have been rediscovered in various contexts and serve as a prime source of (counter)examples, see for instance \cite{JacobsKeane_01Toeplitz, Wil_ToepNotUniqErgod} or \cite{GroeJae_PowEntr}. Our investigation of Toeplitz subshifts is motivated by their utilisation in two seemingly unrelated fields. Firstly, for Grigorchuk's group it has been shown in \cite{GLN_SpectraSchreierAndSO} (see \cite{GLN_Survey} as well) that the Laplacian on the two-sided Schreier graphs of the group is unitary equivalent to the Jacobi operator on a Toeplitz subshift. Thus the investigation of Toeplitz subshifts can help to improve the understanding of self similar groups, see \cite{Voro_SubstGrigo} as well. Secondly, Toeplitz words have also become more popular as a model for quasicrystals over the last decade, see for instance \cite{BJL_ToeplModelSet}. In particular, spectral properties of Schrödinger operators with a potential that is given by an element of a Toeplitz subshift have received quite some attention and are studied for example in \cite{LiuQu_Simple}, \cite{LiuQu_Bounded} and \cite{DamLiuQu_PatternSturm}. The present work addresses fundamental combinatorial properties of so called simple Toeplitz words and their associated subshifts.

\paragraph{Simple Toeplitz words}
\label{para:IntroSimpToepW}

Toeplitz words were introduced in \cite{JacobsKeane_01Toeplitz} as elements in  $ \{ 0, 1 \}^{ \mathbb{Z}} $ with a certain regularity. The regularity stems from the construction of Toeplitz words via so called ``partial words''. These are periodic words with letters in $\{ 0, 1 \}$ and some undetermined positions (``holes''). The holes are then successively filled with other partial words. It is required in \cite{JacobsKeane_01Toeplitz} that no undetermined part remains in the limit of the hole filling process. Thus, in the limit word $\omega \in \{0, 1\}^{\mathbb{Z}} $ every letter $\omega( j )$ is repeated periodically, but the period depends on the position $j$, that is:
\[ \forall j \in \mathbb{Z} \quad \exists p \in \mathbb{N} \quad \text{such that} \quad \omega_{j} = \omega_{j + k p} \quad \text{holds for all} \quad k \in \mathbb{Z} \, . \]
In the following, we will consider the subclass of so called simple Toeplitz words. In their construction, every partial words consists of the repetition of a single letter and there is exactly one hole in every partial word. Note that the details of the notion of simple Toeplitz words in the literature differ corresponding to the different settings that are considered. For example, simple Toeplitz words are defined in \cite{KamZamb_MaxPattCompl} in the context of one-sided infinite words over set of set of cardinality two without an undetermined part. In \cite{QRWX_PatCompDDim}, simple Toeplitz words are two-sided infinite ``words'' in $\mathbb{Z}^{d}$ that must not have an undetermined part. In this work, we will use the definition from \cite{LiuQu_Simple} and define simple Toeplitz words as one-dimensional, two-sided infinite words $ \omega \in \alphab^{\mathbb{Z}} $ over a finite set $\alphab$. Words with an undetermined part are allowed as well, after filling the undetermined position appropriately.

\paragraph{Schreier graphs of self similar groups}
Let $X$ be a finite set and consider $ \cup_{k = 0}^{\infty} X^{k} $. This union can be thought of as the vertex set of a regular rooted tree, where two vertices are connected if and only if they are of the form $u$ and $ua$ with $ u \in \cup_{k = 0}^{\infty} X^{k} $ and $ a \in X $. A group $G$ of automorphisms on such a tree is called self similar if for every $ g \in G $ and every $ a \in X $ there exist elements $ h \in G $ and $ b \in X $ such that $ g( a u ) = b h( u ) $ holds for all  $ u \in \cup_{k = 0}^{\infty} X^{k} $. In other words, for every group element $g$ and every $ a \in X $, there is a group element $h$ such that $g$ acts on the subtree below $a$ in the same way as $h$ acts on the whole tree. Such groups have been studied during the last decades since they provide examples of groups with interesting properties (see for example \cite{BGN_FractalGrSets, BGS_BranchGr, Nekrash_SelfSimGr} and the references therein). For instance, Grigorchuk's group, introduced in \cite{Grig_BurnsideRuss}, is a self similar group of automorphisms on the binary tree and was the first example of a group with intermediate growth (\cite{Grig_DegrOfGrowthRuss}). With $X$ chosen as $ X = \{ 0, 1 \} $, Grigorchuk's group is generated by the four elements $a$, $b$, $c$, $d$ that satisfy
\begin{align*}
a( 0u ) &= 1u & b( 0u ) &= 0 a( u ) & c( 0u ) &= 0 a( u ) & d( 0u ) &= 0 u \\
a( 1u ) &= 0u & b( 1u ) &= 1 c( u ) & c( 1u ) &= 1 d( u ) & d( 1u ) &= 1 b( u ) \, .
\end{align*}
In fact, the above definition was generalized in \cite{Grig_DegrOfGrowthRuss} to a whole family $(G_{\omega})_{\omega}$ of groups: Let $a$ be defined as above. Moreover, let $\omega \in \{ \pi_{b}, \pi_{c}, \pi_{d} \}^{\mathbb{N}} $ be a sequence of the maps
\[ \pi_{b} : b \mapsto id ,\, c \mapsto a , \, d \mapsto a \; , \quad \pi_{c} : b \mapsto a ,\, c \mapsto id , \, d \mapsto a \; , \quad \pi_{d} : b \mapsto a ,\, c \mapsto a , \, d \mapsto id \, . \]
We can now consider the automorphism $\widehat{b}$ that acts like $\omega_{1}(b)$ below the vertex $0$, like $ \omega_{2}(b) $ below the vertex $10$, like $\omega_{3}(b)$ below the vertex $110$, etc. Similarly, we define the automorphism $\widehat{c}$ by the action of $( \omega_{1}(c), \omega_{2}(c), \omega_{3}(c), \hdots ) $ and the the automorphism $\widehat{d}$ by the action of $( \omega_{1}(d), \omega_{2}(d), \omega_{3}(d), \hdots ) $, see Figure~\ref{fig:FamGrigGenerat}.

\begin{figure}
\centering
\footnotesize
%
%
\begin{minipage}[b]{0.3\textwidth}
\centering
\begin{tikzpicture}
[every node/.style ={inner sep=0pt, minimum height=0pt,draw, circle}, every label/.style ={label distance=0.5ex, draw=none, rectangle}, level distance = 2em, sibling distance=6ex]
\node{}
	child{ node[label=below:$\omega_{1}( b )$]{}
	}
	child{ node{}
		child{ node[label=below:$\omega_{2}( b )$]{}
		}
		child{ node{}
			child{ node[label=below:$\omega_{3}( b )$]{} }
			child{ node[label=below:$\ddots$]{} }
		}
	};
\end{tikzpicture}
\subcaption{The generator $\widehat{b}$\label{fig:FamGrigB}}
\end{minipage}
\hfill
%
%
\begin{minipage}[b]{0.3\textwidth}
\centering
\begin{tikzpicture}
[every node/.style ={inner sep=0pt, minimum height=0pt,draw, circle}, every label/.style ={label distance=0.5ex, draw=none, rectangle}, level distance = 2em, sibling distance=6ex]
\node{}
	child{ node[label=below:$\omega_{1}( c )$]{}
	}
	child{ node{}
		child{ node[label=below:$\omega_{2}( c )$]{}
		}
		child{ node{}
			child{ node[label=below:$\omega_{3}( c )$]{} }
			child{ node[label=below:$\ddots$]{} }
		}
	};
\end{tikzpicture}
\subcaption{The generator $\widehat{c}$\label{fig:FamGrigC}}
\end{minipage}
\hfill
%
%
\begin{minipage}[b]{0.3\textwidth}
\centering
\begin{tikzpicture}
[every node/.style ={inner sep=0pt, minimum height=0pt,draw, circle}, every label/.style ={label distance=0.5ex, draw=none, rectangle}, level distance = 2em, sibling distance=6ex]
\node{}
	child{ node[label=below:$\omega_{1}( d )$]{}
	}
	child{ node{}
		child{ node[label=below:$\omega_{2}( d )$]{}
		}
		child{ node{}
			child{ node[label=below:$\omega_{3}( d )$]{} }
			child{ node[label=below:$\ddots$]{} }
		}
	};
\end{tikzpicture}
\subcaption{The generator $\widehat{d}$\label{fig:FamGrigD}}
\end{minipage}
\caption{The generators $\widehat{b}$, $\widehat{c}$ and $\widehat{d}$ of the group $G_{\omega}$}
\label{fig:FamGrigGenerat}
\normalsize
\end{figure}

\noindent
For every such sequence $\omega \in \{ \pi_{b}, \pi_{c}, \pi_{d} \}^{\mathbb{N}} $ we define $G_{\omega}$ as the group that is generated by $a$, $\widehat{b}$, $\widehat{c}$ and $\widehat{d}$.  Note that the periodic sequence $ \omega = (\pi_{d}, \pi_{c}, \pi_{b} , \hdots ) $ yields precisely Grigorchuk's group. The family $ \{ \, G_{\omega} : \omega \in \{ \pi_{b}, \pi_{c}, \pi_{d} \}^{\mathbb{N}} \, \} $ has been studied heavily in the past. Particular relevant to our investigation are the recent works \cite{MBon_TopoFullGr} and \cite{GLN_SpectraSchreierAndSO}, since they connect self similar groups and Toeplitz subshifts. More precisely, it was shown in \cite{MBon_TopoFullGr} that every group $G_{\omega}$ can be embedded into the topological full group of a minimal subshift. This subshift is constructed from so-called Schreier graphs and, although not explicitly mentioned in \cite{MBon_TopoFullGr}, can be shown to be a Toeplitz subshift. In \cite{GLN_SpectraSchreierAndSO}, the Laplacians associated to the Schreier graphs of Grigorchuk's group are treated. A key insight is that these operators are unitarily equivalent to certain Jacobi operators associated to a subshift. This subshift is noted to be a Toeplitz subshift and turns out to be the same as the one that was described in \cite{MBon_TopoFullGr}. However, in \cite{GLN_SpectraSchreierAndSO} the subshift is obtained by a different construction, which is based on the fact that the group action on the tree induces an action on the boundary $X^{\mathbb{N}}$ as well. Hence the following directed, labelled graph can be defined: The set of vertices is given by $ ( \cup_{k \in \mathbb{N}} X^{k} ) \cup X^{\mathbb{N}} $. There is an edge from vertex $u$ to vertex $v$ labelled $s \in \{ a,b,c,d \} $  if and only if $ s( u ) = v $ holds. The connected components of this graph are called Schreier graphs. For a vertex $u \in X^{k}$, the connected component of $u$ corresponds to the $k$-th level of the tree. The Schreier graph of the level $k+1$ can be obtained from two copies of the level-$k$-Schreier graph by connecting them in a certain way that is specified by the element $\omega_{k}$ in the sequence $ \omega = (\pi_{d}, \pi_{c}, \pi_{b} , \hdots ) $. This yields a structure for the Schreier graphs that is similar to the one of a Toeplitz word (see \cite{GLN_SpectraSchreierAndSO} for details): Roughly speaking, every second connection in the graph corresponds to the action of the generator $a$. Of the remaining ``holes'' in the graph, every second connection corresponds to the action defined by $\omega_{1} = \pi_{d}$. Of the connections that are still missing after that, every second one corresponds to the action defined by $\omega_{2} = \pi_{c}$, etc. In fact, this construction of the subshift can be generalized to the whole family of groups, cf. the introductory section in \cite{GLN_SpectraSchreierAndSO}.

\paragraph{Schrödinger operators on quasicrystals}
Aperiodic words can serve as mathematical models for quasicrystals. The quantum mechanical properties of the quasicrystal are then described by the spectrum of the Schrödinger operator whose potential is given by the aperiodic word. For the relevance of quasicrystals see for example the recent books \cite{BaakeGrimm_Aperio} and \cite{MathAperioOrd}. One much studied class of models are Sturmian words, that is, words that have exactly $L+1$ subwords of length $L$ for every $L \in \mathbb{N}$. By the famous Morse-Hedlund theorem (\cite{MorseHedl_SymbDyn}), this is the least possible number of subwords that an aperiodic word can exhibit. Accordingly, these models have played a major role in the investigation of quasicrystals and in particular the associated Schrödinger operators have attracted a lot of attention (\cite{KKT_LocalProbl1D, OPRSS_1DSEqn, Cas_SymbDyn, Suto_QuasiPerSO, Suto_SingContSpect, BellIochScopTest_SpecProp, DamLenz_UniSpectProp1, DamLenz_UniSpectProp2, DamKillipLenz_UniSpectProp3}). By \cite{BellIochScopTest_SpecProp}, the spectrum of a discrete Schrödinger operator with Sturmian potential is always a Cantor set of Lebesgue measure zero. Moreover, the absence of eigenvalues for all Sturmian Potentials was shown in \cite{DamKillipLenz_UniSpectProp3}. Together these two results imply purely singular continuous spectrum.
In recent years, pattern Sturmian models, which are a generalization of Sturmian models, have become a focus of research. Pattern Sturmian words (in the sense of \cite{KamZamb_SequEntro}) have exactly $ 2L $ different subsets of $L$ elements for every $L \in \mathbb{N}$. This is the least possible number of subwords, that an aperiodic word can exhibit and every Sturmian word is pattern Sturmian as well (see \cite{KamZamb_SequEntro}). In \cite{DamLiuQu_PatternSturm}, the above mentioned results of \cite{BellIochScopTest_SpecProp} and \cite{DamKillipLenz_UniSpectProp3} could be generalized to these words: The spectrum of a discrete Schrödinger operator with a potential that is given by a pattern Sturmian Toeplitz sequence, has zero Lebesgue measure and is purely singular continuous. Moreover, it was shown in \cite{GKBYM_MaxPatternToepl} that one-sided infinite simple Toeplitz words over a 2-letter alphabet without an undetermined part are pattern Sturmian.

\paragraph{Content of the paper}
Because of their importance for quasicrystals and self similar groups, the focus of the research has often been on the spectral properties of Schrödinger operators. Here, we give a systematic discussion of fundamental combinatorial properties of simple Toeplitz subshifts instead. The complexity function and the repetitivity function of simple Toeplitz subshifts are explicitly computed in full generality, i.e. over any finite alphabet $\alphab$ and with arbitrary coding sequences $ (a_{k}) \in \alphab^{\mathbb{N}_{0}} $ and $ (n_{k}) \in ( \mathbb{N} \setminus \{ 1 \} )^{\mathbb{N}_{0}} $. This substantially generalizes and extends earlier pieces of work dealing with special cases: In \cite{GLN_Combinat}, the complexity was determined for the subshift that is associated to Grigorchuk's group, i.e. $ \alphab = \{ a,x,y,z \} $, $ (a_{k}) = (a, x, y, z, x, y, z, \hdots) $ and $ (n_{k}) = (2, 2, 2, 2, \hdots) $. In \cite{DKMSS_Regul-Article}, simple Toeplitz subshifts of the form  $ \alphab = \{ a,x,y,z \} $, $ (a_{k}) = (a, x, y, z, x, y, z, \hdots) $ and $ (n_{k}) = (2, 2^{l_{1}}, 2^{l_{2}}, 2^{l_{3}} \hdots) $ with $ l_{1}, l_{2}, l_{3}, \hdots \in \mathbb{N} $ are considered. For them, among other things, a formula for the complexity, estimates for the repetitivity and a characterization of $\alpha$-repetitivity were obtained.

Before discussing the combinatorial properties, Section~\ref{sec:SimpleTSubsh} gives the definition of simple Toeplitz words that will be used throughout this work. Moreover, the associated simple Toeplitz subshifts and the main examples are introduced. We call them the Grigorchuk subshift and the generalized Grigorchuk subshifts. In Section~\ref{sec:SubwordComp}, an explicit formula for the subword complexity of a simple Toeplitz subshift is derived. For this, complexity bounds at certain word lengths are proven and the growth rate of the complexity is estimated. Based on these results, the sequence of de Bruijn graphs is investigated in Section~\ref{sec:DeBruijnPali}. A particular reflection symmetry in the graphs yields an explicit formula for the palindrome complexity. The analysis of combinatorial properties continuous in Section~\ref{sec:Repe} with the deduction of a formula for the repetitivity function, followed by a short discussion of $\alpha$-repetitivity. Finally, the Boshernitzan condition for simple Toeplitz subshifts is studied in Section~\ref{sec:BoshJacobi} and its implication for the spectrum of  Jacobi operators is briefly reviewed.


\section{Simple Toeplitz Subshifts}
\label{sec:SimpleTSubsh}

In this section, our objects of interest are defined. As in \cite{LiuQu_Simple}, we consider simple Toeplitz words which are one-dimensional, two-sided infinite words over a finite alphabet and may have an undetermined position. Some basic properties of simple Toeplitz words are given in the first subsection. In the second part, different ways of associating subshifts to simple Toeplitz words are discussed. Our main examples, the Grigorchuk subshift and what we call the generalized Grigorchuk subshifts, are defined. 

\subsection{Simple Toeplitz Words}
\label{subsec:SimpleTWords}

Let $\alphab$ be a finite set, called the alphabet. Its elements are referred to as letters. A word of length $L \in \mathbb{N}_{0} $ is an element $ u = u(1) \hdots u(L) $ of $ \alphab^{\{ 1 , \hdots , L \}} $. For a given (finite) word $u$, we use $\length{u} \in \mathbb{N}_{0} $ to denote its length, where the word of length zero is called the empty word. A two-sided infinite word is an element $ \alpha = \hdots \alpha(-1) \alpha(0) \alpha(1) \hdots $ of $ \alphab^{\mathbb{Z}}$. We consider the discrete topology on $\alphab$ and equip $\alphab^{\mathbb{Z}}$ with the product topology. By $\cylin{u}{j}$ we denote the cylinder set $ \cylin{u}{j} := \{ \alpha \in \alphab^{\mathbb{Z}}: \alpha(j) = u(1), \hdots , \alpha(j+|u|-1) = u(\length{u}) \} $ of those two-sided infinite words, in which the finite word $ u $ appears at position $j$. The cylinder sets form a base of open sets in $\alphab^{\mathbb{Z}}$. Actually, the topological space $\alphab^{\mathbb{Z}}$ is metrizable and one metric is defined by
\[ d( \alpha_{1}, \alpha_{2} ) := \sum_{j = - \infty}^{\infty} \frac{ \delta( \alpha_{1}(j), \alpha_{2}(j) ) }{ 2^{|j|} } \, , \quad \text{where } \delta( a_{1}, a_{2} ) := \begin{cases} 0 & \text{ if } a_{1} = a_{2} \\ 1& \text{ if } a_{1} \neq a_{2} \end{cases} \]
for two-sided infinite words $ \alpha_{1}, \alpha_{2} \in \alphab^{\mathbb{Z}} $ and letters $a_{1}, a_{2} \in \alphab$. Two words are close in this topology if they agree on a large interval around the origin. Thus a sequence of two-sided infinite words $(\alpha_{k})$ converges to a word $\alpha$ if, for every interval, there exists a number $k_{0}$ such that all $\alpha_{k}$, for $ k \geq k_{0} $, agree with $\alpha$ on this interval.

We turn now to the construction of simple Toeplitz words. As mentioned in the introduction, they are obtained as the limit of a sequence of two-sided infinite periodic words with ``holes'' that are successively filled. To make this precise, we introduce an additional letter $ ? \notin \alphab$, which represents the hole. Let $ \alpha_{1} \in (\alphab \cup \{ ? \})^{\mathbb{Z}} $ be periodic with period $n$ and only a single occurrence of $?$ per period. Then there exists an integer $r$ with $ 0 \leq r < n $ such that $ n \mathbb{Z} + r $ are the positions of the holes in $\alpha_{1}$ (the so called undetermined part). Following \cite{LiuQu_Simple}, we define the filling of $\alpha_{2} \in (\alphab \cup \{ ? \})^{\mathbb{Z}} $ into the holes of $\alpha_{1}$ by
\[ (\alpha_{1} \triangleleft \alpha_{2})(j) := \begin{cases}
\alpha_{1}(j) & \text{for } j \notin n \mathbb{Z} + r\\
\alpha_{2}( \frac{j-r}{n} ) & \text{for } j \in n \mathbb{Z} + r
\end{cases} \, . \]
In other words, we fill $\alpha_{2}(0)$ into the first hole at a non-negative position, $\alpha_{2}(1)$ into the second hole at a non-negative position, and so on, while we fill $\alpha_{2}(-1)$ into first hole at a negative position, $\alpha_{2}(-2)$ into second hole at a negative position, and so on. If both $\alpha_{1}$ and $\alpha_{2}$ are periodic words with holes, then $\alpha_{1} \triangleleft \alpha_{2}$ is a periodic word with holes, too. Thus we can fill a word $ \alpha_{3} \in (\alphab \cup \{ ? \})^{\mathbb{Z}} $ into the holes of $\alpha_{1} \triangleleft \alpha_{2}$ and so on.

While a similar hole filling procedure occurs in the construction of all kinds of Toeplitz words, simple Toeplitz words are distinguished by the type of periodic words $\alpha_{k}$ that are used. They are described by a sequence $ (a_{k})_{k \in \mathbb{N}_{0}} $ of letters $a_{k} \in \alphab$, a sequence $ (n_{k})_{k \in \mathbb{N}_{0}} $ of period lengths $n_{k} \in\mathbb{N} $, $n_{k} \geq 2$ and a sequence $ (r_{k})_{k \in \mathbb{N}_{0}} $ of non-negative integer positions $0 \leq r_{k} < n_{k}$. From these sequences we define the two-sided infinite periodic words
\[ (a_{k}^{n_{k}-1}?)^{\infty} := \hdots a_{k} \hdots a_{k} ? \underbrace{a_{k} \hdots a_{k}}_{n_{k}-1 \text{-times}} ? a_{k} \hdots a_{k} ? \hdots \]
with period $n_{k}$ and undetermined part $ n_{k} \mathbb{Z} + r_{k} $. To keep track of which letters are still to be insert, we define $ \alphab_{k} := \{ a_{j} : j \geq k \} $ for  $k \in \mathbb{N}_{0}$. Following \cite{LiuQu_Simple}, we use $ \alphabEv := \cap_{k \geq 0} \alphab_{k} $ to denote the eventual alphabet, that is, the set of letters that appear infinitely often in the sequence $(a_{k})$. Since $\alphab$ is a finite set, there exists a number $ \eventNr $ such that $ a_{k} \in \alphabEv $ and $ \alphab_{k} = \alphabEv $ hold for all $ k \geq \eventNr $.

To construct a simple Toeplitz word, we insert $(a_{1}^{n_{1}-1}?)^{\infty}$ into the holes of $(a_{0}^{n_{0}-1}?)^{\infty}$, then insert $(a_{2}^{n_{2}-1}?)^{\infty}$ in the remaining holes of the obtained word, then insert $(a_{3}^{n_{3}-1}?)^{\infty}$ and so on. Thus we obtain a sequence  $( \omega_{k} )_{k \in \mathbb{N}_{0}}$ of two-sided infinite words defined by
\[ \omega_{k} := (a_{0}^{n_{0}-1}?)^{\infty} \triangleleft (a_{1}^{n_{1}-1}?)^{\infty} \triangleleft (a_{2}^{n_{2}-1}?)^{\infty} \triangleleft \hdots  \triangleleft (a_{k}^{n_{k}-1}?)^{\infty} \, . \]

\begin{prop}
\label{prop:OmegaN}
The word $ \omega_{k} $ is periodic with period $ n_{0} \cdot n_{1} \cdot \hdots \cdot n_{k} $ and has undetermined part $ U_{k} = n_{0} \cdot \hdots \cdot n_{k} \mathbb{Z} + \Big[ r_{0} + \sum_{j=1}^{k} r_{j}\cdot n_{0} \cdot \hdots \cdot n_{j-1} \Big] $. In particular, there is exactly one undetermined position per period.
\end{prop}

\begin{proof}
We proceed by induction: For $k = 0$ the claim is clearly true, since $ \omega_{0} = (a_{0}^{n_{0}-1}?)^{\infty}$ has by definition period $n_{0}$ and undetermined part $ n_{0} \mathbb{Z} + r_{0} $. Now assume that the claim about the period and the undetermined part holds for $ \omega_{k} $. By definition, we have $ \omega_{k+1} = \omega_{k} \triangleleft  (a_{k+1}^{n_{k+1}-1}?)^{\infty} $. Since $  (a_{k+1}^{n_{k+1}-1}?)^{\infty} $ has undetermined part $ n_{k+1} \mathbb{Z} + r_{k+1} $, the undetermined part of $\omega_{k+1}$ is
\begin{align*}
& n_{0} \cdot \hdots \cdot n_{k} \cdot \Big( n_{k+1} \mathbb{Z} + r_{k+1}  \Big) + \Big[ r_{0} + \sum_{j=1}^{k} r_{j} \cdot n_{0} \cdot \hdots \cdot n_{j-1} \Big] \\* 
=& n_{0} \cdot \hdots \cdot n_{k+1} \mathbb{Z} + \Big[ r_{0} + \sum_{j=1}^{k+1} r_{j} n_{0} \cdot \hdots \cdot n_{j-1} \Big] \, .
\end{align*}
In particular, the period of $\omega_{k+1}$ has to be at least $ n_{0} \cdot \hdots \cdot n_{k+1} $, since this is the distance between two undetermined positions. On the other hand, we obtained $\omega_{k+1}$ by inserting the word $ (a_{k+1}^{n_{k+1}-1}?)^{\infty} $ with period $n_{k+1}$ into the word $\omega_{k}$ with period $ \prod_{j = 0}^{k} n_{j} $ and exactly one hole per period. Hence the resulting word $\omega_{k+1}$ has a period of at most $ n_{k+1} \cdot \prod_{j = 0}^{k} n_{j} $. Since the undetermined part has the same period as the word, there is exactly one undetermined position per period.
\end{proof}

For later use, we define $\block{k}$ to be the block of letters between two consecutive holes in $ \omega_{k} $. By the above proposition, $\block{k}$ is well defined and its length is given by $ \length{\block{k}} + 1 = n_{0} \cdot \hdots \cdot n_{k} $ for all $ k \geq 0 $. In addition, it is convenient to define $\length{\block{-1}} = 0 $, such that the relation
\[ \length{\block{k}} + 1 = n_{k} \cdot ( \length{\block{k-1}} + 1 ) \]
holds for all $k \geq 0$. Moreover it is easy to see from the definition of $\omega_{k}$ that the blocks $\block{k}$ satisfy the recursion relation
\[ \block{0} = a_{0}^{n_{0}-1} \qquad \text{and} \qquad \block{k+1} = \underbrace{ \block{k} \; a_{k+1} \; \block{k} \; \hdots \; \block{k}\; a_{k+1} \; \block{k} }_{n_{k+1} \text{-times } \block{k} \text{ and }  (n_{k+1}-1) \text{-times } a_{k+1} } \, . \]

Now we proceed toward the definition of a simple Toeplitz word by taking the limit $ \omega_{\infty} := \lim_{k \to \infty} \omega_{k} $ in $ ( \alphab \cup \{ ? \} )^{\mathbb{Z}} $. The undetermined parts $ U_{k} $ of the words $\omega_{k}$ form a decreasing sequence of sets. The undetermined part $ U_{\infty} := \cap_{k \geq 0} U_{k} $ of $\omega_{\infty} $ is either empty or a single position. If $ U_{\infty} $ is empty and the limit word $ \omega_{\infty} $ is not periodic, then $\omega_{\infty}$ called a normal Toeplitz word. If $ U_{\infty} $ is not empty, then we insert an arbitrary letter $ \widetilde{a} \in \alphabEv $ into the single undetermined position. If the resulting word $ \omega_{\infty}^{(\tilde{a})} $ is not periodic, then $ \omega_{\infty}^{(\tilde{a})} $ is called an extended Toeplitz word. Following \cite{LiuQu_Simple}, we call a word a simple Toeplitz word if it is either a normal Toeplitz word or an extended Toeplitz word. Note that other definitions are used in the literature as well; see the paragraph \textit{``\titleref{para:IntroSimpToepW}''} in the introduction. The sequences $ (a_{k})_{k \in \mathbb{N}_{0}} $, $ (n_{k})_{k \in \mathbb{N}_{0}} $ and $ (r_{k})_{k \in \mathbb{N}_{0}} $ are called a coding of the obtained simple Toeplitz word. In the following, we will always assume $ a_{k} \neq a_{k+1} $ for all $k$, since subsequent occurrences of the same letter can be expressed as a single occurrence where the period length $n_{k}$ is increased accordingly. Moreover, we will always assume $ \card{ \alphabEv } \geq 2 $, since Toeplitz words must not be periodic and $ \card{ \alphabEv } = 1 $ (that is, an eventually constant sequence of letters) implies a periodic word:

\begin{prop}
Let $ (a_{k}) $, $ (n_{k}) $ and $ (r_{k}) $ be as defined above. Let $\omega$ denote the word $\omega_{\infty}$ if $U_{\infty} = \emptyset$ and the word $ \omega_{\infty}^{(\tilde{a})} $ with arbitrary $\widetilde{a} \in \alphabEv $ if $U_{\infty} \neq \emptyset$. Then $\omega$ is periodic if and only if $\card{ \alphabEv } = 1$ holds.
\end{prop}

\begin{proof}
First assume that $ \alphabEv $ contains only a single element $\widetilde{a}$. Then $ a_{k} = \widetilde{a} $ holds for all $ k \geq \eventNr $ and thus, every hole in $\omega_{ \eventNr - 1 }$ is filled with the letter $\widetilde{a}$ to obtain $\omega$. In the case of $U_{\infty} \neq \emptyset$, the remaining undetermined position is filled with the letter $\widetilde{a}$ as well. Since $\omega_{ \eventNr - 1 }$ is a periodic word over $ \alphab \cup \{ ? \} $, replacing every letter $?$ by the letter $\widetilde{a}$ yields a periodic word $\omega$.

To prove the converse, we follow the proof of Proposition~6.2 in \cite{QRWX_PatCompDDim}. Assume that $\card{ \alphabEv } \geq 2$ holds and that there exists a period $p$ of $\omega$. Choose two letters $\widetilde{a} \neq \widetilde{b}$ that appear infinitely often in the coding sequence and let $k_{0} \in \mathbb{N}$ be large enough such that $ p < \length{ \block{k_{0}} } + 1 $ holds. Let $j_{\widetilde{a}}, j_{\widetilde{b}} \in U_{k_{0}}$ be two positions that are undetermined in $\omega_{k_{0}}$, such that $\omega(j_{\widetilde{a}}) = \widetilde{a}$ and $\omega(j_{\widetilde{b}}) = \widetilde{b}$ hold. Recall that $\omega_{k_{0}}$ and in particular its undetermined part $U_{k_{0}}$ are periodic with period $\length{ \block{k_{0}} } + 1$. Thus $j_{\widetilde{a}} - j_{\widetilde{b}}$ is a multiple of $ \length{ \block{k_{0}} } + 1$ and we obtain the following contradiction:
\begin{align*}
\widetilde{a} &= \omega( j_{\widetilde{a}} ) & & \\
&= \omega( j_{\widetilde{a}} + p ) &  & \text{since } p \text{ is a period of } \omega \\
&= \omega_{k_{0}}( j_{\widetilde{a}} + p ) &  & \text{since } j_{\widetilde{a}} \in U_{k_{0}} \text{ and } p < \length{ \block{k_{0}} } + 1 \text{ imply } j_{\widetilde{a}} +p \notin U_{k_{0}} \\
&= \omega_{k_{0}}( j_{\widetilde{b}} + p ) &  & \text{since } ( j_{\widetilde{a}} + p) - ( j_{\widetilde{b}} + p ) \text{ is a multiple of } \length{ \block{k_{0}} } + 1 \\
&= \omega( j_{\widetilde{b}} + p ) & & \text{since } j_{\widetilde{b}} \in U_{k_{0}} \text{ and } p < \length{ \block{k_{0}} } + 1 \text{ imply } j_{\widetilde{b}} +p \notin U_{k_{0}} \\
&= \omega( j_{\widetilde{b}} ) & & \text{since } p \text{ is a period of } \omega \\
&= \widetilde{b}\, . & & \qedhere
\end{align*}
\end{proof}

\subsection{Subshifts of Simple Toeplitz Words}
\label{subsec:SimToepSub}

For the remainder of this section, our focus changes from elements $\omega \in \alphab^{\mathbb{Z}} $ to subsets $\subshift \subseteq \alphab^{\mathbb{Z}}$. We will define the shift map as well as subshifts and state some properties of (elements of) subshifts associated to simple Toeplitz words.

In the previous subsection, we defined a topology on $\alphab^{\mathbb{Z}}$. With respect to this topology, the (left-)shift, defined by
\[ \Shift : \alphab^{\mathbb{Z}} \to \alphab^{\mathbb{Z}} \quad \text{with} \quad (\Shift \omega) (j) := \omega( j+1 ) \, , \]
is a homeomorphism. A closed subset of $\alphab^{\mathbb{Z}}$ which is invariant under the shift $\Shift$, is called a subshift. A subshift $\subshift \subseteq \alphab^{\mathbb{Z}}$ is called minimal if the $\Shift$-orbit of every element $ \omega \in \subshift $ is dense in $\subshift$. It is called uniquely ergodic if there is a unique $\Shift$-invariant Borel probability measure on $\subshift$. For a (finite of infinite) word $\omega$, we use $\langu{\omega}$ to denote the set of all finite subwords that occur in $\omega$. The empty word of length zero is considered to be a subword of every $\omega$. For a subshift $\subshift$, we define its language as $ \langu{ \subshift } := \cup_{\omega \in \subshift} \langu{ \omega } $.

To a two-sided infinite word $ \omega $ we associate a subshift $ \subshift_{\omega} := \overline{ \{ \Shift^{k} \omega : k \in \mathbb{N} \} } $ by taking the closure of the orbit of $\omega$ under the shift. For the remainder of this paper, we will always assume that $\subshift$ is the subshift associated to a simple Toeplitz word $\omega$, which is defined by the sequences $ (a_{k})_{k \in \mathbb{N}_{0}} $, $ (n_{k})_{k \in \mathbb{N}_{0}} $ and $ (r_{k})_{k \in \mathbb{N}_{0}} $. Such a word $\omega$ is the limit of a sequence of periodic words $\omega_{k}$, where $\omega_{k+1}$ is a completion of $\omega_{k}$ (that is, some holes are filled and non-hole positions remain unchanged). If $\omega$ is a normal Toeplitz word, then for every position $j \in \mathbb{Z}$ there exists an index $k$ such that $ \omega_{k}(j) \in \alphab $ holds. In addition we have seen in Proposition~\ref{prop:OmegaN} that there is precisely one hole per period in $\omega_{k}$ and thus
\[ \frac{\text{number of undetermined positions in } \omega_{k} \text{ per period}}{\text{length of the period of } \omega_{k}} \xrightarrow{k \to \infty} 0 \]
holds. Toeplitz words with these properties are called regular. Their associated subshift is always minimal and uniquely ergodic (see the corollary to Theorem~5 in \cite{JacobsKeane_01Toeplitz} for Toeplitz words over two letters and see e.g. \cite{Downa_OdomToepl} for Toeplitz words over arbitrary finite alphabets). In particular, the subshifts associated to normal Toeplitz words are minimal and uniquely ergodic.

When we construct a simple Toeplitz word $\omega$ with coding sequences $(a_{k})$, $(n_{k})$ and $(r_{k})$, the sequence $(r_{k})$ describes into which hole to map the origin. Therefore a change of $r_{k}$ causes a shift of the resulting word. Since we construct the subshift by taking the orbit closure of $\omega$, such a shift of the word does not change $\subshift$. This was made precise in \cite{LiuQu_Simple}, Proposition~2.3:

\begin{prop}[\cite{LiuQu_Simple}]
\label{prop:DifferentR}
If two simple Toeplitz words $\omega, \widetilde{\omega}$ have the coding sequences $(a_{k})$, $(n_{k})$, $(r_{k})$ and $(a_{k})$, $(n_{k})$, $(\widetilde{r}_{k})$ respectively, then the associated subshifts $\subshift_{\omega}, \subshift_{\widetilde{\omega}}$ are equal. 
\end{prop}

Therefore, we will from now on omit $(r_{k})$ and speak about the subshift generated by the sequences $(a_{k}) $ and $ (n_{k}) $. In addition, we can now extend minimality and unique ergodicity to the subshifts that are associated to arbitrary simple Toeplitz words. The underlying idea is that, if necessary, we can change $(r_{k})$ into a sequence $(\widetilde{r}_{k})$ such that $\widetilde{\omega}$ is a normal simple Toeplitz word, cf. \cite{LiuQu_Simple}, Corollary~2.1.

\begin{prop}[\cite{LiuQu_Simple}]
For every simple Toeplitz word $\omega$, the subshift $\subshift_{\omega}$ is minimal and uniquely ergodic.
\end{prop}
 
In Proposition~2.4 in \cite{LiuQu_Simple} it as was shown that, in a certain sense, the converse of Proposition~\ref{prop:DifferentR} holds as well:

\begin{prop}[\cite{LiuQu_Simple}]
If the subshift $\subshift$ is associated to a simple Toeplitz word $\omega$ with coding sequences $(a_{k})$, $(n_{k}) $ and $ (r_{k})$, then every element $ \widetilde{\omega} \in \subshift_{\omega} $ is a simple Toeplitz word with coding sequences $ (a_{k}) , (n_{k}) $ and $ (\widetilde{r}_{k})$.
\end{prop}

For every $ \omega \in \subshift $ we can therefore find a sequence of periodic words $\omega_{k} $, as described in Proposition~\ref{prop:OmegaN}, that converge to $\omega$ (with the possible exception of one undetermined position). Thus, for every $k \in \mathbb{N}_{0}$ and every $ \omega \in \subshift $, we can write $\omega$ as 
\[ \omega = \hdots \block{k} \star \block{k} \star \block{k} \star \block{k} \hdots \quad , \]
where $ \star $ denotes elements from $ \alphab_{k+1} = \{ a_{j} : j \geq k+1 \} $. Since all $ \omega \in \subshift $ have the same letters in the sequence $(a_{k})$ and the same period length $(n_{k})$, the $\block{k}$-blocks are the same for all $ \omega \in \subshift $. Based on these blocks, we can give an alternative definition of the subshift associated to a simple Toeplitz word $\omega$: Let $(a_{k})$ and $(n_{k})$ be the coding sequences of $\omega$ and define the $\block{k}$-blocks as in the previous subsection by $ \block{0} = a_{0}^{n_{0}-1}$ and $ \block{k+1} = \block{k} \; a_{k+1} \; \block{k} \; \hdots \; \block{k} $ with $(n_{k+1} - 1)$-times $a_{k+1}$  and $n_{k+1}$-times $\block{k}$. Since $ \block{k} $ is a prefix of $\block{k+1}$ for all $ k \geq 0$, there exits a unique one-sided infinite word $\varpi$ such that $ \block{k} $ is a prefix of $\varpi$ for all $ k \geq 0$. We define a subshift by
\[ \widehat{\subshift}_{\omega} := \{ \varrho \in \alphab^{\mathbb{Z}} : \langu{ \varrho } \subseteq \langu{ \varpi }\} \, .\]

\begin{prop}
For every simple Toeplitz word $\omega$ the equality $ \widehat{\subshift}_{\omega} = \subshift_{\omega} $ holds.
\end{prop}

\begin{proof}
First we show $\widehat{\subshift}_{\omega} \subseteq \subshift_{\omega} $. Let $ \varrho \in \widehat{\subshift}_{\omega} $ and $ J \in \mathbb{N} $. The finite subword $ \varrho |_{[-J,J]} $ of $\varrho$ is a subword of $\varpi$ as well. Thus it is contained in the prefix $\block{k}$ of $\varpi$ for every sufficiently large $k$. Because of the decomposition $ \omega = \hdots \star \block{k} \star \block{k} \star \hdots $, there exists a $k_{J}$ such that $ \Shift^{k_{J}} \omega |_{[-J,J]} = \varrho |_{[-J,J]} $. We obtain a sequence $(k_{J})$ with $ \varrho = \lim_{J \to \infty} \Shift^{k_{J}} \omega |_{[-J,J]} $ and thus $ \varrho \in \subshift_{\omega}$.

For the converse, first assume that $\omega$ is a normal simple Toeplitz word. Thus, for every $J \in \mathbb{N}$ there exists a number $k_{J}$ such that all positions in $ \omega |_{[-J, J]} $ are determined in the word $\omega_{k_{J}} = \hdots \block{k_{J}} ? \block{k_{J}} ? \block{k_{J}} \hdots \; $ . Hence $ \omega |_{[-J, J]} $ is contained in $\block{k_{J}}$. Now assume that $\omega$ is an extended simple Toeplitz word. Then, for every $J$ there is a number $k$ such that $ \omega |_{[-J, J]} $ is contained in $ \block{k} ? \block{k} $, where $?$ denotes the position of $U_{\infty}$. Let $\widetilde{a} \in \alphabEv$ denote the letter that is filled into $U_{\infty}$. Since $\widetilde{a}$ appears infinitely often in the coding sequence, there exists a number $k_{J} > k $ such that $ a_{k_{J}} = \widetilde{a} $. Now $ \omega |_{[-J, J]} $ is contained in $ \block{ k } \widetilde{a} \block{ k } $, which is contained in  $ \block{ k_{J}-1 } \widetilde{a} \block{ k_{J}-1 } $, which is contained in $ \block{k_{J}} $. Thus, independent of which kind of simple Toeplitz word $\omega$ is, there is a number $k_{J}$ such that $ \omega |_{[-J, J]} $ is contained in a $\block{k_{J}}$-block. Hence every subword of $ \omega $ is a subword of $\varpi$ and we obtain $ \omega \in \widehat{\subshift}_{\omega}$.
\end{proof}

We conclude the section by defining a class simple Toeplitz subshifts that will serve as our main example throughout the whole paper. As mentioned in the introduction, one motivation for the study of simple Toeplitz subshifts is their connection to self similar groups: For the family of Grigorchuk's groups, the orbital Schreier graphs have a similar structure as Toeplitz words (see \cite{GLN_SpectraSchreierAndSO}). Essentially, in each step every second ``hole'' is filled. In the first step the filling always corresponds to the generator $a$ and in the $k$-th step ($k \geq 2$) it corresponds to the value of $ \omega_{k} \in \{ \pi_{b}, \pi_{c}, \pi_{d} \} $. Thus we consider the following simple Toeplitz subshift:

\begin{exmpl}
Following convention, we consider the four letter alphabet $ \alphab = \{ a, x, y, z \} $. We are interested in the subshifts defined by the constant sequence $ (n_{k})_{k \in \mathbb{N}_{0}} = (2, 2, 2, \hdots )$ and a sequence $(b_{k}) \in \alphab^{\mathbb{N}_{0}} $ with $ b_{0} = a $ and $ b_{k} \in \{ x, y, z \} $ for $ k \geq 1 $. Here $b_{k} = b_{k+1}$ is allowed, but otherwise the same construction is used as in the simple Toeplitz case. In particular, we assume that $ (b_{k})$ is not eventually constant. Via $ x \leftrightarrow \pi_{d} \, ,\; y \leftrightarrow \pi_{c} \, ,\; z \leftrightarrow \pi_{b} $ the sequence $b_{k}$ corresponds to a sequence $\omega_{k}$, which defines an element in the family of Grigorchuk's groups. Within the scope of this text, we will therefore call these subshifts \emph{generalized Grigorchuk subshifts}, in contrast to the (standard) Grigorchuk subshift that is defined in the example below. Note however that these notions are not standard terminology. To describe the subshifts in accordance with our definition of simple Toeplitz subshifts, we have to express subsequent occurrences of the same letter $(b_{k})$ as a single letter $a_{m}$ with associated period length $n_{m} \geq 2$. Assume $ b_{k} \neq b := b_{k+1} = \hdots = b_{k+j} \neq b_{k+j+1} $ and note that
\[ \underbrace{(b ?)^{\infty} \triangleleft \hdots \triangleleft (b ?)^{\infty}}_{j \text{ times}} = ( \underbrace{b \hdots b}_{2^{j}-1 \text{ times}} ?)^{\infty} =  ( b^{2^{j}-1} ?)^{\infty} \]
holds. Hence the generalized Grigorchuk subshifts are precisely those simple Toeplitz subshifts with $ \alphab = \{ a, x, y, z \} $, $ a_{0} = a $ and $ a_{k} \in \{ x, y, z \} $ for $k \geq 1$, where $n_{0} = 2$ and for every $k \in \mathbb{N}$ there exists a number $j_{k} \geq 1$ such that $ n_{k} = 2^{j_{k}} $ holds. Their block length is given by 
\[ \length{\block{k}} + 1 = 2^{j_{k}} \cdot ( \length{\block{k-1}} + 1 ) = \hdots = 2^{ 1 + j_{1} + \hdots + j_{k} } \, .\]
\end{exmpl}

As a special case we will sometimes consider the subshift that corresponds to Grigorchuk's group. This group is obtained when $(\omega_{k})_{k}$ is the periodic sequence  $ \omega = (\pi_{d}, \pi_{c}, \pi_{b} , \hdots ) $:

\begin{exmpl}
As above, consider $ \alphab = \{ a, x, y, z \} $. Let $ (a_{k})_{k \in \mathbb{N}_{0}} = ( a, x, y, z, x, y, z, \hdots ) $ be 3-periodic from $a_{1}$ on and let $ (n_{k})_{k \in \mathbb{N}_{0}} = (2, 2, 2, \hdots )$ be the constant sequence with value two. The associated subshift $\widehat{\subshift}$ is precisely the subshift which is linked to Grigorchuk's group (cf. \cite{GLN_SpectraSchreierAndSO}) and we will refer to it as \emph{Grigorchuk subshift}. The length of the $\block{k}$-blocks is given by $\length{\block{k}} + 1 = 2^{k+1}$. Moreover $ \alphabEv = \{ x, y, z \} $ and $ \eventNr = 1 $ hold. 
\end{exmpl}

Note that our notion of a generalized Grigorchuk subshift includes as a special instance what is called an \emph{$l$-Grigorchuk subshift} in \cite{DKMSS_Regul-Article}. These are the subshifts that are obtained from $ \alphab = \{ a,x,y,z \} $, $ (a_{k}) = (a, x, y, z, x, y, z, \hdots) $ and $ (n_{k}) = (2, 2^{l_{1}}, 2^{l_{2}}, 2^{l_{3}}, \hdots) $ with $ l_{1}, l_{2}, l_{3}, \hdots \in \mathbb{N} $.


\section{Subword Complexity}
\label{sec:SubwordComp}

The aim of this section is to give an explicit formula for the complexity function of a simple Toeplitz subshift. For a different subclass of Toeplitz words this problem has for example been studied in \cite{CassKar_ToeplWords}. There, one-sided infinite words are studied that are obtained by a hole filling procedure from a single word with holes. This word is repeatedly inserted into itself and it is shown that the complexity of the obtained words grows like a polynomial. In addition it is shown in \cite{Kosk_ComplSuites} with a more general construction procedure, that for every rational number $r \in \mathbb{Q} $ a Toeplitz word can be obtained such that $\comp( L )$ grows like $L^{r}$. While both works cover cases that are much more general, they also provide bounds for the special case of a single hole per period: Theorem 5 in \cite{CassKar_ToeplWords} implies that the complexity is dominated by a linear function if one word with holes is repeatedly inserted into itself. Theorem 9 in \cite{Kosk_ComplSuites} implies that this is also the case for a sequence of periodic words with holes instead of a single word, provided that the length of the periods is bounded. In the following, we provide an explicit formula for the complexity of simple Toeplitz words. Our result imply that the complexity is in this case still dominated by a linear function (cf. Proposition~\ref{prop:MinMaxQuotient}).

Our main strategy is similar to the one that was employed in \cite{GLN_Combinat} for the special case of the Grigorchuk subshift (see arxiv version of \cite{GLN_Survey} as well): First we prove an upper bound for the complexity at certain points. Then we prove a lower bound for the growth rate of the complexity function. Together, these inequalities determine the complexity. The same technique was also employed in \cite{DKMSS_Regul-Article} to obtain the complexity of $l$-Grigorchuk subshifts.

\subsection{Inequalities for the Complexity and its Growth}

In the following we will denote the cardinality of a set $A$ by $ \card{ A } $, its characteristic function by $\mathbbm{1}_{A}$ and its complement by $\cmplmt{A}$. Recall that $\langu{\subshift}$ denotes the language of the subshift $\subshift$, that is, the set of all finite words which occur in elements of $\subshift$. The complexity function counts how many words of a given finite length there are and thus measures (one aspect of) how ordered the subshift is. It is defined as
\[ \comp : \mathbb{N}_{0} \rightarrow \mathbb{N} \; , \quad L \mapsto \card{ \{ u \in \langu{ \subshift } : \length{u} = L \} } \, .\]
Since the empty word is always considered to be an element of the language, we have $\comp(0) = 1$. Moreover, we define the growth rate of the complexity as $ \growth( n ) := \comp( n+1 ) - \comp( n ) $ for $ n \in \mathbb{N}_{0}$. First we establish an upper bound for the complexity at the lengths $\length{\block{k}}+1$ with $k \geq 0$.

\begin{prop}
\label{prop:enthalten}
For every word $ u \in \langu{\subshift} $ of length $ |u| \leq \length{\block{k}}+1 $ there is a letter $ a \in \alphab_{k+1} $ such that $u$ is a subword of $ \block{k} a \block{k} $.
\end{prop}

\begin{proof}
This follows immediately from the decomposition $ \omega = \hdots \block{k} \star \block{k} \star \block{k} \star \block{k} \hdots $ with letters $ \star \in \alphab_{k+1} $, which exists for all $ \omega \in \subshift $ and all $ k \geq 0 $.
\end{proof}

\begin{prop}
\label{prop:obereSchrComp}
For all $ k \geq 0 $ the following inequality holds:
\[ \comp( \length{\block{k}} +1 ) \leq ( \card{ \alphab_{k} } - 1 ) \cdot (\length{\block{k}}+1) + \mathbbm{1}_{\alphab_{k+1}}(a_{k}) \cdot ( \length{\block{k-1}}+1 ) \, .\]
\end{prop}

\begin{proof}
First consider the case $ k \geq 1 $. By Proposition~\ref{prop:enthalten} it is sufficient to study subwords of length $ \length{\block{k}}+1 $ of $ \block{k} \, a \, \block{k} $ for $ a \in \alphab_{k+1} $. For a letter $ a \neq a_{k} $ there are at most $ \length{\block{k}} + 1 $ many subwords. If $ a = a_{k} $ is possible at all (that is, if $a_{k} \in \alphab_{k+1} $ holds), then this case will result in at most $ \length{\block{k-1}} + 1 $ subwords. The reason is that the subwords in
\[ \block{k} a_{k} \block{k} = \Big[ \block{k-1} \; a_{k} \; \block{k-1} \; \hdots \; \block{k-1} \Big] a_{k} \Big[ \block{k-1} \; a_{k} \; \block{k-1} \; \hdots \; \block{k-1} \Big] \]
repeat already when reaching the second $\block{k-1}$. This yields
\[ \comp( \length{\block{k}}+1 ) \leq \card{ ( \alphab_{k+1} \setminus \{ a_{k} \} ) } \cdot ( \length{\block{k}}+1) + \mathbbm{1}_{ \alphab_{k+1}}(a_{k}) \cdot (\length{\block{k-1}} + 1) \]
for $k \geq 1$. The case $ k = 0 $ is similar: For every $ a \in \alphab_{1} \setminus \{ a_{0} \} $, there are at most $ \length{\block{0}} + 1 $ subwords in $ \block{0} a \block{0} $ and because of $\block{0} = a_{0}^{n_{0}-1}$, there is only one factor of length $ \length{\block{0}} + 1 $ in $ \block{0} a_{0} \block{0} $, namely $a_{0}^{n_{0}} $. Our definition of $\length{\block{-1}} = 0 $ allows us to write
\[ \comp( \length{\block{0}}+1 ) \leq \card{ ( \alphab_{1} \setminus \{ a_{0} \} ) } \cdot (\length{\block{0}} + 1) + \mathbbm{1}_{\alphab_{1}}(a_{0}) \cdot (\length{\block{-1}} + 1) \, . \]
Finally, we note that $ \card{ ( \alphab_{k+1} \setminus \{ a_{k} \} ) } = \card{ \alphab_{k} } - 1 $ holds for all $ k \geq 0 $: If $ a_{k} \in \alphab_{k+1} $ holds, then this implies $ \alphab_{k+1} = \alphab_{k} $ and $ \card{ ( \alphab_{k+1} \setminus \{ a_{k} \} ) } = \card{ \alphab_{k+1} } - 1 =  \card{ \alphab_{k} } - 1$. If, on the other hand, $ a_{k} \notin \alphab_{k+1} $ holds, then this implies $ \card{ \alphab_{k+1} } = \card{ \alphab_{k} } - 1 $ and $ \card{ ( \alphab_{k+1} \setminus \{ a_{k} \} ) } = \card{ \alphab_{k+1} } =  \card{ \alphab_{k} } - 1$.
\end{proof}

\begin{rem}
In Proposition~\ref{prop:untereSchrComp} we will show that the converse inequality holds as well, thus proving equality of the terms.
\end{rem}

\begin{prop}
\label{prop:geq2}
For all $k \geq 0$ and $ 0 \leq L \leq \length{\block{k}} - \length{\block{k-1}} - 1$ the inequality $ \growth(L) := \comp( L+1 ) - \comp(L) \geq \card{ \alphab_{k} } - 1$ holds.
\end{prop}

\begin{proof}
Consider the suffix $v_{1}$ of $\block{k}$ that consists of the last $L$ letters. We will show that $v_{1}$ is a right special word, that is, a word with more than one extension to the right. More precisely, we will show that there are $\card{\alphab_{k} }$ different extensions, which implies that the complexity increases by at least $\card{ \alphab_{k} } - 1$ when we increase the word length by one.

First note that $v_{1}$ can be extended by all letters in $\alphab_{k+1}$. If $a_{k} \in \alphab_{k+1} $ holds, then the equality $ \card{ \alphab_{k+1} } = \card{ \alphab_{k} } $ follows and we are done. For $a_{k} \notin \alphab_{k+1} $, we will show that $v_{1}$ can nevertheless be extended to the right by $a_{k}$. Thus there are $ \card{ \alphab_{k+1} } + 1 = \card{ \alphab_{k} } $ different extensions, which proves the statement.

For $k \geq 1$ it follows from the decomposition $\block{k} = \block{k-1} \hdots \block{k-1} \; a_{k} \; \block{k-1} $ and $ L \leq \length{\block{k}} - ( \length{\block{k-1}} + 1) $ that $ v_{1} $ is a suffix of the first $ (n_{k} - 1) \cdot \length{\block{k-1}}+ n_{k} - 2 $ letters in $\block{k}$ and therefore can be followed by $ a_{k} $. For $ k = 0 $ and $n_{0} > 2$, the suffix is given by $v_{1} = a_{0}^{L}$ with $L \leq n_{0} - 2 $ and this can be extended to the right by $a_{0}$, as can be seen in $\block{0} = a_{0}^{n_{0}-1}$. For $ k = 0 $ and $ n_{0} = 2 $, the proposition follows immediately from $ \comp(0) = 1 $ and $ \comp(1) = \card{ \alphab } = \card{ \alphab_{0} } $.
\end{proof}

\begin{prop}
\label{prop:geq3}
Let $ k \geq 1 $ and $ \length{\block{k-1}} + 1 \leq L \leq 2 \length{\block{k-1}} - \length{\block{k-2}} $. If $a_{k-1} \in \alphab_{k} $ holds, then $\growth(L)$ is at least by one greater than stated in Proposition~\ref{prop:geq2}.
\end{prop}

\begin{proof}
We will show that there is a word $v_{2}$ of length $L$ that is right special and different from the suffix $v_{1}$ of $\block{k}$ that was considered in Proposition~\ref{prop:geq2}. First assume $k \geq 2$ and note that the word $ \block{k-1} \, a_{k-1} \, \block{k-1} $ occurs in the subshift, since $ a_{k-1} \in \alphab_{k}$ holds by assumption. From the decomposition into $\block{k-1}$-blocks and single letters it is clear that $\block{k-1} \, a_{k-1} \, \block{k-1} $ has to be followed by $a_{k}$. Since $ k \geq 2 $, we can decompose $\block{k-1}$ in $\block{k-2}$-blocks and single letters $ a_{k-1} $, see Figure~\ref{fig:2RightSpecial}. Let $v_{2}$ be the suffix of length $L$ of $\block{k-1} \, a_{k-1} \, \block{k-1}$. The decomposition shows that $v_{2}$ can be extended to the right with both, $a_{k}$ and $a_{k-1}$. Moreover, it follows from $ L \geq \length{\block{k-1}} + 1 $ that the word $v_{2}$ ends with $ a_{k-1} \, \block{k-1} $. It is therefore different from the suffix $v_{1}$ of $ \block{k} = \block{k-1} \, a_{k} \hdots a_{k} \, \block{k-1}$ of length $L$, which ends with $ a_{k} \, \block{k-1} $. 

\begin{figure}
\centering
\footnotesize
\begin{tikzpicture}
[every node/.style ={rectangle, inner sep=0pt, minimum height=0.6cm, draw},
pn/.style={minimum width=7.22cm},
pn-1/.style={minimum width=2.07cm},
pn-2/.style={minimum width=0.35cm},
buchst/.style={minimum width=0.28cm}]
%
\node [right] at (0,2) [pn-1] (A) {$\block{k-1}$};

\node [buchst] (A) [right=0.1cm of A] {$c$};
\node[pn-2] (A) [right=0.1cm of A] {$$};
\node [buchst] (A) [right=0.1cm of A] {$b$};
\node[pn-2] (A) [right=0.1cm of A] {$$};
\node [buchst] (A) [right=0.1cm of A] {$b$};
\node[pn-2] (A) [right=0.1cm of A] {$$};
\node [buchst] (A) [right=0.1cm of A] {$b$};
\node [pn-1] (A) [right=0.1cm of A] {$\block{k-1}$};
\node [buchst] (Fix1) [right=0.1cm of A] {$c$};
\node [pn-1] (A) [right=0.1cm of Fix1] {$\block{k-1}$};

\node [right] at (0,1) [pn-1] (A) {$\block{k-1}$};
\node [buchst] (A) [right=0.1cm of A] {$c$};
\node [pn-1] (A) [right=0.1cm of A] {$\block{k-1}$};

\node [buchst] (A) [right=0.1cm of A] {$b$};
\node [pn-1] (A) [right=0.1cm of A] {$\block{k-1}$};
\node [buchst] (A) [right=0.1cm of A] {$c$};
\node [pn-1] (A) [right=0.1cm of A] {$\block{k-1}$};

\node [right] at (0,0) [pn-1] (A) {$\block{k-1}$};
\node [buchst] (A) [right=0.1cm of A] {$c$};
\node[pn-2] (A) [right=0.1cm of A] {$$};
\node [buchst] (A) [right=0.1cm of A] {$b$};
\node[pn-2] (A) [right=0.1cm of A] {$$};
\node [buchst] (A) [right=0.1cm of A] {$b$};
\node[pn-2] (A) [right=0.1cm of A] {$$};
\node [buchst] (A) [right=0.1cm of A] {$b$};
\node[pn-2] (A) [right=0.1cm of A] {$$};
\node [buchst] (A) [right=0.1cm of A] {$b$};
\node[pn-2] (A) [right=0.1cm of A] {$$};
\node [buchst] (A) [right=0.1cm of A] {$b$};
\node[pn-2] (A) [right=0.1cm of A] {$$};

\node [buchst] (A) [right=0.1cm of A] {$c$};
\node [pn-1] (A) [right=0.1cm of A] {$\block{k-1}$};
\node [right] at (0,-1) [pn-1] (A) {$\block{k-1}$};
\node [buchst] (A) [right=0.1cm of A] {$c$};
\node[pn-2] (A) [right=0.1cm of A] {$$};
\node [buchst] (A) [right=0.1cm of A] {$b$};
\node[pn-2] (A) [right=0.1cm of A] {$$};
\node [buchst] (A) [right=0.1cm of A] {$b$};
\node[pn-1] (A)  [right=0.1cm of A] {$\block{k-1}$};
\node [buchst] (Fix2) [right=0.1cm of A] {$b$};
\node[pn-2] (A) [right=0.1cm of Fix2] {$$};

\node [buchst] (A) [right=0.1cm of A] {$c$};
\node [pn-1] (A) [right=0.1cm of A] {$\block{k-1}$};
\node [buchst] (A) [above=0.3cm of Fix1] {$c$};
\node [left=0.1cm of A, minimum width=3.9cm] {$v_{2}$};
\node [buchst] (A) [below=0.3cm of Fix2] {$b$};
\node [left=0.1cm of A, minimum width=3.9cm] {$v_{2}$};
\end{tikzpicture}
\normalsize
\caption{Two different extensions of $v_{2}$ in the decomposition of $\block{k-1} b \block{k-1}$. For the sake of readability we use the abbreviations $b := a_{k-1}$ and $ c := a_{k} $.}
\label{fig:2RightSpecial}
\end{figure}

For $k=1$ we define $v_{2} = a_{0}^{L}$. Because of $a_{0} \in \alphab_{1}$, this word occurs in $ \block{0} a_{0} \block{0} = a_{0}^{n_{0}-1}  a_{0} a_{0}^{n_{0}-1} $. From the decomposition into $\block{0}$-blocks and single letters it is clear that $\block{0} \, a_{0} \, \block{0} $ has to be followed by $a_{1}$. For every length $ n_{0} \leq L \leq 2 \cdot (n_{0} - 1) $, $v_{2}$ can therefore be extended to the right by both, $a_{0}$ and $a_{1}$. Moreover it ends with $ a_{0}^{n_{0}} $, while the suffix $v_{1}$ of length $L$ of $\block{1}$ ends with $ a_{1} a_{0}^{n_{0}-1} $.
\end{proof}

\begin{rem}
\label{rem:wN-1An+1An}
Note that $ a_{k-1} \in \alphab_{k+1} $ holds if and only if $ a_{k-1} \in \alphab_{k} $ holds: Clearly $ a_{k-1} \in \alphab_{k+1} $ implies $ a_{k-1} \in \alphab_{k} $, as $ \alphab_{k+1} \subseteq \alphab_{k} $ holds for all $k$. Conversely,  $ a_{k-1} \in \alphab_{k} =  \{ a_{j} : j \geq k \} $ implies $ a_{k-1} \in \alphab_{k+1} = \{ a_{j} : j \geq k+1 \}  $ since we assumed $ a_{k-1} \neq a_{k} $.
\end{rem}

\begin{prop}
\label{prop:untereSchrComp}
For all $k \geq 0 $ the following inequality holds:
\[ \comp( \length{\block{k}} +1 ) \geq ( \card{ \alphab_{k} } - 1 ) \cdot (\length{\block{k}}+1) + \mathbbm{1}_{\alphab_{k+1}} (a_{k}) \cdot ( \length{\block{k-1}}+1 ) \]
\end{prop}

\begin{proof}
We proceed by induction and express the complexity as a telescoping sum of the growth of the complexity. For $k = 0$, the growth is bound from below by Proposition~\ref{prop:geq2} and we obtain
\begin{align*}
\comp( \length{\block{0}}+1 ) &= \growth( \length{\block{0}} ) + \sum_{L=0}^{\length{\block{0}}-1} \growth(L) + \comp(0) \\
& \geq \card{ \alphab_{1} } - 1 + ( \card{ \alphab_{0} } - 1 ) \cdot \length{\block{0}} + 1 \\
&= (\card{ \alphab_{0} } - 1) \cdot ( \length{\block{0}} + 1 ) + \mathbbm{1}_{\alphab_{1}}(a_{0} ) \cdot( \length{\block{-1}} + 1 ) \, ,
\end{align*}
where we used $ 1 - \card{ \alphab_{0} } + \card{ \alphab_{1} } = \mathbbm{1}_{\alphab_{1}}(a_{0} ) $ and $ \length{\block{-1}} + 1 = 1 $ in the last line. Now assume that the claim is true for $k-1$. For $k \geq 1$, we employ the bounds from the Propositions~\ref{prop:geq2} and \ref{prop:geq3}. This yields
\begin{align*}
& \comp( \length{\block{k}}+1 ) \\
& = \sum_{ L = \length{\block{k-1}}+1 }^{ \length{\block{k}} - \length{\block{k-1}} -1 }{ \growth(L) } + \sum_{ L = \length{\block{k}} - \length{\block{k-1}} }^{ \length{\block{k}} }{ \growth(L) } + \comp(\length{\block{k-1}}+1) \\
& \geq (\card{ \alphab_{k} } - 1) \cdot ( \length{\block{k}} - 2 \length{\block{k-1}} -1 ) + (\card{ \alphab_{k+1} } - 1) \cdot ( \length{\block{k-1}} + 1 ) \\
& \qquad + \mathbbm{1}_{\alphab_{k+1}}(a_{k-1}) \cdot (\length{\block{k-1}} - \length{\block{k-2}}) + \comp(\length{\block{k-1}}+1) \\
& = (\card{ \alphab_{k} } - 1) \cdot ( \length{\block{k}} +1 - 2 \length{\block{k-1}} -2 ) + (\card{ \alphab_{k} } - 1 + \mathbbm{1}_{\alphab_{k+1}}(a_{k}) - 1) \cdot \\
& \qquad ( \length{\block{k-1}} + 1 ) + \mathbbm{1}_{\alphab_{k+1}}(a_{k-1}) \cdot (\length{\block{k-1}} - \length{\block{k-2}}) + \comp(\length{\block{k-1}}+1) \\
& = (\card{ \alphab_{k} } - 1) \cdot ( \length{\block{k}} + 1) +\mathbbm{1}_{\alphab_{k+1}}(a_{k}) \cdot ( \length{\block{k-1}} + 1 ) \\
& \qquad - ( \card{ \alphab_{k} } - \mathbbm{1}_{\alphab_{k}}(a_{k-1}) ) \cdot( \length{\block{k-1}} +1 ) - \mathbbm{1}_{\alphab_{k}}(a_{k-1}) \cdot (\length{\block{k-2}} + 1) + \comp(\length{\block{k-1}}+1) \\
& = (\card{ \alphab_{k} } - 1) \cdot ( \length{\block{k}} +1) +\mathbbm{1}_{\alphab_{k+1}}(a_{k}) \cdot ( \length{\block{k-1}} + 1 ) \\
& \qquad - \Big[ ( \card{ \alphab_{k-1} } - 1 ) \cdot( \length{\block{k-1}} +1 ) + \mathbbm{1}_{\alphab_{k}}(a_{k-1}) \cdot (\length{\block{k-2}} + 1) \Big] + \comp(\length{\block{k-1}}+1) \\
& \geq (\card{ \alphab_{k} } - 1) \cdot ( \length{\block{k}} +1) +\mathbbm{1}_{\alphab_{k+1}}(a_{k}) \cdot ( \length{\block{k-1}} + 1 ) \, ,
\end{align*}
where we used $ \card{ \alphab_{k} } + 1 - \mathbbm{1}_{\alphab_{k}}(a_{k-1}) =  \card{ \alphab_{k-1} } $ in the next to last line and the induction hypothesis in the last line.
\end{proof}

The Propositions~\ref{prop:obereSchrComp} and \ref{prop:untereSchrComp} yield the exact value of the complexity function at all points of the form $\length{\block{k}} + 1$. In particular, the subwords of the words $ p(k) \, a \, p(k) $ with $a \in \alphab_{k+1} $, that were counted in the proof of Proposition~\ref{prop:obereSchrComp}, are pairwise different. Moreover, the complexity grows exactly by the amount given as lower bounds in Proposition~\ref{prop:geq2} and Proposition~\ref{prop:geq3}. If the growth was faster, then the value at $\length{\block{k}}+1$ could not be obtained.

\begin{coro}
\label{coro:CPnAndGrowth}
For all $ k \geq 0 $, the complexity of $\length{\block{k}} + 1$ is given by
\[ \comp( \length{\block{k}} +1 ) = ( \card{ \alphab_{k} } - 1 ) \cdot (\length{\block{k}}+1) + \mathbbm{1}_{ \alphab_{k+1} } (a_{k}) \cdot ( \length{\block{k-1}}+1 ) \, .\]
The growth rate of the complexity function is given by
\begin{align*}
\growth(L) = \; & \card{ \alphab_{0} } - 1 \qquad \text{for } 0 \leq L \leq \length{\block{0}} - 1 \, , \\
\text{by} \qquad \growth(L) = \; & \card{ \alphab_{1} } - 1 \qquad \text{for } L = \length{\block{0}} \\
\text{and by} \qquad \growth(L) = \; & \card{ \alphab_{k} } - 1\\
& - \begin{cases}
0 &  \text{if } \length{\block{k-1}}+1 \leq L \leq \length{\block{k}} - \length{\block{k-1}}-1 \\
\card{ \alphab_{k} } - \card{ \alphab_{k+1} } & \text{if } \length{\block{k}} - \length{\block{k-1}} \leq L \leq \length{\block{k}} 
\end{cases} \\ & + \begin{cases}
\mathbbm{1}_{\alphab_{k}}(a_{k-1}) &  \text{if } \length{\block{k-1}} + 1 \leq L \leq 2 \length{\block{k-1}} - \length{\block{k-2}} \\
0 &  \text{if } 2 \length{\block{k-1}} - \length{\block{k-2}} + 1 \leq L \leq \length{\block{k}}
\end{cases}
\end{align*}
for $k \geq 1$ and $ \length{\block{k-1}} + 1 \leq L \leq \length{\block{k}} $. Note that $ \card{ \alphab_{k} } - \card{ \alphab_{k+1} } = \mathbbm{1}_{\cmplmt{\alphab_{k+1}}}(a_{k}) $ holds for all $k$.
\end{coro}

\subsection{Computing the Complexity}

With the growth $ \growth(L) = \comp(L+1) - \comp(L) $ known for all $L \geq 0$, we can now prove an explicit formula for the complexity function. For the computation we have to distinguish three cases and we will split up the result into three statements accordingly.

Since the additional increase observed in Proposition~\ref{prop:geq3} occurs only for $ L \geq \length{\block{0}} + 1$, we tread the complexity in the case of $ L \leq \length{\block{0}} + 1 $ separately:
 
\begin{prop}[\textbf{Complexity function I}]
\label{prop:comp1}
The first values of the complexity function are
\begin{align*}
\comp( L ) &= ( \card{ \alphab_{0} } - 1 ) L + 1 & & \hspace{-4em} \text{for } 0 \leq L \leq \length{\block{0}} \\
\text {and} \quad \comp( L) &= ( \card{ \alphab_{0} } - 1 ) L + \mathbbm{1}_{\alphab_{1}}(a_{0}) & & \hspace{-4em} \text{for } L = \length{\block{0}} + 1 \, .
\end{align*}
\end{prop}

\begin{proof}
 The complexity of the length $ L =\length{\block{0}}+1 $ is already known from Corollary~\ref{coro:CPnAndGrowth}. The formula for $ 0 \leq L \leq \length{\block{0}} $ follows from $ \comp(0) = 1 $ and $ \growth(L) = \alphab_{0} - 1 $ for $ 0 \leq L \leq \length{\block{0}} - 1 $.
\end{proof}

The following two theorems will deal with the case $ L \geq \length{\block{0}}+2 $. It turns out that it is important to distinguish between $ n_{k} = 2 $ and $ n_{k} > 2 $, since the order, in which the cases in the formula for $\growth(L)$ in Corollary~\ref{coro:CPnAndGrowth} change, is different for $ n_{k} = 2 $ and $ n_{k} > 2 $.

\begin{thm}[Complexity function II]
\label{thm:comp2}
For $k \geq 1$ with $ n_{k}  = 2 $, the complexity function in the range $ \length{\block{k-1}} + 2 \leq L \leq \length{\block{k}} + 1 $ is given by
\begin{align*}
\comp( L ) &= ( \card{ \alphab_{k+1} } - 1 ) L + ( \card{ \alphab_{k-1} } - \card{ \alphab_{k+1} } ) ( \length{\block{k-1}} + 1) \\
& \quad + \mathbbm{1}_{\alphab_{k}}(a_{k-1}) \cdot \begin{cases}
- \length{\block{k-1}} + \length{\block{k-2}} + L & \text{if } \length{\block{k-1}} + 2 \leq L \leq \length{\block{k}} - \length{\block{k-2}}  \\
\length{\block{k-1}} + 1 & \text{if } \length{\block{k}} - \length{\block{k-2}} + 1 \leq L \leq \length{\block{k}} + 1
\end{cases}
\end{align*}
\end{thm}

\begin{proof}
Because of $n_{k} = 2 $ we have $ \length{\block{k}} = 2 \length{\block{k-1}} + 1 $ and the growth given in Corollary~\ref{coro:CPnAndGrowth} simplifies to
\[ \growth( L ) = \card{ \alphab_{k+1} } - 1 + \begin{cases}
\mathbbm{1}_{\alphab_{k}}(a_{k-1}) &  \text{if } \length{\block{k-1}} + 1 \leq L \leq \length{\block{k}} - \length{\block{k-2}} - 1 \\
0 &  \text{if } \length{\block{k}} - \length{\block{k-2}} \leq L \leq \length{\block{k}}
\end{cases} \, .\]
We can now compute the complexity from the value of $\comp( \length{\block{k-1}} + 1 ) $ and the growth. First we consider the case of $ \length{\block{k-1}} + 2 \leq L \leq \length{\block{k}} - \length{\block{k-2}} $:
\begin{align*}
\comp( L ) &= \comp( \length{\block{k-1}} + 1 ) + \sum_{j = \length{\block{k-1}}+1}^{L-1} \growth( j ) \\
&= ( \card{ \alphab_{k-1} } - 1 ) \cdot (\length{\block{k-1}}+1) + \mathbbm{1}_{ \alphab_{k} } (a_{k-1}) \cdot ( \length{\block{k-2}}+1 ) \\
& \qquad + (L - \length{\block{k-1}} - 1) \cdot ( \card{ \alphab_{k+1} } - 1 + \mathbbm{1}_{\alphab_{k}}(a_{k-1}) ) \\
&= ( \card{ \alphab_{k-1} } - \card{ \alphab_{k+1} } ) ( \length{\block{k-1}} + 1) - \mathbbm{1}_{\alphab_{k}}(a_{k-1}) \cdot ( \length{\block{k-1}} - \length{\block{k-2}} ) \\
& \qquad + ( \card{ \alphab_{k+1} } - 1 + \mathbbm{1}_{\alphab_{k}}(a_{k-1}) ) L 
\end{align*}
Similarly we obtain the complexity for $ \length{\block{k}} - \length{\block{k-2}} + 1 \leq L \leq \length{\block{k}} + 1 $:
\begin{align*}
\comp( L ) &= \comp( \length{\block{k-1}} + 1 ) + \sum_{ j = \length{\block{k-1}}+1 }^{L-1} \growth( j ) \\
&= ( \card{ \alphab_{k-1} } - 1 ) \cdot (\length{\block{k-1}}+1) + \mathbbm{1}_{ \alphab_{k} } (a_{k-1}) \cdot ( \length{\block{k-2}}+1 ) \\
& \qquad + ( \card{ \alphab_{k+1} } - 1 ) ( L - 1 - \length{\block{k-1}} ) + \mathbbm{1}_{\alphab_{k}}( a_{k-1} ) ( \length{\block{k-1}} - \length{\block{k-2}} ) \\
&= ( \card{ \alphab_{k-1} } - \card{ \alphab_{k+1} } ) ( \length{\block{k-1}} + 1) + \mathbbm{1}_{ \alphab_{k} } (a_{k-1}) \cdot ( \length{\block{k-1}} + 1 ) \\
& \qquad + ( \card{ \alphab_{k+1} } - 1 ) L \qedhere
\end{align*} 
\end{proof}

\begin{thm}[Complexity function III]
For $k \geq 1$ with $ n_{k} > 2 $, the complexity function in the range $ \length{\block{k-1}} + 2 \leq L \leq \length{\block{k}} + 1 $ is given by
\begin{align*}
\comp( L ) = (\length{\block{k-1}}+1)  + ( \card{ \alphab_{k} } - 1) L \, + & \begin{cases}
\mathbbm{1}_{\alphab_{k}}(a_{k-1})( L - 2 \length{\block{k-1}} + \length{\block{k-2}} - 1 ) \\
0 \\
- \mathbbm{1}_{ \cmplmt{ \alphab_{k+1} } }(a_{k}) ( L - \length{\block{k}} + \length{\block{k-1}} ) 
\end{cases} \\
& \begin{cases}
\text{if } \length{\block{k-1}} + 2 \leq L \leq 2 \length{\block{k-1}} - \length{\block{k-2}} + 1\\ 
\text{if } 2 \length{\block{k-1}} - \length{\block{k-2}} + 2 \leq L \leq \length{\block{k}} - \length{\block{k-1}}\\
\text{if } \length{\block{k}} - \length{\block{k-1}} +1 \leq L \leq \length{\block{k}} +1
\end{cases} \hspace{-1em} .
\end{align*}
\end{thm}

\begin{proof}
To shorten notation, we introduce the abbreviations
\begin{align*}
I_{1} &:= \mathbb{Z} \cap \big[ \length{\block{k-1}} + 2 \, , \, 2 \length{\block{k-1}} - \length{\block{k-2}} + 1 \big] \, ,\\
I_{2} &:=\mathbb{Z} \cap  \big[ 2 \length{\block{k-1}} - \length{\block{k-2}} + 2 \, , \, \length{\block{k}} - \length{\block{k-1}} \big] \, ,\\
I_{3} &:=\mathbb{Z} \cap  \big[ \length{\block{k}} - \length{\block{k-1}} +1 \, , \, \length{\block{k}} +1 \big] \, .
\end{align*}
We proceed similar to the proof of Theorem~\ref{thm:comp2}. Because of $ n_{k} > 2 $ we have $ \length{\block{ k }} - \length{\block{k-1}} = (n_{k} - 1) ( \length{\block{k-1}} + 1) > 2 \length{\block{k-1}} - \length{\block{ k-2 }} $ and the growth given in Corollary~\ref{coro:CPnAndGrowth} simplifies to
\[ \growth(L) = \card{ \alphab_{k} } - 1 + \begin{cases}
\mathbbm{1}_{\alphab_{k}}(a_{k-1}) &  \text{if } L+1 \in I_{1} \\
0 &  \text{if } L+1 \in I_{2} \\
- ( \card{ \alphab_{k} } - \card{ \alphab_{k+1} }  ) & \text{if } L+1 \in I_{3}
\end{cases} \, .\]
We compute the complexity from the value of $\comp( \length{\block{k-1}} + 1 ) $ and the growth:
\begin{align*}
& \comp( L ) = \comp( \length{\block{k-1}} + 1 ) + \sum_{j = \length{\block{k-1}}+1}^{L-1} \growth( j ) \\
& = ( \card{ \alphab_{k-1} } - 1 ) (\length{\block{k-1}}+1) + \mathbbm{1}_{ \alphab_{k} } (a_{k-1}) ( \length{\block{k-2}}+1 ) + ( \card{ \alphab_{k} } - 1) (L - \length{\block{k-1}}-1) \\
& \quad + \begin{cases}
\mathbbm{1}_{\alphab_{k}}(a_{k-1}) ( L - \length{\block{k-1}} - 1 ) & \text{if } L \in I_{1} \\
\mathbbm{1}_{\alphab_{k}}(a_{k-1})( \length{\block{k-1}} - \length{\block{k-2}} ) & \text{if } L \in I_{2} \\
\mathbbm{1}_{\alphab_{k}}(a_{k-1}) ( \length{\block{k-1}} - \length{\block{k-2}} ) - ( \card{ \alphab_{k} } - \card{ \alphab_{k+1} } ) ( L - \length{\block{k}} + \length{\block{k-1}} ) & \text{if } L \in I_{3}
\end{cases}\\
&= ( \card{ \alphab_{k-1} } - \card{ \alphab_{k} } ) (\length{\block{k-1}}+1) + \mathbbm{1}_{ \alphab_{k} } (a_{k-1}) (\length{\block{k-1}} + 1 ) + ( \card{ \alphab_{k} } - 1) L \\
& \quad + \begin{cases}
\mathbbm{1}_{\alphab_{k}}(a_{k-1})( L - 2 \length{\block{k-1}} + \length{\block{k-2}} - 1 ) & \text{if } L \in I_{1} \\
0 & \text{if } L \in I_{2} \\
- ( \card{ \alphab_{k} } - \card{ \alphab_{k+1} } ) ( L - \length{\block{k}} + \length{\block{k-1}} ) & \text{if } L \in I_{3}
\end{cases} \\
&= (\length{\block{k-1}}+1)  + ( \card{ \alphab_{k} } - 1) L + \begin{cases}
\mathbbm{1}_{\alphab_{k}}(a_{k-1})( L - 2 \length{\block{k-1}} + \length{\block{k-2}} - 1 ) & \text{if } L \in I_{1} \\
0 & \text{if } L \in I_{2} \\
- ( \card{ \alphab_{k} } - \card{ \alphab_{k+1} } ) ( L - \length{\block{k}} + \length{\block{k-1}} ) & \text{if } L \in I_{3}
\end{cases}
\end{align*}
Now the relation $  \card{ \alphab_{k} } - \card{ \alphab_{k+1} } = \mathbbm{1}_{ \cmplmt{ \alphab_{k+1} } }(a_{k}) $ yields the claim.
\end{proof}

Recall from Subsection~\ref{subsec:SimpleTWords} that $ \eventNr $ denotes the number from which on $ \alphab_{k} $ is equal to $\alphabEv$. For $ k \geq \eventNr + 1 $, the above formulas for the complexity simplify and the difference between $ n_{k} = 2 $ and $ n_{k} > 2 $ vanishes:

\begin{coro}
\label{coro:compAlphabEv}
For $ k  \geq \eventNr + 1 $ the complexity function in the range $ \length{\block{k-1}} + 2 \leq L \leq \length{\block{k}} + 1 $ is given by
\begin{align*}
\comp( L ) &= \begin{cases}
\card{ \alphabEv } \cdot L - \length{\block{k-1}} + \length{\block{k-2}} & \text{if } \length{\block{k-1}} + 2 \leq L \leq 2 \length{\block{k-1}} - \length{\block{k-2}} + 1 \\
( \card{ \alphabEv } - 1 ) \cdot L + \length{\block{k-1}} + 1 & \text{if } 2 \length{\block{k-1}} - \length{\block{k-2}} + 2 \leq L \leq \length{\block{k}} + 1
\end{cases} \, .
\end{align*}
\end{coro}

\begin{exmpl}
For the Grigorchuk subshift we have $ \length{\block{k}} = 2^{k+1} - 1 $, $\card{ \alphab_{0} } = 4$, $ \card{ \alphabEv } = 3$ and $ \eventNr = 1$. In particular, $ \mathbbm{1}_{\alphab_{1}}(a_{0}) = 0 $ holds. From Proposition~\ref{prop:comp1} and Theorem~\ref{thm:comp2} for $k = 1$ and from Corollary~\ref{coro:compAlphabEv} for $ k \geq 2 $, we recover the complexity function that was given for the Grigorchuk subshift in \cite{GLN_Combinat}, Theorem 1:
\begin{align*}
\comp( L ) &= 3L + 1 \quad \text{for } 0 \leq L \leq 1 \, ,\\
\comp( L) &= 3 L \hspace{1.7em} \quad \text{for } L = 2 \, , \\
\comp( L) &= 2 L + 2 \quad \text{for } 3 \leq L \leq 4 \\
\text {and} \quad \comp( L ) &= \begin{cases}
3 \cdot L - 2^{k} + 2^{k-1} & \text{if } 2^{k} + 1 \leq L \leq 2^{k+1} -  2^{k-1} \\
2 \cdot L + 2^{k} & \text{if } 2^{k+1} - 2^{k-1} + 1 \leq L \leq 2^{k+1}
\end{cases} \quad \text{for } k \geq 2 \, .
\end{align*}
\end{exmpl}

\begin{prop}
\label{prop:MinMaxQuotient}
For $ k \geq \eventNr + 1 $ and $ \length{\block{k-1}}+2 \leq L \leq \length{\block{k}}+1 $, the quotient $ \frac{\comp( L )}{L} $ lies between
\[ \card{ \alphabEv } - 1 < \min \Big\{ \card{ \alphabEv - \frac{n_{k}-1}{n_{k}} ,\, \card{ \alphabEv } - \frac{n_{k-1}-1}{n_{k-1}}} \Big\} \leq \min_{ \length{\block{k-1}}+2 \leq L \leq \length{\block{k}}+1 } \frac{\comp( L )}{L} \]
and
\[ \max_{ \length{\block{k-1}}+2 \leq L \leq \length{\block{k}}+1 } \frac{\comp( L )}{L} = \card{ \alphabEv } - \frac{ n_{k-1}-1 }{ 2 n_{k-1} - 1 } \leq \card{ \alphabEv } - \frac{1}{3} \, ,\]
where the maximum is taken at $ L = 2 \length{\block{k-1}} - \length{\block{k-2}} + 1 $.
\end{prop}

\begin{proof}
From Corollary~\ref{coro:compAlphabEv} we obtain
\begin{align*}
\frac{ \comp( L ) }{L} &= \begin{cases}
\card{ \alphabEv } - \frac{ \length{\block{k-1}} - \length{\block{k-2}} }{L} & \text{if } \length{\block{k-1}} + 2 \leq L \leq 2 \length{\block{k-1}} - \length{\block{k-2}} + 1 \\
\card{ \alphabEv } - 1 + \frac{ \length{\block{k-1}} + 1}{L} & \text{if } 2 \length{\block{k-1}} - \length{\block{k-2}} + 2 \leq L \leq \length{\block{k}} + 1
\end{cases}
\end{align*}
for $ k \geq \eventNr + 1 $. Therefore, the maximum in the interval $ [ \length{\block{k-1}}+2 \, ,\; \length{\block{k}}+1 ]$ is taken at either $ L_{1} := 2 \length{\block{k-1}} - \length{\block{k-2}} + 1 $ or $ L_{2} := 2 \length{\block{k-1}} - \length{\block{k-2}} + 2 $. Direct computation shows that the value at $L_{1}$ is greater:
\[ \frac{ \comp( L_{1} ) }{L_{1}} = \card{ \alphabEv } - \frac{ \length{\block{k-1}} - \length{\block{k-2}} }{ 2 \length{\block{k-1}} - \length{\block{k-2}} + 1 } = \card{ \alphabEv } - 1 + \frac{ \length{\block{k-1}} + 1 }{ 2 \length{\block{k-1}} - \length{\block{k-2}} + 1 } >  \frac{ \comp( L_{2} ) }{L_{2}} \, . \]
Thus the maximum is taken at $L_{1}$ and by dividing both numerator and denominator of $ \frac{\comp( L_{1} )}{L_{1}} $ by $ \length{\block{k-2}} + 1 $, we obtain the claimed expression. Similarly, the minimum in the interval is taken at either $ L_{3} = \length{\block{k-1}} + 2 $ or at $ L_{4} := \length{\block{k}} + 1 $. While the value of the minimum depends on the values of $n_{k-1}$ and $n_{k}$, a short calculation yields the claim since
\[ \frac{ \comp( L_{3} ) }{L_{3}} = \card{ \alphabEv } - \frac{ \length{\block{k-1}} - \length{\block{k-2}} }{ \length{\block{k-1}} + 2 } > \card{ \alphabEv } - \frac{ ( \length{\block{k-1}} +1)(1 - \frac{1}{n_{k-1}} ) }{ \length{\block{k-1}} + 1 } =  \card{ \alphabEv } - \frac{n_{k-1}-1}{n_{k-1}} \]
and
\[ \frac{ \comp( L_{4} ) }{L_{4}} = \card{ \alphabEv } - 1 + \frac{ \length{\block{k-1}} + 1}{ \length{\block{k}} + 1 } = \card{ \alphabEv } - 1 + \frac{1}{ n_{k} } = \card{ \alphabEv } - \frac{n_{k}-1}{ n_{k} } \, . \qedhere \]
\end{proof}

\begin{rem}
In particular, Proposition~\ref{prop:MinMaxQuotient} yields
\[ \card{ \alphabEv } - 1 \leq \liminf_{L \to \infty} \frac{ \comp( L ) }{L} \leq \limsup_{L \to \infty} \frac{ \comp( L ) }{L} \leq \card{ \alphabEv } - \frac{1}{3} <  \card{ \alphabEv } \, . \]
Besides the regularity of the Toeplitz word, this is an alternative way to prove unique ergodicity for simple Toeplitz words with $ \card{ \alphabEv } \leq 3 $: By Theorem~1.5 in \cite{Boshernitzan_UniErgodic}, every minimal subshift with $ \limsup_{L \to \infty} \frac{\comp( L )}{L} < 3 $ is uniquely ergodic. Such a proof was for example used in \cite{DKMSS_Regul-Article}.
\end{rem}


\section{De Bruijn Graphs and Palindrome Complexity}
\label{sec:DeBruijnPali}

In this section, we investigate a sequence of graphs which are called de Bruijn graphs or Rauzy graphs. The $L$-th graph in this sequence encodes in its vertices which words of length $L$ occur in the subshift. In its edges, it shows by which letter(s) a word can be extended to the right or to the left. To construct the graphs, we rely heavily on the results about the complexity and their proofs from the previous Section~\ref{sec:SubwordComp}. The construction is carried out in detail in the first subsection. In the second subsection, palindromes are discussed. We prove that reflection symmetry in the graphs corresponds to reflection symmetry of the words. This yields an explicit formula for the palindrome complexity.

\subsection{Description of the de Bruijn Graphs}
\label{subsec:DescripdeBruijn}

Let $\subshift$ be a subshift. The $L$-th de Bruijn graph $\debruijn{L} = (\vertices{L}, \edges{L}) $ associated to $\subshift$ is a directed graph with vertices
\[ \vertices{L} := \{ u \in \langu{\subshift} : \length{u} = L \} \]
and an edge $ (u, v) \in \edges{L} \subseteq \vertices{L} \times \vertices{L} $ from $ u = u(1) \hdots u(L) \in \vertices{L}$ to $ v = v(1) \hdots v(L) \in \vertices{L}$ if
\[ u(2) \hdots u(L) = v(1) \hdots v(L-1) \quad  \text{and} \quad  u \; v(L) =  u(1) \; v \in \langu{\subshift} \]
hold. The edges can be interpreted as words of length $L+1$, thus encoding the possible extensions of words of length $L$. Hence, the results from the previous Section~\ref{sec:SubwordComp} about right special words and their extensions describe the branching points of the graph and their branching behaviour.  We know from the proof of Proposition~\ref{prop:geq2} that the suffix $v_{1}$, that consists of the last $L$ letters of $\block{k}$, is right special and can be extended by all letters in $\alphab_{k}$ if $k \geq 0$ and $L \leq \length{\block{k}} - \length{\block{k-1}} - 1 $. Another right special word is given in the proof of Proposition~\ref{prop:geq3} for $ L \geq \length{ \block{ 0 }} +1 $. We will discuss its implications later, but first we consider $ 1 \leq L \leq \length{ \block{ 0 }} = n_{0} - 1 $, where Proposition~\ref{prop:geq3} does not apply.

The suffix of length $ L \leq \length{ \block{0} } $ of $\block{0}$ is given by $v_{1} = a_{0}^{L}$. For $ L \leq \length{\block{0}} - \length{\block{-1}} - 1 = n_{0} - 2  $, the extension with all letters in $\alphab_{0}$ is possible. Extension by $a_{0}$ yields an edge back to the vertex $v_{1}$. Extension by a letter $b \in \alphab_{0} \setminus \{ a_{0} \} $ yields an edge from $v_{1}$ to the word $a_{0}^{L-1} \, b $. Since $v_{1}$ is the prefix of $\block{0}$ as well, the word $v_{1} b v_{1}$ exists and we get a loop with $ L+1 $ edges from $v_{1}$ via $a_{0}^{L-1} \, b $ back to $v_{1}$. The words encountered in the loop are clearly pairwise different and also different from all words encountered when we extend $v_{1}$ with a letter different from $b$. We obtain the graph shown in Figure~\ref{fig:deBruijn1ToL0}.  For $L = n_{0}-1$, the extensions with letters $b \neq a_{0} $ remain the same, but the edge from $v_{1}$ directly back to itself exists if and only if $a_{0} \in \alphab_{1}$ holds.

\begin{figure}
\centering
\footnotesize
\begin{tikzpicture}
[ every path/.style = {shorten <=1pt, shorten >=1pt, >=stealth},
  vertex/.style = {circle, minimum size=17pt, inner sep=2pt, draw}]

\node (dotstop) {$\hdots$};
\node [vertex] (vtopleft)  [left = of dotstop] {};
\node [vertex] (vtopright) [right = of dotstop] {};
\node (dotsmiddle) [below = of dotstop] {$\vdots$};
\node (dotsbottom) [below = of dotsmiddle] {$\hdots$};
\node [vertex] (vbottomleft) [left = of dotsbottom] {};
\node [vertex] (vbottomright) [right = of dotsbottom] {};
\node [vertex] (w0w0) [below = of dotsbottom] {$v_{1}$};
\graph{
(w0w0) -> [controls ={ +(10em, -5ex) and + (8em, 0)}] (vtopright.east);
(vtopright.west)  -> (dotstop) -> (vtopleft.east);
(w0w0) <-  [controls ={ +(-10em, -5ex) and + (-8em, 0ex)}] (vtopleft.west);
(w0w0) -> [controls ={ +(10em, 0) and + (10em, 0)}, edge node ={ node [at end, above, xshift = -1em, yshift = 1.5ex] {$L + 1$ edges on each arc} }] (dotsmiddle.north east);
(w0w0) <- [controls ={ +(-10em, 0) and + (-10em, 0)}]  (dotsmiddle.north west);
(w0w0) -> [controls ={ +(8.5em, 5ex) and + (8.5em, 0)}]  (dotsmiddle.south east);
(w0w0) <- [controls ={ +(-8.5em, 5ex) and + (-8.5em, 0)}] (dotsmiddle.south west);
(w0w0) -> [bend right = 15, edge node= {node [right, yshift=-1ex, fill = white, inner sep = 0] {{$\underbrace{\hspace{1.8cm}}_{\card{ \alphab_{0} } - 1 }$} } } ] (vbottomright) -> (dotsbottom) -> (vbottomleft)  -> [bend right = 15] (w0w0);
};
\draw [->] (w0w0.south east) .. controls +(1.6cm, -1.8cm) and +(-1.6cm, -1.8cm) .. node [right, near start, xshift=0.5em, align = left]{for $L = n_{0}-1$ this edge exists\\ if and only if $ a_{0} \in \alphab_{1} $ holds} (w0w0.south west);
\end{tikzpicture}
\normalsize
\caption{The de Bruijn graph for $ 1 \leq L \leq \length{ \block{ 0 }} $ .}
\label{fig:deBruijn1ToL0}
\end{figure}

Next we turn to the case $ k \geq 1 $ and $ \length{ \block{ k-1 }}+1 \leq L \leq \length{ \block{ k }} $. To keep the notation short, we define $r := L \mod (\length{ \block{ k-1 }} + 1)$. By $u_{1}$ (respectively, $v_{1}$) we denote the prefix (respectively, the suffix) of length $L$ of $\block{k}$. It follows from the decomposition $ \block{k} = \block{k-1} a_{k} \block{k-1} a_{k} \hdots a_{k} \block{k-1} $ that $v_{1}$ is obtained when $u_{1}$ is shifted $ \length{ \block{ k-1 }}-r $ positions to the right.

As before we know that $v_{1}$ can be extended to the right by all letters in $ \alphab_{k} $ if $ \length{ \block{ k-1 }} + 1 \leq L \leq \length{ \block{ k }} - \length{ \block{ k-1 }} - 1 $ holds (see  Proposition~\ref{prop:geq2}). For $ \length{ \block{ k }} - \length{ \block{ k-1 }}  \leq L \leq \length{ \block{ k }}$, the extension of $v_{1}$ by $a_{k}$ exists if and only if $a_{k} \in \alphab_{k+1}$ holds. When $v_{1}$ is extended with a letter in $\alphab_{k} \setminus \{ a_{k} \}$, then shifting $v_{1}$ by $ L+1 $ positions along this extension will result in $u_{1}$. When $v_{1}$ is extended with $a_{k}$, then already a shift by $ r+1 $ positions will result in $u_{1}$.  This yields the graph in Figure~\ref{fig:deBruijnAllg1}. If Proposition~\ref{prop:geq3} doesn't apply because either $a_{k-1} \notin \alphab_{k} $ or $ 2 \length{\block{k-1}} - \length{\block{k-2}} < L $ holds, then this is the de Bruijn graph for $L$.

\begin{figure}
\centering
\footnotesize
\begin{tikzpicture}
[ every path/.style = {shorten <=1pt, shorten >=1pt, >=stealth, absolute, draw},
  vertex/.style = {circle, minimum size=14pt, inner sep=1pt, draw}]
\matrix[row sep = 3ex, column sep = 3em]{
& \node [vertex] (zu) {}; & \node (dotstopz) {$\cdots$}; & \node [vertex] (vz) {}; & \\
& & \node (dotstopy2){$\cdots$}; & &\\
& & \node (dotstopy1) {$\cdots$}; & &\\
& \node [vertex] (xu) {}; & \node (dotstopx) {$\cdots$}; & \node [vertex] (vx) {}; & \\
\node [vertex] (u) {$u_{1}$}; & \node [vertex] (ux) {}; & \node (dotsmiddle) {$\cdots$}; & \node [vertex] (xv) {};	& \node [vertex] (v) {$v_{1}$}; \\
& \node [vertex] (tu) {}; &  \node (dotsbottom) {$\cdots$}; & \node [vertex] (vt) {}; & \\
};
\graph [use existing nodes]{
u -> ux ->[edge node= {node [above, xshift = 2.5em] {$\length{ \block{ k-1 }} - r $ edges} }]  dotsmiddle -> xv -> v -> {
	vt [> out = 240, > in = 0,  target edge node = { node [right, xshift = 1em, yshift = -1.5ex, align = left] {this arc exists for $ L \geq \length{ \block{ k }} - \length{ \block{ k-1 }} $ \\ if and only if $ a_{k} \in \alphab_{k+1} $ holds }}] -> dotsbottom [target edge node= {node [below, xshift = 2em, yshift=-1.5ex] {$ v_{1}|_{[2, L]} a_{k} $} }, source edge node= {node [below, xshift = 1.5em] {$ r+1 $ edges} }] -> tu [< out = 180, < in = 300],
	vx [> out = 120, > in = 0] -> dotstopx -> xu [< out = 180, < in = 60],	
	dotstopy1 [> out = 90, > in = 0, target edge node = {node [above, near end] {$\vdots$} } , < out = 180, < in = 90, source edge node = {node [above, near start] {$\vdots$} }],	
	dotstopy2 [> out = 60, > in = 0, < out = 180, < in = 120],	
	vz [> out =30, > in = 340, target edge node = {node [xshift = -1em, yshift = -8ex, fill = white ,inner sep = 1pt] { $\overbrace{\hspace{1.4cm}}^{\card{ \alphab_{k} } -1 }$} }] -> dotstopz [target edge node= {node [above, xshift = 6.5em, yshift=1.5ex] {$ v_{1}|_{[2, L]} \, a \quad $with $ a \in  \alphab_{k} \setminus \{ a_{k} \} $} }]  -> [edge node= {node [below, xshift = 2.5em, yshift=-2ex] {$ L + 1 $ edges on each arc} }] zu [< out = 200, < in = 150]
} -> u;
};
\end{tikzpicture}
\normalsize
\caption{The de Bruijn graph $\debruijn{L}$ for $ k \geq 1 $ and $ \length{ \block{ k-1 }} + 1 \leq L \leq \length{ \block{ k }} $, when either $a_{k-1} \notin \alphab_{k} $ or $ L > 2 \length{\block{k-1}} - \length{\block{k-2}} $ holds. The number of edges of an arc refers to the distance between $u_{1}$ and $v_{1}$.}
\label{fig:deBruijnAllg1}
\end{figure}

If Proposition~\ref{prop:geq3} applies, then the de Bruijn graph has to be adjusted, since there is another right special word $v_{2} \neq v_{1}$. It is the suffix of length $L$ of $ \block{k-1} a_{k-1} \block{k-1} $ and can be extended with both, $a_{k-1}$ and $a_{k}$. Note that every suffix of $ \block{k-1} a_{k-1} \block{k-1} $ is contained in $ \block{k} a_{k-1} \block{k} $. To see how many positions we have to shift from the suffix $v_{1}$ of the first $ \block{k} $ to reach $ v_{2} $, we observe that after the shift, the right end of the word has to align with the right end of a $\block{k-2}$-block. When we reach $v_{2}$ for the first time, the leftmost letter of the word has to be in the leftmost $ \block{k-2} a_{k-1} $ of the rightmost $\block{k-1}$-block of the left $\block{k}$-block, see Figure~\ref{fig:v1v2Location}.

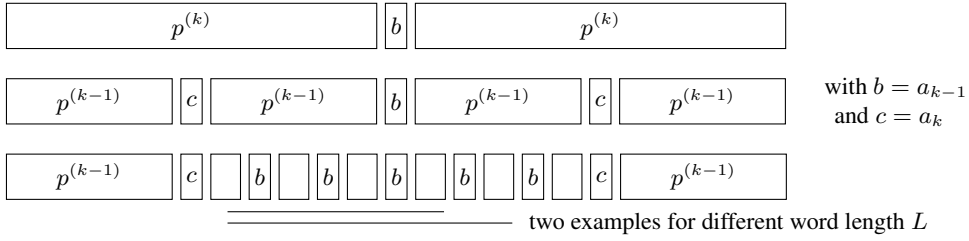
\begin{figure}
\centering
\footnotesize

\begin{tikzpicture}
[every node/.style ={rectangle, inner sep=0pt, minimum height=0.6cm},
pn/.style={minimum width=4.87cm, draw},
pn-1/.style={minimum width=2.18cm, draw},
pn-2/.style={minimum width=0.39cm, draw},
buchst/.style={minimum width=0.28cm, draw}]
\node [right] at (0,0) [pn] (A) {$\block{k}$};
\node [buchst] (A) [right=0.1cm of A] {$b$};
\node [pn] (A) [right=0.1cm of A] {$\block{k}$};
\node [right] at (0,-1) [pn-1] (A) {$\block{k-1}$};
\node [buchst] (A) [right=0.1cm of A] {$c$};
\node [pn-1] (A) [right=0.1cm of A] {$\block{k-1}$};
\node [buchst] (A) [right=0.1cm of A] {$b$};
\node [pn-1] (A) [right=0.1cm of A] {$\block{k-1}$};
\node [buchst] (A) [right=0.1cm of A] {$c$};
\node [pn-1] (A) [right=0.1cm of A] {$\block{k-1}$};
\node [right=0.5cm of A, align=left] {with $b = a_{k-1}$\\ \, and $c = a_{k}$};
\node [right] at (0,-2) [pn-1] (A) {$\block{k-1}$};
\node [buchst] (A) [right=0.1cm of A] {$c$};
\node[pn-2] (Fix1) [right=0.1cm of A] {$$};
\node [buchst] (A) [right=0.1cm of Fix1] {$b$};
\node[pn-2] (A) [right=0.1cm of A] {$$};
\node [buchst] (A) [right=0.1cm of A] {$b$};
\node[pn-2] (A) [right=0.1cm of A] {$$};
\node [buchst] (A) [right=0.1cm of A] {$b$};
\node[pn-2] (Fix2) [right=0.1cm of A] {$$};
\node [buchst] (A) [right=0.1cm of Fix2] {$b$};
\node[pn-2] (Fix3) [right=0.1cm of A] {$$};
\node [buchst] (A) [right=0.1cm of Fix3] {$b$};
\node[pn-2] (A) [right=0.1cm of A] {$$};
\node [buchst] (A) [right=0.1cm of A] {$c$};
\node [pn-1] (A) [right=0.1cm of A] {$\block{k-1}$};
\node [minimum height=0cm] (A) [below=0.15 of Fix1.south] {};
\node [minimum height=0cm] (B) [below=0.15 of Fix2.south east] {};
\node [minimum height=0cm] (C) [below=0.3 of Fix1.south] {};
\node [minimum height=0cm] (D) [below=0.3 of Fix3.south east, label={right: \; $v_{2}$}] {};
\draw (A) -- (B);
\draw (C) -- (D);
\end{tikzpicture}

\normalsize
\caption{Location of the first instance of $v_{2}$ that is reached by shifting from $v_{1}$: two examples for different word lengths $L$.}
\label{fig:v1v2Location}
\end{figure}

Let $ \widetilde{r} := L \mod ( \length{\block{ k-2 }} + 1 ) $. Starting in $v_{1}$, we reach the branching point $v_{2}$ for the first time after $ r + 1 + \length{\block{ k-2 }} - \widetilde{r} $ shifts. The vertex $v_{2}$ lies on the arc that represents the extension of $v_{1}$ by $a_{k-1} \in \alphab_{k} \setminus \{ a_{k} \}$. When we follow the arc from $v_{1}$ to $v_{2}$, we pass through the prefix of length $L$ of $ \block{ k-1 } a_{k-1} \block{ k-1 } $. Denote this prefix by $u_{2}$. We reach it after $ r+1 $ shifts from $v_{1}$. Continuing from $u_{2}$ along the arc, we reach $v_{2}$, where the arc splits in two paths. One path is the extension of $v_{2}$ with $a_{k-1}$. On this path, we reach the prefix $u_{2}$ after $ \widetilde{r}+1 $ shifts and upon further shifting proceed once more to $v_{2}$. The other path is the extension of $v_{2}$ with $a_{k}$. This path leads us from $v_{2}$ to $u_{1}$, which is reached after $ L - \length{ \block{ k-1 }} = r+1 $ shifts. When we take these facts into account, we obtain the de Bruijn graph as shown in Figure~\ref{fig:deBruijnAllg2}.

\begin{figure}
\centering
\footnotesize
\begin{tikzpicture}
[ every path/.style = {shorten <=1pt, shorten >=1pt, >=stealth, absolute, draw},
  vertex/.style = {circle, minimum size=14pt, inner sep=1pt, draw}]
\matrix[row sep = 3ex, column sep = 1.5em]{
& \node [vertex] (zu) {}; & & \node (dotstopz) {$\cdots$}; & & \node [vertex] (vz) {}; & \\
& & & \node (dotstopy2){$\cdots$}; & & &\\
& & & \node (dotstopy1) {$\cdots$}; & & &\\
& \node [vertex] (xu) {}; & & \node (dotstopx) {$\cdots$}; & & \node [vertex] (vx) {}; & \\
\node [vertex] (u) {$u_{1}$}; & \node [vertex] (ux) {}; & & \node (dotsmiddle) {$\cdots$}; & & \node [vertex] (xv) {}; & \node [vertex] (v) {$v_{1}$}; \\
& \node (dotsbottoms1) {$\cdots$}; & \node [vertex] (vtilde) {$v_{2}$}; & \node (dotsbottoms2) {$\cdots$}; & \node [vertex] (utilde) {$u_{2}$}; & \node (dotsbottoms3) {$\cdots$}; & \\
& & & \node (dotsbottomr){$\cdots$}; & & &\\
& \node [vertex] (tu) {}; & & \node (dotsbottomt) {$\cdots$}; & & \node [vertex] (vt) {}; & \\
};
\graph [use existing nodes]{
u -> ux ->[edge node= {node [above, xshift = 3em] {$\length{ \block{ k-1 }} - r $ edges} }]  dotsmiddle -> xv -> v -> {
	vt [> out = 270, > in = 0, target edge node = {node [xshift = 8.5em, yshift = -2.5ex, align = left, fill=white] {this arc exists for  $ L \geq \length{ \block{ k }} - \length{ \block{ k-1 }}$\\ if and only if $ a_{k} \in \alphab_{k+1} $ holds} }] -> dotsbottomt [target edge node= {node [xshift = 3.5em, yshift=-3.5ex] {$ v_{1}|_{[2, L]} a_{k} $} }, source edge node= {node [below, xshift = 3em] {$ r+1 $ edges}}] -> tu [< out = 180, < in = 270],	
	dotsbottoms3 [> out = 240, > in = 0, target edge node = {node [xshift = 3em, yshift=-4ex, inner sep = 0.5ex, fill = white] {$r+1$ edges between $v_{1}$ and $u_{2}$}}] -> utilde -> dotsbottoms2 -> vtilde [target edge node = {node [xshift = 1.7em, yshift = 3ex] {$ \length{ \block{ k-2 }} - \widetilde{r} $ edges between $u_{2}$ and $v_{2}$} }] -> dotsbottoms1 [< out = 180, < in = 300, source edge node = { node [xshift = -3em, yshift = -4ex, inner sep=0.5ex, fill = white] {$r+1$ edges between $v_{2}$ and $u_{1}$} }],
	vx [> out = 120, > in = 0] -> dotstopx -> xu [< out = 180, < in = 60],	
	dotstopy1 [> out = 90, > in = 0, target edge node = {node [above, near end] {$\vdots$} } , < out = 180, < in = 90, source edge node = {node [above, near start] {$\vdots$} }],	
	dotstopy2 [> out = 60, > in = 0, < out = 180, < in = 120],	
	vz [> out =30, > in = 340, target edge node = {node [xshift = -1.3em, yshift = -8.2ex, fill = white, inner sep = 0.5pt] { $\overbrace{\hspace{1.2cm}}^{ \card{ \alphab_{k} } - 2}$} }] -> dotstopz [target edge node= {node [xshift = 8.5em, yshift=3.5ex] {$ v_{1}|_{[2, L]} \, a \quad $with $ a \in  \alphab_{k} \setminus \{ a_{k-1}, a_{k} \} $} }] -> [edge node= {node [xshift = 3.5em, yshift=-3.5ex] {$ L + 1 $ edges on each arc} }] zu [< out = 200, < in = 150]
} -> u;
vtilde ->[out=270, in= 180, edge node= {node [at end, below, xshift = 1em]{$ \widetilde{r}+1 $ edges between $v_{2}$ and $u_{2}$} }] dotsbottomr ->[out = 0, in = 270] utilde;
};
\end{tikzpicture}
\normalsize
\caption{The de Bruijn graph $\debruijn{L}$ for $ k \geq 1 $ and $ \length{ \block{ k-1 }}+1 \leq L \leq 2 \length{\block{k-1}} - \length{\block{k-2}}$ when $a_{k-1} \in \alphab_{k} $ holds. Unless stated otherwise, the number of edges refers to the distance between $u_{1}$ and $v_{1}$.}
\label{fig:deBruijnAllg2}
\end{figure}

From what we have seen in Section~\ref{sec:SubwordComp} it follows that this describes the de Bruijn graph completely. More precisely, we know from Corollary~\ref{coro:CPnAndGrowth} that there are no other branching points besides the ones described in Proposition~\ref{prop:geq2} and \ref{prop:geq3} and that these cannot branch into more paths than discussed. Moreover, we know that every word of length $ L \leq \length{ \block{ k }} $ is contained in a word $ \block{k} a \block{k} $ with $ a \in \alphab_{k+1} $ (see Proposition~\ref{prop:enthalten}). For all $a \in \alphab_{k+1}$, the extension of $v_{1}$ by $a$ is included in the graph, together with all other words that follow when we shift further. Thus, neither edges nor vertices are missing.

\begin{exmpl}
For the Grigorchuk subshift, $\length{ \block{ 0 }} = 1$ and $a = a_{0} \notin \alphab_{1}$ hold. Thus, we obtain the graph $\debruijn{1}$ as shown in Figure~\ref{fig:deBruGrigo1} as a special case of Figure~\ref{fig:deBruijn1ToL0}. Because of $a_{0} \notin \alphab_{1}$, it is Figure~\ref{fig:deBruijnAllg1} which describes the de Bruijn graph for $k=1$. The result is shown in the Figures~\ref{fig:deBruGrigo2} and \ref{fig:deBruGrigo3}. For $ k \geq 2 $, we always have $a_{k-1} \in \alphab_{k}$. Therefore the graph $\debruijn{L}$ is given by Figure~\ref{fig:deBruijnAllg2} for $ \length{ \block{ k-1 }}+1 \leq L \leq \length{\block{k}} - \length{\block{k-2}} - 1 $ and by Figure~\ref{fig:deBruijnAllg1} for $ \length{\block{k}} - \length{\block{k-2}} \leq L \leq \length{\block{k}} $. The existence condition for the bottommost arc in each figure is always satisfied, since $a_{k} \in \alphab_{k+1} $ holds for all $ k \geq 1 $. Using $ \alphab_{k} = \{ a_{k-1}, a_{k}, a_{k+1} \} $ and $\length{ \block{ k }} = 2^{k+1} - 1$, we obtain the graphs shown in the Figures~\ref{fig:deBruGrigoAllg1} and \ref{fig:deBruGrigoAllg2}.
\end{exmpl}

\begin{figure}
\centering
\footnotesize
%
\begin{minipage}[b]{0.3\textwidth}
\centering
\begin{tikzpicture}
[ every path/.style = {shorten <=1pt, shorten >=1pt, >=stealth},
  vertex/.style = {circle, minimum size=17pt, inner sep=2pt, draw}]
\node [vertex] (x) {$x$};
\node [vertex] (y)  [below = of x] {$y$};
\node [vertex] (z) [below = of y] {$z$};
\node [vertex] (a) [below = of z] {$a$};
\graph{
(a) -> [bend right = 90] (x) -> [bend right = 90] (a);
(a) -> [bend right = 60] (y) -> [bend right = 60] (a);
(a) -> [bend right = 30] (z) -> [bend right = 30] (a);
};
\end{tikzpicture}
\subcaption{$\debruijn{1}$ with $v_{1} = a$\label{fig:deBruGrigo1}}
\end{minipage}
\hfill
%
\begin{minipage}[b]{0.3\textwidth}
\centering
\begin{tikzpicture}
[ every path/.style = {shorten <=1pt, shorten >=1pt, >=stealth, absolute, draw},
  vertex/.style = {circle, minimum size=14pt, inner sep=1pt, draw}]
\matrix[row sep = 4ex, column sep = 2em]{
& \node [vertex] (ya) {$ya$}; & \node [vertex] (ay) {$ay$}; & \\
& \node [vertex] (za) {$za$}; & \node [vertex] (az) {$az$}; & \\
\node [vertex] (u) {$ax$}; & & & \node [vertex] (v) {$xa$}; \\
};
\graph [use existing nodes]{
u  -> v -> {
	az [> out = 120, > in = 0] -> za [< out = 180, < in = 60],	
	ay [> out =90, > in = 0] -> ya [< out = 180, < in = 90]
} -> u;
v [< out = 270, < in = 270] -> u;
};
\end{tikzpicture}
\subcaption{$\debruijn{2}$ with $u_{1}= ax$, $v_{1} = xa$ and $r = 0$\label{fig:deBruGrigo2}}
\end{minipage}
\hfill
%
\begin{minipage}[b]{0.35\textwidth}
\centering
\begin{tikzpicture}
[ every path/.style = {shorten <=1pt, shorten >=1pt, >=stealth, absolute, draw},
  vertex/.style = {circle, minimum size=14pt, inner sep=1pt, draw}]
\matrix[row sep = 4ex, column sep = 2em]{
\node [vertex] (yax) {$yax$}; & \node [vertex](aya) {$aya$}; & \node [vertex] (xay) {$xay$}; \\
\node [vertex] (zax) {$zax$}; & \node [vertex] (aza) {$aza$}; & \node [vertex] (xaz) {$xaz$}; \\
& \node [vertex] (axa) {$axa$}; & \\
&  \node [vertex] (xax) {$xax$}; & \\
};
\graph [use existing nodes]{
axa -> {
	xax [> bend left=30, < bend left = 30],	
	xaz [> bend right = 30] -> aza -> zax [< bend right = 30],
	xay [> bend right = 90] -> aya -> yax [< bend right = 90]
} -> axa;
};
\end{tikzpicture}
\subcaption{$\debruijn{3}$ with $u_{1} = v_{1} = axa$ and $r =1$\label{fig:deBruGrigo3}}
\end{minipage}
\caption{The first de Bruijn graphs for the Grigorchuk subshift}
\label{fig:deBruGrigo123}
\normalsize
\end{figure}

\begin{figure}
\centering
\footnotesize
\begin{tikzpicture}
[ every path/.style = {shorten <=1pt, shorten >=1pt, >=stealth, absolute, draw},
  vertex/.style = {circle, minimum size=14pt, inner sep=1pt, draw}]
\matrix[row sep = 3ex, column sep = 1.5em]{
& \node [vertex] (zu) {}; & & \node (dotstopz) {$\cdots$}; & & \node [vertex] (vz) {}; & \\
\node [vertex] (u) {$u_{1}$}; & \node [vertex] (ux) {}; & & \node (dotsmiddle) {$\cdots$}; & & \node [vertex] (xv) {}; & \node [vertex] (v) {$v_{1}$}; \\
& \node (dotsbottoms1) {$\cdots$}; & \node [vertex] (vtilde) {$v_{2}$}; & \node (dotsbottoms2) {$\cdots$}; & \node [vertex] (utilde) {$u_{2}$}; & \node (dotsbottoms3) {$\cdots$}; & \\
& & & \node (dotsbottomr){$\cdots$}; & & &\\
& \node [vertex] (tu) {}; & & \node (dotsbottomt) {$\cdots$}; & & \node [vertex] (vt) {}; & \\
};
\graph [use existing nodes]{
u -> ux ->[edge node= {node [above, xshift = 3em] {$2^{k+1} -  L - 1 $ edges} }]  dotsmiddle -> xv -> v -> {
	vt [> out = 270, > in = 0] -> dotsbottomt [target edge node= {node [xshift = 3.5em, yshift=-3.5ex] {$ v_{1}|_{[2, L]} a_{k} $} }, source edge node= {node [xshift = 3em, yshift=-2.3ex] {$  L - 2^{k} + 1 $ edges}}] -> tu [< out = 180, < in = 270],	
	dotsbottoms3 [> out = 240, > in = 0, target edge node = {node [xshift = 3em, yshift=-4ex, fill = white, inner sep = 0.5ex] {$ L - 2^{k}+1$ edges between $v_{1}$ and $u_{2}$}}] -> utilde -> dotsbottoms2 -> vtilde [target edge node = {node [above, xshift = 1.7em, yshift = 1ex] {$ 2^{k-1} - 1 - \widetilde{r} $ edges between $u_{2}$ and $v_{2}$} }] -> dotsbottoms1 [< out = 180, < in = 300, source edge node = { node [xshift = -3em, yshift = -4ex, fill = white, inner sep = 0.5ex] {$ L - 2^{k} + 1$ edges between $v_{2}$ and $u_{1}$} }],
	vz [> out =90, > in = 0] -> dotstopz [target edge node= {node [xshift = 3.5em, yshift=3.5ex] {$ v_{1}|_{[2, L]} \, a_{k+1} $} }] -> [edge node= {node [xshift = 3.5em, yshift=2ex] {$ L + 1 $ edges} }] zu [< out = 180, < in = 90]
} -> u;
vtilde ->[out=270, in= 180, edge node= {node [at end, below, xshift = 1em]{$ \widetilde{r}+1 $ edges between $v_{2}$ and $u_{2}$} }] dotsbottomr ->[out = 0, in = 270] utilde;
};
\end{tikzpicture}
\normalsize
\caption{The de Bruijn graph for the Grigorchuk subshift for $ k \geq 2 $ for and $ 2^{k} \leq L \leq 2^{k+1} - 2^{k-1} - 1$. We have $r = L - 2^{k}$. Unless stated otherwise, the number of edges refers to the distance between $u_{1}$ and $v_{1}$.}
\label{fig:deBruGrigoAllg1}
\end{figure}

\begin{figure}
\centering
\footnotesize
\begin{tikzpicture}
[ every path/.style = {shorten <=1pt, shorten >=1pt, >=stealth, absolute, draw},
  vertex/.style = {circle, minimum size=14pt, inner sep=1pt, draw}]
\matrix[row sep = 3ex, column sep = 3em]{
& \node [vertex] (zu) {}; & \node (dotstopz) {$\cdots$}; & \node [vertex] (vz) {}; & \\
& \node [vertex] (xu) {}; & \node (dotstopx) {$\cdots$}; & \node [vertex] (vx) {}; & \\
\node [vertex] (u) {$u_{1}$}; & \node [vertex] (ux) {}; & \node (dotsmiddle) {$\cdots$}; & \node [vertex] (xv) {};	& \node [vertex] (v) {$v_{1}$}; \\
& \node [vertex] (tu) {}; &  \node (dotsbottom) {$\cdots$}; & \node [vertex] (vt) {}; & \\
};
\graph [use existing nodes]{
u -> ux ->[edge node= {node [above, xshift = 2.5em] {$ 2^{k+1} - L - 1 $ edges} }]  dotsmiddle -> xv -> v -> {
	vt [> out = 240, > in = 0] -> dotsbottom [target edge node= {node [below, xshift = 2em, yshift=-1.5ex] {$ v_{1}|_{[2, L]} a_{k} $} }, source edge node= {node [below, xshift = 2em] {$  L -2^{k} +1 $ edges} }] -> tu [< out = 180, < in = 300],
	vx [> out = 120, > in = 0] -> dotstopx [target edge node= {node [below, xshift = 1.5em, yshift=-0.5ex] {$ v_{1}|_{[2, L]} \, a_{k+1} $} }] -> xu [< out = 180, < in = 60],	
	vz [> out =90, > in = 0] -> dotstopz [target edge node= {node [above, xshift = 2em, yshift=1.5ex] {$ v_{1}|_{[2, L]} \, a_{k-1} $} }]  -> [edge node= {node [below, xshift = 2.5em, yshift=-2ex] {$ L + 1 $ edges on each arc} }] zu [< out = 180, < in = 90]
} -> u;
};
\end{tikzpicture}
\normalsize
\caption{The de Bruijn graph for the Grigorchuk subshift for $ k \geq 2 $ for and $ 2^{k+1} - 2^{k-1} \leq L \leq 2^{k+1} - 1 $. We have $ r = L -2^{k}$. The number of edges of an arc refers to the distance between $u_{1}$ and $v_{1}$.}
\label{fig:deBruGrigoAllg2}
\end{figure}

\FloatBlock 

\subsection{Application: Palindrome Complexity}

A word is called a palindrome if it remains the same when read backwards, that is, if
\[ u = u(1) \hdots u(L) = u(L) \hdots u(1) =: \rev{u} \]
holds, where $\rev{u}$ denotes the reflection of the word $u$ at its midpoint. Note that $\rev{\bullet}$ defines an involution on the set of subwords of length $L$ of the subshift.

\begin{exmpl}
For every $ k \in \mathbb{N}_{0} $, the word $\block{k}$ is a palindrome: For $\block{0} = a_{0}^{n_{0}-1}$ this is obviously true and for $ k \geq 1 $ it follows by induction from the decomposition
\[ \rev{ \block{k+1} } = \rev{ \block{k} } \; a_{k+1} \; \rev{ \block{k} } \; \hdots \; \rev{ \block{k}} =  \block{k} \; a_{k+1} \; \block{k} \; \hdots \; \block{k} = \block{k+1} \, . \]
\end{exmpl}

Similar to the subword complexity, we define the palindrome complexity as
\[ \pali : \mathbb{N} \rightarrow \mathbb{N} \quad L \mapsto \card{ \{ u \in \langu{\subshift} : \length{u} = L \, \text{ and } \rev{u} = u \} } \, . \]
To investigate the palindrome complexity of simple Toeplitz subshifts, we will show that reflection symmetry of words corresponds to reflection symmetry in the de Bruijn graphs. Thus, Palindromes are precisely the vertices that lie on the reflection axis. Recall that $u_{1}$ resp. $v_{1}$ in $ \vertices{L} $ denote the prefix resp. suffix of length $L$ of $\block{k}$. Note that $u_{1} = v_{1}$ holds for $1 \leq L \leq \length{\block{0}}$.

\begin{prop}
Reflection at the midpoint of a word corresponds in the Figures~\ref{fig:deBruijn1ToL0}, \ref{fig:deBruijnAllg1} and \ref{fig:deBruijnAllg2} to reflection of the graph at a vertical axis through the middle of the graph.
\end{prop}

\begin{proof}
First we consider the arcs from $v_{1}$ to $u_{1}$, except for the arc from $v_{2}$ to $u_{2}$ in Figure~\ref{fig:deBruijnAllg2}, which will be considered later. The vertices on every arc from $v_{1}$ to $u_{1}$ are the subwords that occur in a word of the type $ \block{k} a \block{k} $, $ a \in \alphab_{k} $, between the suffix of the first $\block{k}$ and the prefix of the second $\block{k}$. The word $ \block{k} |_{[\length{ \block{ k }} - L +1+j, \length{ \block{ k }}]} \, a \, \block{k} |_{[1, j-1]} $ is the $j$-th vertex on such an arc, counted from $v_{1}$. Here, an empty interval as index means that no letter of this $\block{k}$ occurs. The reversed word is
\[ \rev{ \block{k} |_{[1, j-1]} }  a \rev{ \block{k} |_{[\length{ \block{ k }} - L +1+j, \length{ \block{ k }}]} } = \block{k} |_{[ \length{ \block{ k }} - j + 2, \length{ \block{ k }} ]} a \block{k} |_{[1, L-j]} \, . \]
This is the $j$-th vertex on the same arc when counted from $u_{1}$. Thus reversing the words corresponds to reflection of the path in the middle.

Now we deal with the path from $u_{1}$ to $v_{1}$. Let $ u := \block{k} |_{ [j, j+L-1] } $ denote a subword of length $L$ that is contained in $\block{k}$ and starts at the $j$-th letter from the left. Then the reflected word is $ \rev{u} = \block{k} |_{[ \length{\block{k}} - j - L + 2 , \length{\block{k}} - j + 1 ]}$, that is, $\rev{u}$ is the subword of $\block{k}$ that ends at the $j$-th letter from the right. The $j$-th vertex on the path from $u_{1}$ to $v_{1}$ is precisely the subword that starts at the $j$-th letter from the left. Hence, reversing the words corresponds to reflecting the path in the middle. 

Finally, we consider the arc from $v_{2}$ to $u_{2}$. Because of $ L \leq 2 \length{\block{k-1}} - \length{\block{k-2}}$ there is a copy of $u_{2}$ that begins at the start of the second $\block{k-2}$-block in $ \block{k-1} a_{k-1} \block{k-1}$, see Figure~\ref{fig:u2v2Location}. Moreover, we can find a copy of $v_{2}$ that begins in the first $\block{k-2}$-block in $ \block{k-1} a_{k-1} \block{k-1}$. The path from $v_{2}$ to $u_{2}$ corresponds to the words between them.

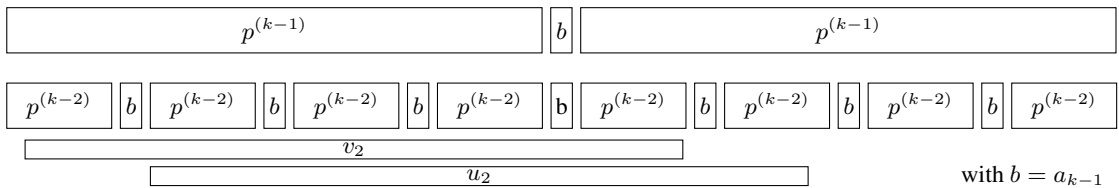
\begin{figure}
\centering
\footnotesize

\begin{tikzpicture}
[every node/.style ={rectangle, inner sep=0pt, minimum height=0.6cm},
pn/.style={minimum width=7.04cm, draw},
pn-1/.style={minimum width=3.53cm, draw},
pn-2/.style={minimum width=1.38cm, draw},
buchst/.style={minimum width=0.28cm, draw}]
%
\node [right] at (0,0) [pn] (A) {$ \block{ k-1  } $};
\node [buchst] (A) [right=0.1cm of A] {$b$};
\node [pn] (A) [right=0.1cm of A] {$ \block{ k-1 } $};
\node [right] at (0,-1) [pn-2] (Fix1) {$ \block{ k-2 }  $};
\node [buchst] (A) [right=0.1cm of Fix1] {$b$};
\node[pn-2] (A) [right=0.1cm of A] {$ \block{ k-2 }  $};
\node [buchst] (A) [right=0.1cm of A] {$b$};
\node[pn-2] (A) [right=0.1cm of A] {$ \block{ k-2 } $};
\node [buchst] (A) [right=0.1cm of A] {$b$};
\node[pn-2] (A) [right=0.1cm of A] {$ \block{k-2 }  $};
\node [buchst] (A) [right=0.1cm of A] {$$b};
\node[pn-2] (A) [right=0.1cm of A] {$ \block{ k-2 }  $};
\node [buchst] (Fix2) [right=0.1cm of A] {$b$};
\node[pn-2] (A) [right=0.1cm of Fix2] {$ \block{ k-2 } $};
\node [buchst] (A) [right=0.1cm of A] {$b$};
\node[pn-2] (A) [right=0.1cm of A] {$ \block{ k-2 }  $};
\node [buchst] (A) [right=0.1cm of A] {$b$};
\node[pn-2] (A) [right=0.1cm of A] {$ \block{ k-2 }  $};
\node [below left=0.2 of Fix2, minimum width=8.65cm, minimum height=0.25cm, draw]{$v_{2}$};
\node (U) [below right=0.7 of Fix1, minimum width=8.65cm, minimum height=0.25cm, draw]{$u_{2}$};
\node [right=2 of U]{with $b = a_{k-1}$};
\end{tikzpicture}

\normalsize
\caption{Location of $v_{2}$ and $u_{2}$ in $\block{ k-1 } a_{k-1} \block{ k-1 }$.}
\label{fig:u2v2Location}
\end{figure}

Recall the notation $ \widetilde{r} := L \mod ( \length{\block{ k-2 }} + 1 ) $. Moreover, let $x_{j}$ denote $j$-th vertex after $v_{2}$ on the path to $u_{2}$, where $j=0$ is $v_{2}$ and $j=\widetilde{r}+1$ is $u_{2} $. The already established symmetry of the arcs between $v_{1}$ and $u_{1}$ yields the equality $ \rev{x_{0}} = \rev{v_{2}} = u_{2} = x_{\widetilde{r}+1}$. For $j = 1 , \hdots , \widetilde{r}$, the $j$-th vertex after $v_{2}$ on the path to $u_{2}$ is the word
\begin{align*}
x_{j} &= \block{k-1}|_{[ \length{ \block{k-2} } + 1 - \widetilde{r} + j \,,\, \length{ \block{k-1} } ]} \; a_{k-1} \;  \block{k-1}|_{[ 1 \,,\, L - \length{\block{ k-1 }} + \length{\block{ k-2 }} - \widetilde{r} + j - 1 ]} \\
&= \Big[ \block{k-2}|_{[ \length{ \block{k-2} } + 1 - \widetilde{r} + j \,,\, \length{ \block{k-2} } ]} \; a_{k-1} \; \block{k-2} \hdots \block{k-2} \Big] \; a_{k-1} \\
& \quad \, \, \Big[ \block{k-2} \hdots \block{k-2} \; a_{k-1} \; \block{k-2}|_{[ 1 \,,\, j-1 ]} \Big] \, .
\end{align*}
Now we obtain the reflection symmetry of the arc from $v_{2}$ to $u_{2}$ from the fact that the $\block{k-2}$-blocks are palindromes: 
\begin{align*}
\rev{x_{j}} &= \rev{ \block{k-2}|_{[ 1 \,,\, j-1 ]} } \; a_{k-1} \; \block{k-2} \hdots \block{k-2} \; a_{k-1} \; \rev{ \block{k-2}|_{[ \length{ \block{k-2} } + 1 - \widetilde{r} + j \,,\, \length{ \block{k-2} } ]} } \\
&= \block{k-2}|_{[\length{ \block{k-2} } - j + 2 \,,\, \length{ \block{k-2} }]} \; a_{k-1} \; \block{k-2} \hdots \block{k-2} \; a_{k-1} \; \block{k-2}|_{[1 \,,\, \widetilde{r} - j ]} \\
&= x_{\widetilde{r} +1 - j} \qedhere
\end{align*}
\end{proof}

The symmetry of the graphs with respect to $\rev{\bullet}$ implies that the number of palindromes of length $L$ is the number of arcs in $\debruijn{L}$ with an even number of edges. The formula for the palindrome complexity follows now from the description of the graphs in the previous subsection (see also Figure~\ref{fig:deBruijn1ToL0}, \ref{fig:deBruijnAllg1} and \ref{fig:deBruijnAllg2}). As before, we denote $r := L \mod (\length{ \block{ k-1 }} + 1)$ and $ \widetilde{r} := L \mod ( \length{\block{ k-2 }} + 1 ) $.

\begin{coro}
For $ 1 \leq L \leq \length{\block{0}}$, the palindrome complexity is given by
\[ \pali(L) = ( \card{\alphab_{0}} - 1) \cdot (L \mod 2) + 1 \, .\]
For $ k \geq 1 $ and $\length{ \block{ k-1 }} + 1 \leq L \leq \length{ \block{ k }} $, the palindrome complexity is given by
\begin{align*}
\pali(L) = \, &(\card{ \alphab_{k} } - 1 ) \cdot ( L \mod 2 ) + ( \length{ \block{ k-1 } } + 1 - r \mod 2 ) \\
&+ ( r \mod 2 ) \cdot \begin{cases}
1 & \text{if } \length{ \block{ k-1 }} + 1 \leq L \leq \length{ \block{ k }} - \length{ \block{ k-1 }} - 1 \\
\mathbbm{1}_{\alphab_{k+1}}(a_{k}) & \text{if } \length{ \block{ k }} - \length{ \block{ k-1 }} \leq L \leq \length{ \block{ k }}
\end{cases} \\
&+ \begin{cases}
0 \\
( \widetilde{r} \mod 2) + ( \length{\block{k-2}} + 1 - \widetilde{r} \mod 2) - ( L \mod 2 ) \\
\end{cases} \\
& \quad \, \begin{cases}
\text{if } a_{k-1} \notin \alphab_{k} \text{ or } L > 2 \length{\block{k-1}} - \length{\block{k-2}} \\
\text{if } a_{k-1} \in \alphab_{k} \text{ and }  L \leq 2 \length{\block{k-1}} - \length{\block{k-2}}
\end{cases}
\end{align*}
\end{coro}

\begin{exmpl}
For a generalized Grigorchuk subshift and every $ k\geq 0 $, the length $ \length{ \block{k} } + 1 $ is a power of two and thus even. Therefore $ (r \mod 2) = (L \mod 2) = (\widetilde{r} \mod 2) $ holds and the palindrome complexity simplifies for a generalized Grigorchuk subshift to
\[ \pali(L) = ( \card{\alphab_{0}} - 1) \cdot (L \mod 2) + 1 \qquad \text{for } 1 \leq L \leq \length{\block{0}} \]
and 
\begin{align*}
\pali(L) =  ( L \mod 2 ) \cdot \Bigg( & \card{ \alphab_{k} } + \begin{cases}
1 & \text{if } \length{ \block{ k-1 }} + 1 \leq L \leq \length{ \block{ k }} - \length{ \block{ k-1 }} - 1 \\
\mathbbm{1}_{\alphab_{k+1}}(a_{k}) & \text{if } \length{ \block{ k }} - \length{ \block{ k-1 }} \leq L \leq \length{ \block{ k }}
\end{cases} \\
& + \begin{cases}
0 & \text{if } a_{k-1} \notin \alphab_{k} \text{ or } L > 2 \length{\block{k-1}} - \length{\block{k-2}} \\
1 & \text{if } a_{k-1} \in \alphab_{k} \text{ and }  L \leq 2 \length{\block{k-1}} - \length{\block{k-2}}
\end{cases} \Bigg)
\end{align*}
for $ k \geq 1 $ and $\length{ \block{ k-1 }} + 1 \leq L \leq \length{ \block{ k }} $.
\end{exmpl}

\begin{exmpl}
For the (standard) Grigorchuk subshift, $ \card{\alphab_{0}} = 4 $ and $a_{0} \notin \alphab_{1}$ hold. Moreover, we have $\length{\block{k}} + 1 = 2^{k+1}$ for all $ k \geq 0 $, as well as $ \card{\alphab_{k}} = 3 $ and $a_{k} \in \alphab_{k+1}$ for all $ k \geq 1$. This yields
\[ \pali(L) = \begin{cases}
4 \cdot ( L \mod 2 ) & \text{if } 1 \leq L \leq 3 \\
5 \cdot ( L \mod 2 ) & \text{if } 2^{k} \leq L \leq  2^{k+1} - 2^{k-1} - 1 \text{ for } k \geq 2\\
4 \cdot ( L \mod 2 ) & \text{if } 2^{k+1} - 2^{k-1} \leq L \leq 2^{k+1}  -1 \text{ for } k \geq 2
\end{cases} \, .\]
\end{exmpl}


\section{Repetitivity}
\label{sec:Repe}

In this section we continue our study of combinatorial properties of simple Toeplitz subshifts by investigating the repetitivity function. In the first part of the section, we introduce notation and two tools that describe combinatorial properties of the coding sequence $(a_{k})$. In the second part, we give an explicit formula for the repetitivity function of a simple Toeplitz subshift for all $ L \geq  \length{ \block{ m_{1} }} - \length{ \block{ m_{1}-1 }} + 1 $. In the third part, we characterize $\alpha$-repetitivity of the subshift as well as the special case of linear repetitivity (i.e. $\alpha = 1$).

\subsection{General Notions}

Recall that $\langu{\subshift}$ denotes the set of finite words that appear in elements of the subshift $\subshift$. The repetitivity function is defined as $\repe : \, \mathbb{N} \to \mathbb{N} ,\, L \mapsto \min \{ \widetilde{L} : $ every word of length $L$ in $\langu{\subshift}$ is contained in every word of length $\widetilde{L}$ in $ \langu{\subshift} \} $. It is easy to see that the repetitivity function is strictly monotonically increasing. In order to compute the repetitivity of $\subshift $, we need certain information about the structure of the coding sequence $(a_{k})$. The first one is expressed by the values of the function
\[ \alleB : \mathbb{N}_{0} \to \mathbb{N} \, , k \mapsto \min \{ j > k : \{ a_{k+1}, \hdots , a_{j} \} = \alphab_{k+1} \} \, . \]
This function describes at which point all letters that occur after $a_{k}$ have occurred at least once. Note that $ \alleB(k) \geq k + \# \alphab_{k+1} \geq k + \card{ \alphabEv }$ holds for all $ k \in \mathbb{N}_{0} $. 

\begin{prop}
\label{prop:AlleBMon}
The function $\alleB$ is monotonically increasing.
\end{prop}

\begin{proof}
We show that $ \alleB(k+1) \geq \alleB(k) $ holds for every $ k \in \mathbb{N}_{0} $. In the case of $ \alphab_{k+2} = \alphab_{k+1} $, this follows from the definition of $\alleB$, since
\[ \alleB(k+1) = \min\{ j : \{ a_{k+2}, \hdots , a_{j} \} = \alphab_{k+1} \} \geq \min\{ j : \{ a_{k+1}, \hdots , a_{j} \} = \alphab_{k+1} \} = \alleB(k) \]
holds. In the case of $\alphab_{k+2} \neq \alphab_{k+1}$, we have $ \alphab_{k+2} = \alphab_{k+1} \setminus \{ a_{k+1} \} $. Now the claim follows from the relations
\begin{align*}
\alleB(k+1) &= \min\{ j : \{ a_{k+2}, \hdots , a_{j} \} = \alphab_{k+1} \setminus \{ a_{k+1} \} \} \\
&= \min\{ j : \{ a_{k+1}, \hdots , a_{j} \} = \alphab_{k+1} \} \\
&= \alleB(k) \, . \qedhere
\end{align*}
\end{proof}

As can be seen in the above proof, $\alleB$ is in general monotonically, but not strictly monotonically, increasing. The second information we need about $(a_{k})$ to compute the repetitivity is at which positions the value of $\alleB$ actually changes. Therefore, we define the sequence $(m_{i})_{i \in \mathbb{N}_{0}}$ by $m_{0} = 0 $ and  $ \alleB(m_{i}) = \alleB( m_{i} + 1 ) = \hdots = \alleB( m_{i+1}-1 ) < \alleB( m_{i+1} ) $. It denotes those positions where $\alleB$ is growing. The following two propositions give alternative descriptions of $(m_{i})$ that will be helpful when computing the repetitivity.

\begin{prop}
\label{prop:MUndK}
The sequence $(m_{i})$ can be characterized by $m_{0} = 0$ and the recursion relation $m_{i+1} :=  \max \{ j \leq \alleB(m_{i}) : \{ a_{j} , a_{j+1} , \hdots , a_{\alleB(m_{i})} \} = \alphab_{m_{i}+1} \}$.
\end{prop}

\begin{proof}
Let $(m_{i})$ denote the previously defined sequence of points at which $\alleB$ is increasing and let $(\widetilde{m}_{i})$ denote the sequence defined by the recursion relation in this proposition. We have $ m_{0} = 0 = \widetilde{m}_{0} $ by definition and we proceed by induction, so assume that $ m_{i} = \widetilde{m}_{i} $ holds.

By definition of $\alleB$, the equality $ \{ a_{\widetilde{m}_{i}+1} , \hdots , a_{\alleB( \widetilde{m}_{i} )} \} = \alphab_{ \widetilde{m}_{i}+1 } $ holds, which implies $ \widetilde{m}_{i+1} \geq \widetilde{m}_{i} +1 $. From $ \{ a_{\widetilde{m}_{i+1}}, \hdots , a_{\alleB( \widetilde{m}_{i} )} \} = \alphab_{\widetilde{m}_{i} + 1} \supseteq \alphab_{\widetilde{m}_{i+1} } $, the inequality
\[ \alleB( \widetilde{m}_{i+1}-1 ) =  \min \{ j > \widetilde{m}_{i+1}-1 : \{ a_{ \widetilde{m}_{i+1} }, \hdots , a_{j} \} = \alphab_{ \widetilde{m}_{i+1} } \} \leq \alleB( \widetilde{m}_{i} ) = \alleB( m_{i} ) \]
follows, which yields $ \widetilde{m}_{i+1} - 1 \leq m_{i} \leq m_{i+1} -1$. On the other hand, we know from the definition of $ m_{i+1} $ that $ \alleB( m_{i+1}-1 ) = \alleB( m_{i} ) $ holds, which implies $ \{ a_{ m_{i+1} } , \hdots , a_{ \alleB( m_{i} ) } \} = \alphab_{ m_{i+1} } \subseteq \alphab_{ m_{i}+1 } $ and thus $ \widetilde{m}_{i+1} \geq m_{i+1} $.
\end{proof}

\begin{prop}
\label{prop:an=aKn}
For $ k \geq 1 $, the equality $a_{k} = a_{\alleB(k)}$ holds if and only if $ k = m_{i} $ for some $i \geq 1$.
\end{prop}

\begin{proof}
First assume $ k = m_{i} $. The definition of $\alleB( m_{i}-1 )$ yields $ \{ a_{m_{i}} , \hdots , a_{\alleB( m_{i}-1 )} \} = \alphab_{m_{i}} $. Since $\alleB$ is increasing at $m_{i}$, we obtain $\alleB(m_{i}-1) \leq \alleB(m_{i}) -1$ and thus $ \{ a_{m_{i}} , \hdots , a_{\alleB(m_{i}) -1} \} = \alphab_{m_{i}} $. Moreover, it follows from the definition of $ \alleB( m_{i} ) $ that
\[ \{ a_{m_{i}+1}, \hdots, a_{\alleB(m_{i})-1} \} = \alphab_{m_{i}+1} \setminus  \{ a_{\alleB(m_{i})} \} \subseteq \alphab_{m_{i}} \setminus  \{ a_{\alleB(m_{i})} \} \]
holds, which implies $ a_{m_{i}} = a_{\alleB( m_{i} )} $.

Now assume that $a_{k} = a_{\alleB(k)}$ holds for a certain $ k \geq 1 $. Since $a_{k}$ appears again at position $\alleB(k) \geq k+1$, we have $ \alphab_{k} = \alphab_{k+1} $. Using $ a_{k} = a_{\alleB(k)} $ and the definition of $\alleB(k)$ we obtain
\[ \{ a_{k} , \hdots , a_{ \alleB(k)-1 } \} = \{ a_{k+1} , \hdots , a_{ \alleB(k) } \} = \alphab_{k+1} = \alphab_{k} \]
and thus $ \alleB( k-1 ) \leq \alleB(k)-1 $. Therefore $\alleB$ is increasing at $k$ and there exists an $ i \geq 1$ such that $ k = m_{i} $ holds.
\end{proof}

\begin{exmpl}
Consider a simple Toeplitz subshift with $ \card{ \alphabEv } = 2$. Let $ k \geq \eventNr $, that is, let $k$ be large enough such that $ a_{k} \in \alphabEv $ and $ \alphab_{k} = \alphabEv $ hold. For these $k$, the sequence $ (a_{k}) $ will alternate between the two letters in $ \alphabEv $. This implies $ \alleB( k ) = k+2 $ for all $ k \geq \eventNr - 1$ and hence $m_{i+1} = m_{i}+1$.
\end{exmpl}

\begin{exmpl}
For a generalized Grigorchuk subshift, $ \card{ \alphabEv } $ is either equal to 2 or 3. The case $ \card{ \alphabEv } = 2 $ was discussed in the previous example. For $ \card{ \alphabEv } = 3 $, we obtain $ \alphabEv = \{ x, y, z \} $ and $ \eventNr = 1 $. Then $(m_{i})$ is given by $m_{0} = 0$ and $m_{i+1} = \alleB( m_{i} ) - 2 $ for $ i \geq 0 $, which can be seen as follows: For all $ i \geq 0 $ the number $ \alleB( m_{i} ) $ is minimal with $ \{ a_{m_{i}+1},  \hdots , a_{\alleB(m_{i})} \} = \alphab_{m_{i}+1} = \alphabEv $. This yields $ a_{\alleB(m_{i})} \notin \{ a_{m_{i}+1},  \hdots , a_{\alleB(m_{i}) - 1} \}$. Thus $ a_{\alleB(m_{i}) - 2}$, $a_{\alleB(m_{i}) - 1}$ and $ a_{\alleB(m_{i}) } $ are pairwise different and $ \{ a_{\alleB(m_{i}) - 2} , a_{\alleB(m_{i}) - 1} , a_{\alleB(m_{i}) } \} = \alphabEv $ holds. By Proposition~\ref{prop:MUndK} we obtain $ m_{i+1} = \alleB( m_{i}) - 2 $. Note that we have $ \alleB( k ) = k+3 $ for all $ k \geq 0 $ for the special case of the standard Grigorchuk subshift, and thus $ m_{i+1} = m_{i}+1 $. There is, however, no explicit formula for $ \alleB(k) $ for the generalized Grigorchuk subshift with $ \card{ \alphabEv } = 3 $, since arbitrary long blocks in the sequence $(a_{k})$ are possible in which one letter of $ \alphabEv $ is missing.
\end{exmpl}

\subsection{Computing the Repetitivity Function}

In this subsection we give an explicit formula for the repetitivity function of a simple Toeplitz subshift. We start by proving lower and upper bounds for the repetitivity function. Before we do so, recall from Proposition~\ref{prop:obereSchrComp} and Corollary~\ref{coro:CPnAndGrowth} that the words of length less or equal  $ \length{ \block{k} } + 1 $ are precisely the subwords of $ \block{k} a \block{k} $ for $ a \in \alphab_{k+1} $. Moreover all subwords of length $ \length{ \block{k} } + 1 $ that start in $ \block{k} a $ for $ a \in \alphab_{k+1} \setminus \{ a_{k} \} $ and all subwords that start in $ \block{k-1} a_{k} $ (provided that $a_{k} \in \alphab_{k+1}$ holds) are pairwise different.

\begin{prop}
\label{prop:M-M-1Geq}
The inequality 
\[ \repe( \length{ \block{m_{i}} }  - \length{ \block{m_{i}-1} } + 1 ) \geq 2 \cdot ( \length{ \block{ \alleB( m_{i} ) -1 }} + 1 ) \]
holds for all $i \geq 1$.
\end{prop}

\begin{proof}
By definition of $ \alleB $, the letter $ a_{\alleB( m_{i} )} $ is not in $ \{ a_{m_{i}+1}, \hdots , a_{\alleB(m_{i})-1} \} $. Because of $ a_{\alleB( m_{i} )} \in \alphab_{ m_{i}+1 }$, the word $ \block{m_{i}} a_{\alleB( m_{i} )} \block{m_{i}} $ occurs in the subshift. Let $v$ denote the suffix of length $ \length{ \block{m_{i}} }+1 - \length{ \block{ m_{i}-1 } } $ of $ \block{m_{i}} a_{\alleB( m_{i} )} $. Now decompose $ \block{m_{i}} $ into $\block{m_{i}-1}$-blocks:
\begin{align*}
\block{m_{i}} \;  a_{\alleB( m_{i} )} &= \Big[ \block{ m_{i} -1 } a_{ m_{i} } \block{ m_{i} -1 } \hdots \block{ m_{i} -1 } a_{ m_{i} } \block{ m_{i} -1 } \Big] \; a_{\alleB( m_{i} )} \\
&= \; \; \block{ m_{i} -1 } \; \; \underbrace{ a_{\alleB( m_{i} )} \; \;  \block{ m_{i} -1 } \hdots \block{ m_{i} -1 } \; \; a_{\alleB( m_{i} )} \; \; \block{ m_{i} -1 } \; \; a_{\alleB( m_{i} )} }_{\text{the suffix } v } \, ,
\end{align*}
where we used that $ a_{ \alleB( m_{i} ) } = a_{m_{i}} $ holds by Proposition~\ref{prop:an=aKn}. We conclude that $ n_{m_{i}} $ consecutive single letters in the $ \block{ m_{i}-1} $-block decomposition of $v$ are the letter $a_{\alleB( m_{i} )}$.

Now let $ b \in \alphab_{ \alleB(m_{i}) } \setminus \{ a_{\alleB( m_{i} ) } \}$ be another letter. The word $ u:= \block{ \alleB( m_{i} ) -1 } b \block{ \alleB( m_{i} ) -1 } $ occurs in the subshift since $b \in \alphab_{ \alleB(m_{i}) + 1 }$ holds. We can split $ \block{ \alleB( m_{i} ) -1 } $ into $ \block{ m_{i}} $-blocks and single letters from $  \{ a_{m_{i}+1}, \hdots , a_{\alleB(m_{i})-1} \} $. Since $a_{\alleB( m_{i} ) }$ is not among these letters, at most $n_{m_{i}}-1$ consecutive single letters in the $ \block{ m_{i}-1} $-block decomposition of $u$ are the letter $a_{\alleB( m_{i} ) }$. Thus $v$ is not contained in $u$, which yields
\[ \repe( \length{ \block{m_{i}} }+1 - \length{ \block{ m_{i}-1 } } ) = \repe( \length{v} ) > \length{u} = 2 \cdot \length{ \block{ \alleB( m_{i} ) -1 }} +1 \, . \qedhere \]
\end{proof}

\begin{prop}
\label{prop:RepeM+2Geq}
The inequality 
\[ \repe( \length{ \block{m_{i}} }+2 ) \geq 2 \cdot ( \length{ \block{ \alleB( m_{i} )-1 }} + 1 )+ \length{ \block{ m_{i} }} + 1 \]
holds for all $i \geq 1$.
\end{prop}

\begin{proof}
Similar to the previous proof, we construct a word $v$ of length $ \length{ \block{m_{i}} }+2 $ and a word  $u$ of length $ 2 \cdot ( \length{ \block{ \alleB( m_{i} )-1 }} + 1 )+ \length{ \block{ m_{i} }} $ which does not contain $v$. Let $ b \in \alphab_{ \alleB(m_{i})+1 } \setminus \{ a_{\alleB( m_{i} )} \}$. First note that $ u := \block{ \alleB( m_{i} ) -1 } b \block{ \alleB( m_{i} ) -1 } a_{\alleB( m_{i} )} \block{ m_{i} } $ occurs in $ \block{ \alleB(m_{i})  } b \block{ \alleB(m_{i}) } $ and that $ v := a_{\alleB( m_{i} )} \block{ m_{i} } a_{m_{i}+1} $ occurs in $\block{ \alleB( m_{i} ) } $.

We have seen in the previous proof that $a_{\alleB( m_{i} )}$ doesn't appear as a single letter in the $\block{ m_{i} }$-block decomposition of $ \block{ \alleB( m_{i} )-1 } $. Any two $\block{ m_{i} }$-blocks in $ \block{ \alleB(m_{i}) - 1 } b \block{ \alleB(m_{i}) - 1 } a_{\alleB( m_{i} )} \block{ m_{i} } $ are therefore separated by a letter that is not $a_{\alleB( m_{i} )}$, except for the next to last block and the last block. Hence, the word $v = a_{\alleB( m_{i} )} \block{ m_{i} } a_{m_{i}+1} $ does not occur in $ \block{ \alleB(m_{i}) - 1 } b \block{ \alleB(m_{i}) - 1 } $. Because of $ a_{\alleB( m_{i} )} = a_{m_{i}} \neq a_{m_{i}+1} $, it does not occur in the $ a_{m_{i}+1} \block{ m_{i} } a_{\alleB( m_{i} )} \block{ m_{i} } $-suffix of $u$ either, since subwords of length $\length{ \block{m_{i}} } + 1$ of $ \block{m_{i}} a \block{m_{i}} $ are pairwise different for different letters $a$. Thus, $v$ does not occur in $u$, which proves the claim.
\end{proof}

\begin{rem}
\label{rem:L=1M=M+1Geq}
Note that in the case of $m_{i+1} = m_{i}+1$ and $n_{m_{i+1}} = 2$ the Propositions~\ref{prop:M-M-1Geq} and \ref{prop:RepeM+2Geq} give a lower bound for the repetitivity function at the same length, since
\[ \length{ \block{ m_{i+1} }} - \length{ \block{ m_{i+1}-1 }} + 1 = n_{m_{i+1}} ( \length{ \block{m_{i}} } + 1 ) - \length{\block{m_{i}}} = \length{ \block{ m_{i} }}+2  \] 
holds. The lower bound given by Proposition~\ref{prop:M-M-1Geq} is stronger (that is, higher), as can be seen from the following direct computation:
\begin{align*}
2 \cdot ( \length{ \block{ \alleB( m_{i+1} )-1 }} + 1 ) 
& = 2 \cdot n_{\alleB( m_{i+1} )-1} \cdot ( \length{ \block{ \alleB( m_{i+1} )-2 }} + 1 ) \\
& \geq 2 \cdot n_{\alleB( m_{i+1} )-1} \cdot ( \length{ \block{ \alleB( m_{i} )-1 }} + 1 ) \\
& = 2 \cdot ( \length{ \block{ \alleB( m_{i} )-1 }} + 1 ) + 2 \cdot (n_{\alleB( m_{i+1} )-1} - 1) ( \length{ \block{ \alleB( m_{i} )-1 }} + 1 ) \\
& > 2 \cdot ( \length{ \block{ \alleB( m_{i} )-1 }} + 1 ) + \length{ \block{ m_{i} }} + 1 \, .
\end{align*}
The reason that the case $n_{m_{i}+1} = 2$ is special, is that the word $ a_{m_{i}+1} \block{m_{i}} a_{m_{i}+1} $ is not a subword of $\block{m_{i} +1}$ any more. Moreover, $m_{i+1} = m_{i}+1$ implies $ a_{m_{i}+1} = a_{m_{i+1}} = a_{\alleB( m_{i+1} ) } $ and this letter is lacking in $ \{ a_{m_{i}+2}, \hdots, a_{\alleB( m_{i+1} ) -1} \}$. Thus we now need to look at a much longer word to see $ a_{m_{i}+1} \block{m_{i}} a_{m_{i}+1} $ than when it appears as a subword in every $\block{m_{i} +1}$ block.
\end{rem}

\begin{prop}
\label{prop:RepeM-M-1Leq}
The inequality
\[ \repe( \length{ \block{ m_{i} }} - \length{ \block{ m_{i} - 1 }} ) \leq 2 \cdot ( \length{ \block{ \alleB( m_{i} - 1 ) - 1 }} + 1 ) + \length{ \block{ m_{i} }} - \length{ \block{ m_{i} - 1 }} - 1 \]
holds for all $i \geq 1$.
\end{prop}

\begin{proof}
To shorten notation, we introduce $ \widetilde{m} := \alleB(m_{i-1})-1 = \alleB( m_{i} - 1 ) - 1 $. We have to prove that every word of length $ 2 \cdot ( \length{ \block{ \widetilde{m} }} + 1 ) + \length{ \block{ m_{i} }} - \length{ \block{ m_{i} - 1}} - 1 $ contains all subwords of length $\length{ \block{ m_{i} }} - \length{ \block{ m_{i} - 1 }} $ of all words $ \block{ m_{i} } a \block{ m_{i} } $ with $ a \in \alphab_{m_{i}+1} $.

Clearly, every word of length $ 2 \cdot ( \length{ \block{ \widetilde{m} }} + 1 ) + \length{ \block{ m_{i} }} - \length{ \block{ m_{i} - 1 }} - 1 $ contains the block $ \block{ \widetilde{m} } $ at least once. For $ a = a_{m_{i}} $, it is sufficient to consider subwords of $ \block{ m_{i} } a \block{ m_{i} } $ that start in $ \block{ m_{i} - 1 } a$. Consequently, all subwords of $ \block{ m_{i} } a_{m_{i}} \block{ m_{i} } $ of length $ \length{\block{ m_{i} }} - \length{\block{ m_{i} - 1 }} $ are contained in $ \block{m_{i} }$, which is contained in $\block{\widetilde{m}}$. For all $ a \in \{  a_{m_{i}+1}, \hdots , a_{\widetilde{m}} \} $, the decomposition
\[ \block{ \widetilde{m} } = \block{ m_{i}  } a_{m_{i}+1} \hdots a_{m_{i}+1} \block{ m_{i} } a_{m_{i}+2} \block{ m_{i}  } a_{m_{i}+1} \hdots a_{m_{i}+1} \block{ m_{i} } \hdots a_{\widetilde{m}} \hdots \]
yields that the word $ \block{ m_{i} } a \block{ m_{i} } $ is contained in $\block{\widetilde{m}}$ as well. 

Because of $\alphab_{m_{i}+1} \subseteq \alphab_{m_{i}} = \{ a_{m_{i}}, \hdots, a_{\alleB(m_{i}-1)} \} $, the only remaining case is $a = a_{\alleB( m_{i}-1 )}$. For this case, we need to refine the above arguments. First we note that a word of length $ 2 \cdot ( \length{ \block{ \widetilde{m} }} + 1 ) + \length{ \block{ m_{i} }} - \length{ \block{ m_{i} - 1 }} - 1 $ contains a $ \block{ \widetilde{m} }$-block together with both neighbouring single letters. From the $ \block{ \widetilde{m} } $-block decomposition it is clear that at least one of the neighbouring single letters has to be $a_{\alleB( m_{i}-1 )}$. Let's assume it is the right letter (the case where it is the left one, can be treated similarly). To the right of $a_{\alleB( m_{i}-1 )}$, the next $ \block{ m_{i} } $ begins. We distinguish two cases:

First assume that there are at least $ \length{ \block{ m_{i} }} - \length{ \block{ m_{i} - 1 }} - 1 $ letters to the right of $a_{\alleB( m_{i}-1 )}$. In this case, all subwords of length $ \length{ \block{ m_{i} }} - \length{ \block{  m_{i} - 1 }} $ of $ \block{ m_{i} } a_{\alleB( m_{i}-1 )} \block{ m_{i} } $ which start in $ \block{ m_{i} } a_{\alleB( m_{i}-1 )}$ are contained in our word. Secondly, we have the case where there are $0 \leq j < \length{ \block{ m_{i} }} - \length{ \block{ m_{i} - 1 }} - 1$ letters to the right of $a_{\alleB( m_{i}-1 )}$, see Figure~\ref{fig:decompMtilde}. Then there are $ \length{ \block{ \widetilde{m} }} + 1 + \length{ \block{ \widetilde{m} }} + 1 + \length{ \block{ m_{i} }} - \length{ \block{ m_{i} - 1 }} - j - 2 $ letters to the left of $a_{\alleB( m_{i}-1 )}$, that is, a $ \block{ \widetilde{m} } $-block, a single letter, another $ \block{ \widetilde{m} } $-block, another single letter and the rightmost $\length{ \block{ m_{i} }} - \length{ \block{ m_{i} - 1 }} - 2 - j $ letters of $ \block{ m_{i} } $.

\begin{figure}
\centering
\footnotesize

\begin{tikzpicture}
\draw (-7.5, 0) -- (-6.5, 0) -- node[midway, left]{$\block{ m_{i} }$} (-6.5, 0.6) -- (-7.5, 0.6);
\draw (-6.4, 0) rectangle (-6, 0.6) node[midway]{$\star$};
\draw (-5.9, 0) rectangle (-3.3, 0.6) node[midway]{$\block{ \widetilde{m} }$};
\draw (-3.2, 0) rectangle (-2.8, 0.6) node[midway]{$\star$};
\draw (-2.7, 0) rectangle (-0.1, 0.6) node[midway]{$\block{ \widetilde{m} }$};
\draw (0, 0) rectangle (0.4, 0.6) node[midway]{$a$};
\draw (2, 0) -- (0.5, 0) -- node[midway, right]{$\block{ m_{i} }|_{[1, j]} \hspace{4em} \text{with } a = a_{\alleB( m_{i}-1 )}$} (0.5, 0.6) -- (2, 0.6);
\end{tikzpicture}

\normalsize
\caption{Decomposition of a word of length $ 2 \cdot ( \length{ \block{ \widetilde{m} }} + 1 ) + \length{ \block{ m_{i} }} - \length{ \block{ m_{i} - 1}} - 1 $.}
\label{fig:decompMtilde}
\end{figure}

Again, at least one of the single letters is $a_{\alleB( m_{i}-1 )}$. If it is the one in the middle, then our word contains $ \block{ m_{i} } a_{\alleB( m_{i}-1 )} \block{ m_{i} } $ and we are done. Otherwise, note that the right end of our word is $ \block{ m_{i} } a_{\alleB( m_{i}-1 )} \block{ m_{i} }|_{[1, j]} $. This contains all subwords of length $ \length{ \block{ m_{i} }} - \length{ \block{ m_{i} - 1 }} $ of $ \block{ m_{i} } a_{\alleB( m_{i}-1 )} \block{ m_{i} } $ that start in $ \block{ m_{i} }|_{[1, \length{ \block{ m_{i} - 1 }} +j+2]} $. In addition, the left end of our word is $ \block{ m_{i} }|_{[ \length{ \block{ m_{i} - 1 }} +j+3, \length{ \block{ m_{i} }} ]} \, a_{\alleB( m_{i}-1 )} \, \block{ m_{i} } $. This contains the remaining subwords of length $ \length{ \block{ m_{i} }} - \length{ \block{ m_{i} - 1 }} $ of $\block{ m_{i} } a_{\alleB( m_{i}-1 )} \block{ m_{i} }$.
\end{proof}

\begin{prop}
\label{prop:M+1Leq}
The inequality
\[ \repe( \length{ \block{ m_{i} }}+1 ) \leq 2 \cdot ( \length{ \block{ \alleB(m_{i})-1 }} + 1 )+ \length{ \block{ m_{i}-1 }} \]
holds for all $i \geq 1$.
\end{prop}

\begin{proof}
The proof is similar to the proof of Proposition~\ref{prop:RepeM-M-1Leq}. We write $\widetilde{m} = \alleB(m_{i})-1$ and show that every word of length $2 \cdot ( \length{ \block{ \widetilde{m} }} + 1 )+ \length{ \block{ m_{i}-1 }}$ contains all subwords of length $ \length{ \block{ m_{i} }} + 1 $ of all words of the form $ \block{ m_{i} } a \block{ m_{i} } $ with $ a \in \alphab_{m_{i}+1}$. First note that every word of length $2 \cdot ( \length{ \block{ \widetilde{m} }} + 1 )+ \length{ \block{ m_{i}-1 }}$ contains at least once the word $ \block{ \widetilde{m} } $. Because of the decomposition
\[ \block{ \widetilde{m} } = \block{ m_{i}  } a_{m_{i}+1} \hdots a_{m_{i}+1} \block{ m_{i} } a_{m_{i}+2} \block{ m_{i}  } a_{m_{i}+1} \hdots a_{m_{i}+1} \block{ m_{i} } \hdots a_{\widetilde{m}} \hdots \]
all words $ \block{ m_{i} } a \block{ m_{i} } $ with $ a \in \{ a_{m_{i}+1}, \hdots , a_{\alleB(m_{i})-1} \} = \alphab_{m_{i}+1} \setminus \{ a_{\alleB( m_{i} ) } \} $ are contained in $ \block{ \widetilde{m} } $. Hence only the case $ a = a_{\alleB( m_{i} ) } = a_{m_{i}} $ remains. To deal with this case, note that a word of length $2 \cdot ( \length{ \block{ \widetilde{m} }} + 1 )+ \length{ \block{ m_{i}-1 }}$ actually contains a complete $ \block{ \widetilde{m} } $-block together with both neighbouring letters. At least one of these letters is $a_{\widetilde{m}+1} = a_{m_{i}}$ and we will assume that it is the right one (the other case can be treated similarly).

Recall that it suffices to consider those subwords of length $\length{ \block{ m_{i} }}+1$ of $ \block{ m_{i} } a_{m_{i}} \block{ m_{i} }$ that start in $ \block{ m_{i}-1 } a_{m_{i}}$. If there are at least $\length{ \block{ m_{i} - 1 }}$ letters right of $a_{m_{i}}$, then these subwords are all contained in our word. Now assume that there are only $0 \leq j < \length{ \block{ m_{i} - 1 }}$ letters right of $a_{m_{i}}$. Then there are $ \length{ \block{ m_{i}-1 }} - 1 - j +1+ \length{ \block{ \tilde{m} }} + 1+ \length{ \block{ \widetilde{m} }} $ letters left of $a_{m_{i}}$, that is, a $\block{ \widetilde{m} }$-block, a single letter, another $\block{ \widetilde{m} }$-block, another single letter and the rightmost $\length{ \block{ m_{i}-1 }} - 1 - j $ letters of $\block{ m_{i}-1 }$. Again, at least one of the single letters is $a_{m_{i}}$. If it is the one in the middle, then our word contains $ \block{ m_{i} } a_{m_{i}} \block{ m_{i} } $ and we are done. Otherwise, note that the right end of our word is $ \block{ m_{i} } a_{m_{i}} \block{ m_{i}-1 }|_{[1, j]} $. This contains all subwords of length $\length{ \block{ m_{i} }}+1$ of $\block{ m_{i} } a_{m_{i}} \block{ m_{i} }$ that start in $ \block{ m_{i}-1 }|_{[1, 1+j]} $. The left end of our word is $ \block{ m_{i}-1 }|_{[2+j, \length{ \block{ m_{i}-1 }}]} a_{m_{i}} \block{ m_{i} } $. This contains the remaining subwords of length $\length{ \block{ m_{i} }}+1$ of $\block{ m_{i} } a_{m_{i}} \block{ m_{i} }$.
\end{proof}

\begin{rem}
Similar to what we observed for the lower bounds, the Propositions~\ref{prop:RepeM-M-1Leq} and \ref{prop:M+1Leq} refer to the same length in the case of $ m_{i+1} = m_{i}+1 $ and $ n_{m_{i+1}} = 2 $, since in that case
\[  \length{ \block{ m_{i+1} }} - \length{ \block{ m_{i+1}-1 }} = \length{ \block{ m_{i+1}-1 }} + 1  = \length{ \block{ m_{i} }}+1 \] 
holds. The upper bound given by Proposition~\ref{prop:M+1Leq} is stronger (that is, lower), as a direct computation shows:
\begin{align*}
2 \cdot ( \length{ \block{ \alleB(m_{i})-1 }} + 1 )+ \length{ \block{ m_{i}-1 }} 
&< 2 \cdot ( \length{ \block{ \alleB(m_{i})-1 }} + 1 )+ \length{ \block{ m_{i} }} \\
& = 2 \cdot ( \length{ \block{ \alleB(m_{i+1}-1)-1 }} + 1 )+ \length{ \block{ m_{i+1} }} - \length{ \block{ m_{i+1}-1 }} - 1 \, .
\end{align*}
\end{rem}

\begin{thm}
\label{thm:RepeFunct}
For $i \geq 1$ and $ \length{ \block{ m_{i} }} - \length{ \block{ m_{i}-1 }} + 1 \leq L \leq \length{ \block{ m_{i+1} }} - \length{ \block{ m_{i+1}-1 }} $, the repetitivity function is given by
\begin{align*} \repe(L)  = &\begin{cases}
2  \length{ \block{ \alleB( m_{i} ) -1 }} + 1 - \length{ \block{ m_{i} }} + \length{ \block{ m_{i}-1 }} +  L \\
 2  \length{ \block{ \alleB( m_{i} )-1 }} + 1 + L
\end{cases} \\
&\begin{cases}
\text{for } \length{ \block{ m_{i} }} - \length{ \block{ m_{i}-1 }} + 1 \leq L \leq \length{ \block{ m_{i} }} + 1 \\
\text{for } \length{ \block{ m_{i} }} + 2 \leq L \leq \length{ \block{ m_{i+1} }} - \length{ \block{ m_{i+1}-1 }} 
\end{cases} \hspace{-0.5 em} .
\end{align*}
\end{thm}

\begin{proof}
First we show that the lower bound in Proposition~\ref{prop:M-M-1Geq} and the upper bound in Proposition~\ref{prop:M+1Leq} are actually equalities and that the repetitivity function increases by exactly one in between. We use that the repetitivity function is strictly increasing and its growth is therefore always at least one:
\begin{align*}
& \; 2 \cdot ( \length{ \block{ \alleB(m_{i})-1 }} + 1 ) + \length{ \block{ m_{i}-1 }} \\
& \geq \repe( \length{ \block{ m_{i} }}+1 ) \qquad \text{by Proposition~\ref{prop:M+1Leq}} \\
& = \repe(\length{ \block{ m_{i} }} - \length{ \block{ m_{i}-1 }}+1) + \sum_{L = \length{ \block{ m_{i} }} - \length{ \block{ m_{i}-1 }}+1 }^{ \length{ \block{ m_{i} }} } \big[ \repe(L+1) - \repe(L) \big] \\
& \geq 2 \cdot ( \length{ \block{ \alleB( m_{i} ) -1 }} + 1 ) + 1 \cdot \length{ \block{ m_{i}-1 }} \qquad \text{by Proposition~\ref{prop:M-M-1Geq}} .
\end{align*}
This yields 
\[ \repe(L) = 2 \cdot ( \length{ \block{ \alleB( m_{i} ) -1 }} + 1 ) + L - \big( \length{ \block{ m_{i} }} - \length{ \block{ m_{i}-1 }} +1 \big) \]
for $L$ from $ \length{ \block{ m_{i} }} - \length{ \block{ m_{i}-1 }} + 1 $ to $ \length{ \block{ m_{i} }} + 1 $. If $ m_{i+1} = m_{i}+1 $ and $n_{m_{i+1}} = 2$ hold, then $ \length{ \block{ m_{i} }} + 1 = \length{ \block{ m_{i+1}-1 }} + 1 = \length{ \block{ m_{i+1} }} - \length{ \block{ m_{i+1}-1 }} $ follows and we are done. If either $ m_{i+1} > m_{i}+1 $ or $ n_{m_{i+1}} > 2 $ holds, then $ \length{ \block{ m_{i} }} + 1 < \length{ \block{ m_{i+1} }} - \length{ \block{ m_{i+1}-1 }} $ follows and we have yet to consider the lengths $L$ from $ \length{ \block{m_{i} }} + 2 $ to $ \length{ \block{ m_{i+1} }} - \length{ \block{ m_{i+1}-1 }} $. For this we show that the inequalities in the Propositions~\ref{prop:RepeM+2Geq} and \ref{prop:RepeM-M-1Leq} are actually equalities and that the repetitivity function increases by exactly one in between:
\begin{align*}
& \; 2 \cdot ( \length{ \block{ \alleB(m_{i+1}-1) - 1 }} + 1 ) + \length{ \block{ m_{i+1} }} - \length{ \block{ m_{i+1} - 1 }} - 1 \\
& \geq \repe( \length{ \block{ m_{i+1} }} - \length{ \block{ m_{i+1} - 1 }} ) \qquad \text{by Proposition~\ref{prop:RepeM-M-1Leq}} \\
& = \repe( \length{ \block{ m_{i} }}+2 ) + \sum_{L = \length{ \block{ m_{i} }}+2 }^{ \length{ \block{ m_{i+1} }} - \length{ \block{ m_{i+1} - 1 }} - 1 } \big[ \repe(L+1) - \repe(L) \big] \\
& \geq 2 \cdot ( \length{ \block{ \alleB( m_{i} )-1 }} + 1 )+ \length{ \block{ m_{i} }} + 1 + \length{ \block{ m_{i+1} }} - \length{ \block{ m_{i+1} - 1 }} - \length{ \block{ m_{i} }} - 2 \\
& \hspace{1.4em} \text{by Proposition~\ref{prop:RepeM+2Geq}} \\
& = 2 \cdot ( \length{ \block{ \alleB( m_{i} )-1 }} + 1 )+ \length{ \block{ m_{i+1} }} - \length{ \block{ m_{i+1} - 1 }} - 1 \\
& = 2 \cdot ( \length{ \block{ \alleB( m_{i+1}-1 )-1 }} + 1 )+ \length{ \block{ m_{i+1} }} - \length{ \block{ m_{i+1} - 1 }} - 1 \qedhere
\end{align*}
\end{proof}

\begin{rem}
For $ m_{i+1} = m_{i}+1 $ and $ n_{m_{i+1}} = 2 $, the repetitivity function simplifies to
\[ \repe(L) = 2 \cdot \length{ \block{ \alleB( m_{i} ) -1 }} + 1 - \length{ \block{ m_{i} }} + \length{ \block{ m_{i}-1 }} +  L \]
for $ \length{ \block{ m_{i} }} - \length{ \block{ m_{i}-1 }} + 1 \leq L \leq \length{ \block{ m_{i}+1 }} - \length{ \block{ m_{i} }} $ .
\end{rem}

\begin{rem}
Roughly speaking, the jump between $ \repe( \length{ \block{  m_{i} }}+1 )$ and $ \repe( \length{ \block{  m_{i} }}+2 ) $ is caused by the fact that it is sufficient to consider the subwords of $ \block{ m_{i} } a \block{ m_{i} } $ with $a \in \alphab_{m_{i}+1}$ when we are interested in all words of length $ \length{ \block{ m_{i} }} + 1$, but this is not true when we wish to deal with words of length $ \length{ \block{ m_{i} }}+2 $. Here we have to consider words of the type $ a_{1} \block{ m_{i} } a_{2} $ for $ a_{1}, a_{2} \in \alphab_{m_{i}+1} $ as well, and to see all possibilities, we have to look at a much longer word. The jump between $\repe( \length{ \block{ m_{i} }} - \length{ \block{ m_{i} - 1 }} ) $ and $ \repe( \length{ \block{  m_{i} }} - \length{ \block{  m_{i}-1 }} + 1 ) $ is caused by a similar reason: All subwords of length $\length{ \block{ m_{i} }} - \length{ \block{ m_{i} - 1 }} $ of $ \block{ m_{i} } a_{m_{i}} \block{ m_{i} } $ are contained in $ \block{ m_{i} }$. This is not true for subwords of length $ \length{ \block{ m_{i} }} - \length{ \block{ m_{i} - 1 }}+1$, where we have to look at a much longer word to see the next occurrence of $a_{m_{i}}$. In the special case of $ n_{m_{i+1}} = 2 $ and $ m_{i+1}= m_{i}+1 $, the positions for the two jumps coincide.
\end{rem}

\subsection{Application: \texorpdfstring{$\alpha$}{Alpha}-Repetitivity}

In this subsection, we will use the above formula for the repetitivity function to investigate linear repetitivity and, more general, $\alpha$-repetitivity of simple Toeplitz subshifts. Since the repetitivity is strictly increasing and clearly $ 1 < \card{\alphab} \leq \repe(1) $ holds, we obtain $1 < \frac{\repe(L)}{L}$ for all $L \geq 1$. If there exists a constant $C$ such that $\frac{\repe( L )}{L} \leq C $ holds for all $L \geq 1$, then the subshift is called linear repetitive. This was generalized in \cite{GKMSS_AperioJarnik}, Definition 2.9, where a subshift is called $\alpha$-repetitive for $ \alpha \geq 1 $, if $ 0 < \limsup_{L \to \infty} \frac{\repe(L)}{L^{\alpha}} < \infty $ holds. In \cite{DKMSS_Regul-Article}, Theorem~4.10, $\alpha$-repetitivity is characterized for $l$-Grigorchuk subshifts. Below we give a characterisation for simple Toeplitz subshifts. For $m_{i} = i$, $ \alleB( m_{i} ) = m_{i} + 3 $ and $ n_{j} = 2^{l_{j}} $, the following Proposition yields precisely the result for $l$-Grigorchuk subshifts from \cite{DKMSS_Regul-Article}.

\begin{prop}
\label{prop:CharAlphaRepeAllg}
Let $\alpha \geq 1$. A simple Toeplitz subshift is $\alpha$-repetitive if and only if the inequalities
\[ 0 < \limsup_{i \to \infty}  \frac{ \length{ \block{ \alleB( m_{i} )-1 }} + 1 }{( \length{ \block{ m_{i} }}+1 )^{\alpha}} = \limsup_{i \to \infty}  \frac{   n_{0} \cdot \hdots \cdot n_{\alleB( m_{i} )-1}  }{  n_{0}^{\alpha} \cdot \hdots \cdot n_{m_{i}}^{\alpha}  }  < \infty \]
hold.
\end{prop}

\begin{proof}
We use the result from Theorem~\ref{thm:RepeFunct}. The quotient takes the form $ \frac{\repe( L )}{L^{\alpha}} = \frac{\text{const}}{L^{\alpha}} + L^{1-\alpha} $, where the constant depends on whether $ \length{ \block{ m_{i} }} - \length{ \block{ m_{i}-1 }} + 1 \leq L \leq \length{ \block{ m_{i} }} + 1 $ or $ \length{ \block{ m_{i} }} + 2 \leq L \leq \length{ \block{ m_{i+1} }} - \length{ \block{ m_{i+1}-1 }} $ holds and on the value of $i$. For every $i$, the quotient is maximal either at $ L_{1}(i) := \length{ \block{m_{i} }} - \length{ \block{ m_{i}-1 }} + 1 $ or at $ L_{2}(i) :=  \length{ \block{ m_{i} }} + 2 $. At these points, the upper and lower bounds
\[ 2 \cdot \frac{ \length{ \block{ \alleB( m_{i} ) -1 }} + 1 }{ ( \length{ \block{ m_{i} }} + 1)^{\alpha} } < \frac{\repe( L_{1}(i) )}{ L_{1}(i)^{\alpha}} <  \frac{2}{(1 - \frac{ 1  }{  n_{m_{i}}  }) ^{\alpha}} \cdot \frac{ \length{ \block{ \alleB( m_{i} ) -1 }} + 1}{ ( \length{ \block{ m_{i} }} + 1)^{\alpha}  } \leq 2^{1+\alpha} \cdot \frac{ \length{ \block{ \alleB( m_{i} ) -1 }} + 1 }{ ( \length{ \block{ m_{i} }} + 1)^{\alpha}  } \]
and
\[ 2^{1-\alpha} \cdot \frac{ \length{ \block{ \alleB( m_{i} )-1 }} + 1 }{ ( \length{ \block{ m_{i} }} + 1 )^{\alpha} } = \frac{ 2 \cdot ( \length{ \block{ \alleB( m_{i} )-1 }} + 1 ) }{ ( 2 \cdot (\length{ \block{ m_{i} }} + 1) )^{\alpha} } < \frac{ \repe( L_{2}(i) ) }{L_{2}(i)^{\alpha}} < 3 \cdot \frac{ \length{ \block{ \alleB( m_{i} )-1 }} + 1 }{ ( \length{ \block{ m_{i} }} + 1 )^{\alpha} } \]
hold. If we assume that $ 0 < \limsup_{i \to \infty}  \frac{ \length{ \block{ \alleB( m_{i} )-1 }} + 1 }{( \length{ \block{ m_{i} }}+1 )^{\alpha}} < \infty$ holds, then the above bounds yield
\[ 0 < \limsup_{i \to \infty} 2 \cdot \frac{ \length{ \block{ \alleB( m_{i} )-1 }} + 1 }{( \length{ \block{ m_{i} }}+1 )^{\alpha}} \leq \limsup_{L \to \infty} \frac{\repe( L )}{L^{\alpha}} \leq \limsup_{i \to \infty} 2^{1+\alpha} \cdot  \frac{ \length{ \block{ \alleB( m_{i} )-1 }} + 1}{ ( \length{ \block{ m_{i} }} + 1)^{\alpha}  } < \infty \, .\]
and thus, the subshift is $\alpha$-repetitive. Conversely, if we assume that the subshift is $\alpha$-repetitive, that is, $ 0 < \limsup_{L \to \infty} \frac{\repe( L )}{L^{\alpha}} < \infty $ holds, then the above bounds yield
\[  0 < 2^{-1-\alpha} \cdot \limsup_{L \to \infty} \frac{\repe(L)}{L^{\alpha}} < \limsup_{i \to \infty} \frac{ \length{ \block{ \alleB( m_{i} ) -1 }} + 1 }{ ( \length{ \block{ m_{i} }} + 1)^{\alpha} } < 2^{-1+\alpha} \cdot \limsup_{L \to \infty} \frac{\repe(L)}{L^{\alpha}} < \infty \, . \qedhere \]
\end{proof}

\begin{coro}
\label{coro:CharLinRepe}
A simple Toeplitz subshift is linear repetitive if and only if $ \left( \prod_{j = m_{i}+1}^{\alleB(m_{i}) - 1} n_{j} \right)_{i \geq 1} $ is bounded from above.
\end{coro}

\begin{proof}
By Proposition~\ref{prop:CharAlphaRepeAllg}, the subshift is linear repetitive ($\alpha = 1$) if and only if
\[ 0 < \limsup_{i \to \infty} \prod_{j = m_{i}+1}^{\alleB(m_{i})-1} n_{j} < \infty \]
holds. Because of $ n_{j} \geq 2 $ for all $ j \geq 0 $, and $ \alleB(k) \geq k + \# \alphab_{k+1} \geq k + 2 $ for all $ k \geq 0 $, the product is bounded from below by 2 for all $ i \geq 0 $.
\end{proof}

\begin{rem}
The product $\prod_{j = m_{i}+1}^{\alleB(m_{i})-1} n_{j} $ gives the length of the period of the word $ (a_{m_{i}+1}^{n_{m_{i}+1}-1} ? )^{\infty} \triangleleft \hdots \triangleleft (a_{\alleB(m_{i})-1}^{n_{\alleB(m_{i})-1}-1} ? )^{\infty} $. Thus a simple Toeplitz subshift is linear repetitive if and only if the sequence of the period lengths of the words $(a_{m_{i}+1}^{n_{m_{i}+1}-1} ? )^{\infty} \triangleleft \hdots \triangleleft (a_{\alleB(m_{i})-1}^{n_{\alleB(m_{i})-1}-1} ? )^{\infty} $ is bounded. 
\end{rem}

\begin{coro}
On the one hand, if the sequence $(n_{k})_{k \geq 0}$ is bounded, then the subshift is linear repetitive if and only if the sequence $( \alleB( m_{i} ) - m_{i})_{i \geq 0}$ is bounded. On the other hand, if the sequence $( \alleB( m_{i} ) - m_{i})_{i \geq 0}$ is bounded, then the subshift is linear repetitive if and only if the sequence $(n_{k})_{k \geq 0}$ is bounded. 
\end{coro}

\begin{rem}
The difference $ \alleB(k) - k $ describes how many positions in the coding sequence, starting in $a_{k+1}$, we have to look at to see all letters that can occur from this point on. Taking the difference $ \alleB(k) - k $ at the points $ k = m_{i} $ ensures that we get the largest possible distances, since $(m_{i})$ denotes the positions where the value of $\alleB$ increases.
\end{rem}
 
\begin{coro}
\label{coro:CharAlphaRepeConstL}
If $ (n_{k})_{k \geq 0} $ is a constant sequence, then the subshift is $\alpha$-repetitive if and only if $ -\infty < \limsup_{i \to \infty}\big[ \alleB(m_{i}) - \alpha \cdot m_{i} \big] < \infty $ holds.
\end{coro}

\begin{proof}
Let $ n $ be the constant value of the sequence $ (n_{k})_{k} $. By Proposition~\ref{prop:CharAlphaRepeAllg}, a subshift is $\alpha$-repetitive if and only if the following holds:
\begin{align*}
& \; 0 < \limsup_{i \to \infty} \frac{  n_{0} \cdot \hdots \cdot n_{\alleB( m_{i} )-1}  }{   n_{0}^{\alpha} \cdot \hdots \cdot n_{m_{i}}^{\alpha}  } < \infty \\
\Longleftrightarrow \quad & \; 0 < \limsup_{i \to \infty} \, n^{\alleB( m_{i} ) - \alpha \cdot (m_{i} + 1)} < \infty \\
\Longleftrightarrow \quad & \; - \infty  < \limsup_{i \to \infty} \big[ \alleB( m_{i} ) - \alpha \cdot (m_{i} + 1) \big] < \infty \, .\qedhere
\end{align*}
\end{proof}

\begin{exmpl}
For the Grigorchuk subshift, $ (n_{k})_{k} $ is the constant sequence with value 2. Moreover $m_{i} = i$ and $\alleB(k) = k+3$ hold for all $i, k \geq 0$. Hence the Grigorchuk subshift is $\alpha$-repetitive if and only if $-\infty < \limsup_{i \to \infty}\big[ (1- \alpha) i + 3 \big] < \infty $ holds. Thus, it is $1$-repetitive, that is, linear repetitive.
\end{exmpl}

\begin{coro}
Let $(\alleB(m_{i}) - m_{i})_{i \geq 0} $ be a constant sequence with value $c$. If $ n_{j + c - 1}  = n_{j}^{\alpha} $ holds for all $j \geq 0$, then the subshift is $\alpha$-repetitive. In particular, the subshift is $\alpha$-repetitive if $ \alleB(m_{i}) - m_{i} = c $ is constant and $ n_{j+1} = n_{j}^{\sqrt[c-1]{\alpha}} $ holds for all $j \geq 0$.
\end{coro}

\begin{proof}
Since $\alleB(k) \geq k+2 $ holds for all $k \geq 0$, we have $c \geq 2$. The product
\[ \frac{  n_{0} \cdot \hdots \cdot n_{\alleB( m_{i} )-1}  }{  n_{0}^{\alpha} \cdot \hdots \cdot n_{m_{i}}^{\alpha} } = \frac{ \prod_{j = 0}^{ c - 2} n_{j} \cdot \prod_{j = c-1}^{m_{i}+c-1} n_{j} }{  \prod_{j = 0}^{m_{i}} n_{j}^{\alpha} } = \frac{  \prod_{j = 0}^{ c - 2} n_{j} \cdot \prod_{j = 0}^{m_{i}} n_{j+c-1} }{ \prod_{j = 0}^{m_{i}} n_{j}^{\alpha} } = \prod_{j = 0}^{ c - 2} n_{j} \]
is positive, finite and independent of $i$. Now Proposition~\ref{prop:CharAlphaRepeAllg} yields the claim.
\end{proof}


\section{The Boshernitzan Condition and Jacobi Cocycles}
\label{sec:BoshJacobi}

In this section the Boshernitzan condition is discussed. It can be thought of as a weaker analogue of linear repetitivity and was characterized for simple Toeplitz subshifts in \cite{LiuQu_Simple}. Based on another result from \cite{LiuQu_Simple}, we give a different characterization, which describes the Boshernitzan condition in terms of the function $\alleB$ and the sequence of period lengths $(n_{k})_{k}$. As corollaries, we obtain particular simple descriptions of the Boshernitzan condition for generalized Grigorchuk subshifts and, more general, for simple Toeplitz subshifts with either $n_{k}=2^{j_{k}}$ or $\card{ \alphabEv } = 3$. 

As an application, a result from \cite{BeckPogo_SpectrJacobi} shows that the Boshernitzan condition implies Cantor spectrum of Lebesgue measure zero for Jacobi operators on this subclass of subshifts. This is briefly discussed in the second subsection. It serves mostly as a reminder about the definition of these operators and their associated cocycles. The implications of the characterization in Subsection~\ref{subsec:BoshCond} are made explicit. Moreover we recall a result from \cite{LiuQu_Simple} for the special case of Schrödinger operators on simple Toeplitz subshifts, where the spectrum is always a Cantor set of Lebesgue measure zero, independent of the Boshernitzan condition.

\subsection{The Boshernitzan-Condition}
\label{subsec:BoshCond}

Recall from Section~\ref{subsec:SimpleTWords} that $ \cylin{u}{j} $ denotes the cylinder set of all elements in which a finite word $u$ occurs at position $j$. The set of all finite subwords of a subshift $\subshift$ is denoted by $\langu{ \subshift }$. As we have seen in Section~\ref{subsec:SimToepSub}, simple Toeplitz subshifts are uniquely ergodic due to their regularity. Let now $\nu$ denote the unique $\Shift$-invariant ergodic probability measure on $\subshift$ and define
\[ \eta( L ) := \min \{ \nu( \cylin{u}{1} ) \, : u \in \langu{ \subshift } \, ,\; \length{u} = L \}\, . \]
A subshift is said to satisfy the Boshernitzan condition, if
\begin{equation*}
\limsup_{L \to \infty} L \cdot \eta( L )  > 0 \tag{B}
\end{equation*}
holds. In \cite{LiuQu_Simple}, a number of results related to the Boshernitzan condition are proven: It is shown that every simple Toeplitz subshift satisfies (B) if $ \card{ \alphabEv } = 2 $ holds (Proposition~4.1). Moreover, for $ \card{ \alphabEv } \geq 3 $ a description of $\eta( L )$ is given (Proposition 4.2), which is used to characterize (B) for simple Toeplitz subshifts (Corollary~4.1).

Here, a different characterization of (B) in terms of the function $\alleB$ and the period lengths $(n_{k})_{k}$ is provided. Since the proof is based on the mentioned description of $\eta( L )$ from \cite{LiuQu_Simple}, it is stated below in our notation. As in Section~\ref{sec:Repe}, let $\alleB( k ) = \min \{ j > k: \{ a_{k+1}, \hdots , a_{j} \} = \alphab_{k+1} \}$ and recall that $ \eventNr $ denotes a number such that $ a_{k} \in \alphabEv $ and $ \alphab_{k} = \alphabEv $ hold for all $ k \geq \eventNr $. Moreover, we define $ s_{j} := n_{0} \cdot \hdots \cdot n_{j-1} $.

\begin{prop}[\cite{LiuQu_Simple}]
\label{prop:Eta}
For a simple Toeplitz subshift with $ \card{ \alphabEv } \geq 3 $, there exist constants $ 0 < c_{1} \leq c_{2} $ such that for every $ L >  s_{\eventNr} $ and $j$ defined by the property $ s_{j-1} < L \leq s_{j} $, the following holds:

If $ s_{j-1} < L < 2s_{j-1} $, then $ c_{1} \cdot \eta( L ) \leq \min \Big\{ \frac{ \Big\lceil \frac{ 2s_{ j-1 } - L }{ s_{ j-2 } } \Big\rceil }{ s_{ \min \{ i > j-1 \, : \, a_{i} = a_{j-2} \} } } \; , \; \frac{1}{ s_{ \alleB( j-3 ) } } \Big\} \leq c_{2} \cdot \eta( L ) $ .

If $ 2s_{j-1} \leq L \leq s_{j} $, then $ c_{1} \cdot \eta( L ) \leq \frac{1}{ s_{\alleB( j-2 )} }  \leq c_{2} \cdot \eta( L ) $ .
\end{prop}

Using the above description of $\eta( L )$ from \cite{LiuQu_Simple}, we will now prove the following:

\begin{prop}
\label{prop:BoundedProdBosh}
A simple Toeplitz subshift satisfies (B) if and only if there exists a sequence $(k_{r})_{r}$ of natural numbers with $ \lim_{r \to \infty} k_{r} = \infty $ such that $ \prod_{j = k_{r}+1}^{\alleB(k_{r}-1)-1} n_{j} $ is bounded.
\end{prop}

\begin{proof}
For $\card{\alphabEv}= 2$, the claimed equivalence is true for trivial reasons: For all $r$ such that $ k_{r} > \eventNr$ holds, we have $ \alleB( k_{r}-1 ) = k_{r}+1 $. Hence $ \prod_{j = k_{r}+1}^{\alleB(k_{r}-1)-1} n_{j} = 1$ is the empty product and therefore bounded. In addition (B) is always satisfied according to Proposition~4.1 in \cite{LiuQu_Simple}.

For $\card{\alphabEv} \geq 3$ we first prove that boundedness of the product $ \prod_{j = k_{r}+1}^{\alleB(k_{r}-1)-1} n_{j} $ implies that $ \limsup_{L \to \infty} L \cdot \eta( L )  > 0 $ holds. For this, we consider the subsequence $(L_{r})$ that is given by $ L_{r} := s_{k_{r}+1} = n_{0} \cdot \hdots \cdot n_{k_{r}}$. The description of $\eta(L)$ in Proposition~\ref{prop:Eta} yields $ \eta( L_{r} ) \geq ( c_{2} \cdot s_{\alleB( k_{r}-1 )} )^{-1} $. Thus we obtain
\[ \limsup_{L \to \infty} L \cdot \eta( L ) \geq \limsup_{r \to \infty} L_{r} \cdot \eta( L_{r} ) \geq \limsup_{r \to \infty} \frac{s_{k_{r}+1}}{c_{2} \cdot s_{\alleB( k_{r}-1 )} } = \frac{1}{c_{2}} \limsup_{r \to \infty} \frac{1}{\prod_{j=k_{r}+1}^{\alleB( k_{r}-1 ) - 1} n_{j} } > 0 \, . \]

To prove the converse, assume that no sequence $(k_{r})$ exists for which the product is bounded. Then $\lim_{k \to  \infty} \prod_{j = k+1}^{\alleB(k-1)-1} n_{j} = \infty$ holds and we will show that this implies $ \lim_{L \to \infty} L \cdot \eta( L ) = 0 $. For every $L$, we define $j$ as above by $ s_{j-1} < L \leq s_{j} $. In the case of $ s_{j-1} < L < 2s_{j-1} $, Proposition~\ref{prop:Eta} yields
\begin{align*}
L \cdot \eta( L )
& \leq \frac{L}{c_{1}} \cdot \min \Big\{ \frac{ \Big\lceil \frac{ 2s_{ j-1 } - L }{ s_{ j-2 } } \Big\rceil }{ s_{ \min \{ i > j-1 \, : \, a_{i} = a_{j-2} \} } } \; , \; \frac{1}{ s_{ \alleB( j-3 ) } } \Big\} \\
& < \frac{2s_{j-1}}{c_{1}} \cdot \frac{1}{ s_{ \alleB( j-3 ) } } \\
& = \frac{2}{c_{1}} \cdot \frac{ 1 }{ \prod_{i=j-1}^{\alleB( j-3 )-1} n_{i} } \xrightarrow{j \to \infty} 0 \, .
\end{align*}
In the case of $ 2s_{j-1} \leq L \leq s_{j} $, Proposition~\ref{prop:Eta} yields
\[ L \cdot \eta( L ) \leq \frac{L}{c_{1}} \cdot \frac{1}{ s_{\alleB( j-2 )} } \leq \frac{s_{j}}{c_{1}} \cdot \frac{1}{ s_{\alleB( j-2 )} } = \frac{1}{c_{1}} \cdot \frac{ 1 }{ \prod_{i=j}^{\alleB( j-2 )-1} n_{i} } \xrightarrow{j \to \infty} 0 \, . \qedhere \]
\end{proof}

We will now characterize the existence of such a sequence $(k_{r})$ in terms of the sequence $(m_{i})$. Recall that $(m_{i})$ was defined as those positions where $\alleB$ increases. Hence, for $ k = m_{i}+1, \hdots , m_{i+1} $ the value of $\alleB( k-1 ) $ is constant. Therefore, in this range, the product $ \prod_{j = k+1}^{\alleB(k-1)-1} n_{j} $ is minimal at $ k  =  m_{i+1}$.

\begin{prop}
\label{prop:SequeNMLimInf}
There exists a sequence $(k_{r})_{r} $ with $ \lim_{r \to \infty} k_{r} = \infty $ such that $ \prod_{j = k_{r}+1}^{\alleB(k_{r}-1)-1} n_{j} $ is bounded if and only if there exists a subsequence $(m_{i_{r}})_{r}$ of $(m_{i})$ such that $ \prod_{j = m_{i_{r}}+1}^{\alleB(m_{i_{r}}-1)-1} n_{j} $ is bounded.
\end{prop}

\begin{proof}
One implication is clear: If $ (m_{i_{r}}) $ is such a subsequence, then we define $ k_{r} := m_{i_{r}} $ and we are done. For the converse implication, assume that $(k_{r})$ is a sequence such that the product $ \prod_{j = k_{r}+1}^{\alleB(k_{r}-1)-1} n_{j} $ is bounded. For every $r$ there is an index $i_{r}$ such that $ m_{i_{r}-1} \leq k_{r}-1 < m_{i_{r}} $ holds. By definition of the sequence $(m_{i})$ this implies $ \alleB( m_{i_{r}-1} ) =  \alleB( k_{r}-1 ) = \alleB( m_{i_{r}} - 1 ) < \alleB( m_{i_{r}} ) $. This yields
\begin{align*}
\prod_{j = m_{i_{r}}+1}^{\alleB(m_{i_{r}}-1)-1} n_{j} &\leq \prod_{j = k_{r}+1}^{\alleB(k_{r}-1)-1} n_{j} \, ,
\end{align*}
which shows that $ \prod_{j = m_{i_{r}}+1}^{\alleB(m_{i_{r}}-1)-1} n_{j} $ is bounded.
\end{proof}

\begin{rem}
As shown in Proposition~\ref{coro:CharLinRepe}, a simple Toeplitz subshift is linear repetitive if and only if the sequence $ \left( \prod_{j = m_{i}+1}^{\alleB(m_{i}) - 1} n_{j} \right)_{i \geq 1} $ is bounded. By combining the Propositions~\ref{prop:BoundedProdBosh} and \ref{prop:SequeNMLimInf} we obtain that a simple Toeplitz subshift satisfies (B) if and only if there exists a subsequence $(m_{i_{r}})_{r}$ of $(m_{i})$ such that $ \left( \prod_{j = m_{i_{r}}+1}^{\alleB(m_{i_{r}}-1)-1} n_{j} \right)_{r} $ is bounded. In that sense, the Boshernitzan condition is a weaker analogue of linear repetitivity. 
\end{rem}

\begin{coro}
\label{coro:A3Bosh}
Let $\subshift$ be the subshift associated to a simple Toeplitz word with  $\card{\alphabEv}=3$. Then $\subshift$ satisfies (B) if and only if $ \liminf_{i \to \infty} n_{m_{i}+1} < \infty$ holds.
\end{coro}

\begin{proof}
The definitions of $\alleB$ and $(m_{i})$ imply $ a_{\alleB(m_{i-1})} \notin \{ a_{m_{i}} , \hdots , a_{\alleB(m_{i-1}) - 1} \}$. For sufficiently large $i$, we have $m_{i} \geq m_{i-1}+1 \geq \eventNr $ and obtain $ \card{ \{  a_{m_{i}} , \hdots , a_{\alleB(m_{i-1}) - 1} \} } = 2 $. Since $m_{i}$ is maximal with this property (cf. Proposition~\ref{prop:MUndK}) and consecutive letters are different, we obtain $ \alleB( m_{i}-1 ) = m_{i}+2$ for all such $i$ that $m_{i-1}+1 \geq \eventNr $ holds. Now the Propositions~\ref{prop:BoundedProdBosh} and \ref{prop:SequeNMLimInf} yield the claim.
\end{proof}

\begin{rem}
It was shown in \cite{LiuQu_Simple}, Corollary~4.3 that simple Toeplitz subshifts with $ \card{ \alphabEv } \geq 3 $, bounded sequence $n_{k}$ and 
\[ \lim_{k \to \infty} ( \alleB(k-2) - k ) = \infty \]
don't satisfy (B). However, for $\card{ \alphabEv } = 3 $ the assumptions of this statement cannot be satisfied, since $\card{ \alphabEv } = 3 $ implies $ \alleB( m_{i}-1 ) = m_{i} + 2 $, as we have seen in the proof of Corollary~\ref{coro:A3Bosh}. In particular, it follows from Corollary~\ref{coro:A3Bosh} that every simple Toeplitz subshift with $\card{\alphabEv} = 3$ and a bounded sequence $n_{k}$ satisfies (B). Note that this is not true for $\card{\alphabEv} > 3$. The 4-letter coding sequence
\[ a_{1} a_{2} a_{3} \hdots = (a b) c (a b)^{2} d (a b)^{3} c (a b)^{4} d \hdots \]
was given in \cite{LiuQu_Simple} as an example of a simple Toeplitz word with $\lim_{k \to \infty} ( \alleB(k-2) - k ) = \infty $. Hence, the associated subshift does not satisfy (B), independent of the sequence $(n_{k})$.
\end{rem}

We conclude this subsection with a discussion of generalized Grigorchuk subshifts. Recall the reason why they were defined with period lengths of the form $n_{k} = 2^{j_{k}}$: They are obtained from the constant sequence $ (n_{k})_{k \in \mathbb{N}_{0}} = (2, 2, 2, \hdots )$ and a sequence $(b_{k}) \in \alphab^{\mathbb{N}_{0}} $, where $b_{k} = b_{k+1}$ is allowed, by combining consecutive occurrences of the same letter. When a letter $b$ is repeated $j$ times, the period length of the resulting word is $ 2^{j} $. Conversely, when we have period lengths of the form $n_{k} = 2^{j_{k}}$, we can interpret the condition on $(n_{k})$ in Proposition~\ref{prop:BoundedProdBosh} as a condition on the sequence $(b_{k})$:

\begin{coro}
\label{coro:Power2Bosh}
Let $\subshift$ be a simple Toeplitz subshift with period lengths of the form $n_{k} = 2^{j_{k}}$ with $j_{k} \in \mathbb{N}$. Then $\subshift$ satisfies (B) if and only if there exists a constant $C$ and a sequence $(t_{r})_{r}$  with $ \lim_{r \to \infty} t_{r} = \infty $ such that for every $r$ the equality $ \{ b_{t_{r}} , \hdots , b_{t_{r}+C} \} = \{b_{i} : i \geq t_{r} \} $ holds.
\end{coro}

\begin{proof}
From Proposition~\ref{prop:BoundedProdBosh} and $ n_{k} = 2^{j_{k}} $ we obtain the equivalence
\begin{align*}
&\subshift \text{ satisfies (B)}\\
\Longleftrightarrow \; &\text{There is a sequence } (k_{r})_{r} \text{ with } \lim_{r \to \infty} k_{r} = \infty \text{ such that } \prod \nolimits_{i = k_{r}+1}^{\alleB(k_{r}-1)-1} n_{i} \text{ is bounded.} \\
\Longleftrightarrow \; &\text{There is a sequence } (k_{r})_{r} \text{ with } \lim_{r \to \infty} k_{r} = \infty  \text{ such that } \sum \nolimits_{i = k_{r}+1}^{\alleB(k_{r}-1)-1} j_{i} \text{ is bounded.}
\end{align*}
Let the letter $a_{k}$ correspond to the letters $b_{t} = \hdots = b_{t+ j_{k} - 1}$, the letter $a_{k+1}$ correspond to the letters $b_{t + j_{k}} = \hdots = b_{t+ j_{k} + j_{k+1}-1} $ and so on, such that the letter $a_{\alleB(k-1)}$ corresponds to the letters $ b_{t + j_{k} + \hdots + j_{\alleB(k-1) -1}} = \hdots = b_{j_{k} + \hdots + j_{\alleB(k-1)}-1} $. The definition of $\alleB(k-1) $ yields
\[ \alphab_{k} = \{ a_{k}, \hdots , a_{\alleB(k-1)} \} = \{ b_{t+ j_{k} - 1}, b_{t + j_{k}} , \hdots , b_{t + j_{k} + \hdots + j_{\alleB(k-1) -1}-1}, b_{t + j_{k} + \hdots + j_{\alleB(k-1) -1}} \} \, . \]

The set on the right hand side contains $2 + \sum_{i = k+1}^{\alleB(k-1) -1} j_{i}$ elements. When there exists a sequence $ k_{r} $ such that this sum is bounded, then every set of the form
\[ \{ b_{t_{r}+ j_{k_{r}} - 1}, \hdots , b_{t_{r} + j_{k_{r}} + \hdots + j_{\alleB(k_{r}-1) -1}} \} \]
has the claimed property. Conversely, assume that a sequence $t_{r}$ exists such that for every $r$ the property $ \{ b_{t_{r}} , \hdots , b_{t_{r}+C} \} = \{b_{i} : i \geq t_{r} \} $ holds. Let $ a_{k_{r}} $ denote the letter that corresponds to $b_{t_{r}}$. Then $a_{\alleB( k_{r}-1 ) - 1}$ corresponds to a letter $b_{t}$ with $t \leq t_{r}+C$. Since $j_{k}$ denotes the multiplicity of the letter $b_{t}$ corresponding to $a_{k}$, we obtain
\[  \sum \nolimits_{i = k_{r}+1}^{\alleB(k_{r}-1)-1} j_{i} \leq (t_{r}+C) - (t_{r}+1) +1 = C \, . \qedhere \]
\end{proof}

\begin{rem}
The criterion for the Boshernitzan condition for generalized Grigorchuk subshifts in the previous Corollary~\ref{coro:Power2Bosh} is similar to a criterion for the Boshernitzan condition for self similar groups by Nagnibeda and Pérez (\cite{NagnPerez_SchreierGr}): Let $G_{\omega}$ be the element of the family of Grigorchuk's groups (or, more general, a spinal group) that is defined by the sequence $\omega = (\omega_{t})_{t}$ of epimorphisms. The action of $G_{\omega}$ on the boundary $\partial T$ of the tree $T$ defines for every ray $\xi \in \partial T$ a rooted Schreier graph with root $\xi$. Consider the closure of the set of all these rooted Schreier graphs, except the one defined by the rightmost ray in the tree. Then $G_{\omega}$ acts on this set by shifting the root. Francoeur, Nagnibeda and Pérez show that this action satisfies the Boshernitzan condition if and only if there exists a constant $C$ and a sequence $ (t_{r})_{r} $ with $ \lim_{r \to \infty} t_{r} = \infty $ such that for every $r$ the equality $ \{ \omega_{t_{r}} , \hdots , \omega_{t_{r} + C} \} = \{ \omega_{i} : i \geq t_{r} \} $ holds.
\end{rem}

\subsection{Application: Jacobi Operators and Cocycles}

In this subsection, we briefly discuss the implications of the Boshernitzan condition for the spectrum of Jacobi operators associated to a subshift. First we recall the standard notions of transfer matrices and cocycles. By a result from \cite{LiuQu_Simple}, which is based on cocycles, Schrödinger operators on simple Toeplitz subshifts always have Cantor spectrum of Lebesgue measure zero. Finally, we quote a result from \cite{BeckPogo_SpectrJacobi}, which connects the Boshernitzan condition to Cantor spectrum of Lebesgue measure zero for Jacobi operators. As a corollary, the results about the Boshernitzan condition from the previous subsection allow conclusions about the spectrum.

We begin with the definitions of our objects of interest: Let $p : \subshift \to \mathbb{R} \setminus \{ 0 \}$ and $q : \subshift \to \mathbb{R}$ be continuous functions. For an element $ \omega \in \Omega $, the operator
\[ \Jac_{\omega} \! : \ell^{2}(\mathbb{Z}) \to \ell^{2}(\mathbb{Z}) \; , \;\; ( \Jac_{\omega} \psi )(k) =p(\Shift^{k} \omega) \psi(k-1) + q( \Shift^{k} \omega ) \psi(k) + p( \Shift^{k+1} \omega ) \psi(k+1) \]
is called the Jacobi operator associated to $\omega$. In the following, we will always assume that $p$ and $q$ take only finitely many values. Moreover we require the dynamical system that is defined by
\[ \widetilde{\subshift} := \left\{  \begin{pmatrix} p( \omega ) \\ q( \omega ) \end{pmatrix} \, : \omega \in \subshift \right\} \quad \text{and}\qquad \widetilde{ \Shift } : \widetilde{\subshift} \to \widetilde{\subshift} \, ,\; \begin{pmatrix} p( \omega ) \\ q( \omega ) \end{pmatrix} \mapsto \begin{pmatrix} p( \Shift \omega ) \\ q( \Shift \omega ) \end{pmatrix} \, ,\]
with the product topology on $\widetilde{\subshift}$, to be aperiodic. Since simple Toeplitz subshifts are minimal, the spectrum of $\Jac_{\omega}$ is, as a set, independent of $\omega \in \subshift$ and we call it the spectrum of the Jacobi operator on the subshift.

When investigating the spectrum of the Jacobi operator, an important tool are transfer matrices: A solution $\varphi$ to the eigenvalue equation $ \Jac_{\omega} \varphi = E \varphi $, with $ E \in \mathbb{R} $, is determined by its value at two consecutive positions. All other values can be computed from
\[ \varphi(k+1) = \frac{ ( E - q( \Shift^{k} \omega ) ) \varphi(k) - p(\Shift^{k} \omega) \varphi(k-1) }{p( \Shift^{k+1} \omega )}  \]
and
\[  \varphi(k-1) = \frac{ ( E - q( \Shift^{k} \omega ) ) \varphi(k) - p( \Shift^{k+1} \omega ) \varphi(k+1) }{p(\Shift^{k} \omega)} \, . \]
When we choose $\varphi(0)$ and $\varphi(1)$ as start values and express the above dependence with the help of matrices, we obtain for $k > 0$ the equality
\[ \Bigg(\begin{matrix} \varphi(k+1) \\ \varphi(k) \end{matrix} \Bigg) = \Bigg( \begin{matrix} \frac{ E - q( \Shift^{k} \omega ) }{p( \Shift^{k+1} \omega )} & - \frac{ p(\Shift^{k} \omega) }{p( \Shift^{k+1} \omega )} \\ 1 & 0 \end{matrix}\Bigg) \cdot \hdots \cdot \Bigg( \begin{matrix} \frac{ E - q( \Shift \omega ) }{p( \Shift^{2} \omega )} & - \frac{ p(\Shift \omega) }{p( \Shift^{2} \omega )} \\ 1 & 0 \end{matrix} \Bigg) \Bigg( \begin{matrix} \varphi(1) \\ \varphi(0) \end{matrix}\Bigg) \]
and for $k < 0$ the equality
\[ \Bigg( \begin{matrix} \varphi(k+1) \\ \varphi(k) \end{matrix} \Bigg) 
= \Bigg( \begin{matrix} 0 & 1 \\ - \frac{p( \Shift^{k+2} \omega )}{p(\Shift^{k+1} \omega)} & \frac{ E - q( \Shift^{k+1} \omega ) }{p(\Shift^{k+1} \omega)} \end{matrix} \Bigg) \cdot \hdots \cdot \Bigg( \begin{matrix} 0 & 1 \\ - \frac{p( \Shift \omega )}{p(\omega)} & \frac{ E - q(\omega) }{p(\omega)} \end{matrix} \Bigg) \Bigg( \begin{matrix} \varphi(1) \\ \varphi(0) \end{matrix} \Bigg) \, . \]
For a continuous map $M : \subshift \to \GL( 2, \mathbb{R} ) $, the associated cocycle $ M : \mathbb{Z} \times \subshift \to \GL( 2, \mathbb{R} ) $ is defined as 
\[ (n, \omega) \mapsto \begin{cases}
M( \Shift^{n-1} \omega ) \cdot \hdots \cdot M( \omega ) & \text{for } n > 0 \\
Id & \text{for } n = 0 \\
M( \Shift^{n} \omega )^{-1} \cdot \hdots \cdot M( \Shift^{-1} \omega )^{-1} & \text{for } n < 0
\end{cases} \, . \]
The transfer matrices are the maps
\[ M^{E} : \subshift \to \GL( 2, \mathbb{R} ) \, ,\; \omega \mapsto \begin{pmatrix} \frac{ E - q( \Shift \omega ) }{p( \Shift^{2} \omega )} & - \frac{ p(\Shift \omega) }{p( \Shift^{2} \omega )} \\ 1 & 0 \end{pmatrix} \, , \]
with $E \in \mathbb{R}$, and their associated cocycles are precisely the above matrix products that determine an eigenfunction $\varphi$ from two consecutive positions.

The properties of the transfer matrices are connected to properties of the spectrum of the Jacobi operator in the following way: A function $ M : \subshift \to \GL(2, \mathbb{R})$ is called uniform, if $ \lim_{n \to \infty} \frac{1}{n} \ln \left( \| M( n , \omega) \| \right) $ exists for all $ \omega \in \subshift $ and the convergence is uniform on $\subshift$ (cf. \cite{Furman_MultiErgodThm}, page~803). It was shown in \cite{BeckPogo_SpectrJacobi}, Theorem~3, that the spectrum of the Jacobi operator on a minimal, uniquely ergodic and aperiodic subshift is a Cantor set of Lebesgue measure zero if the transfer matrix is uniform for every $E \in \mathbb{R}$. Bases on this, Cantor spectrum can be deduced in the following cases:

\paragraph{Schrödinger operators}
For $p = 1$, the resulting operator 
\[ \Jac_{\omega} \! : \ell^{2}(\mathbb{Z}) \to \ell^{2}(\mathbb{Z}) \; , \;\; ( \Jac_{\omega} \psi )(k) = \psi(k-1) + q( \Shift^{k} \omega ) \psi(k) + \psi(k+1) \]
is called the Schrödinger operator associated to $\omega$. In \cite{LiuQu_Simple}, Theorem~1.1, it was shown for every simple Toeplitz subshift that the transfer matrix $M^{E}$ of the Schrödinger operator is uniform for all $E \in \mathbb{R}$. Thus, the spectrum of a Schrödinger operator on a simple Toeplitz subshift is always a Cantor set of Lebesgue measure zero.

\paragraph{Boshernitzan condition}
For the uniformity of $M^{E}$ it is sufficient that the subshift satisfies (B). More precisely, the following proposition is obtained as Corollary~4 in \cite{BeckPogo_SpectrJacobi}, using a result from \cite{DamLenz_Boshern}:

\begin{prop}[\cite{BeckPogo_SpectrJacobi}]
Let $(\subshift, \Shift)$ be a minimal, aperiodic subshift such that the Boshernitzan condition holds. Consider the family of the corresponding Jacobi operators $ \{ \Jac_{\omega} \}_{\omega \in \subshift} $ where the continuous maps p and q take finitely many values and the aperiodicity of the subshifts carries over to $( \widetilde{ \subshift }, \widetilde{ \Shift })$. Then the transfer matrix $M^{E} \, : \subshift \to \GL(2, \mathbb{R})$ is uniform for each $E \in \mathbb{R}$. In particular, the spectrum $\Sigma$ is a Cantor set of Lebesgue measure zero.
\end{prop}

By combining this statement with the results from Proposition~\ref{prop:BoundedProdBosh}, Corollary~\ref{coro:A3Bosh} and \cite{LiuQu_Simple}, Proposition~4.1 we obtain:

\begin{coro}
Let $\subshift$ be a simple Toeplitz subshift with coding $(a_{k})$ and $(n_{k})$. If there exists a sequence $(k_{r})_{r}$ of natural numbers with $ \lim_{r \to \infty} k_{r} = \infty $ such that $ \prod_{j = k_{r}+1}^{\alleB(k_{r}-1)-1} n_{j} $ is bounded, then the Jacobi operator on the subshift has Cantor spectrum of Lebesgue measure zero. This is in particular the case if either $ \card{ \alphab } = 2 $ holds or $\card{\alphabEv}=3$ and $ \liminf_{i \to \infty} n_{m_{i}+1} < \infty$ hold.
\end{coro}


\section*{Acknowledgement}

The author would like to thank his thesis advisor Daniel Lenz for his guidance, encouragement and support and Michael Baake for his interest in the authors work. He would also like to express his thanks to Siegfried Beckus, Rostislav Grigrochuk and Aitor Pérez for helpful discussions as well as to Tatiana Smirnova-Nagnibeda for fruitful suggestions and the hospitality at the University of Geneva. In addition, the author would like to thank the anonymous reviewer for the careful reading of the manuscript and valuable comments during the publication process at ``Ergodic Theory and Dynamical Systems''. The author gratefully acknowledges financial support in form of the Thuringian state scholarship (``Landesgraduiertenstipendium'') as well as of the DFG Research Training Group ``Quantum and Gravitational Fields'' (GRK 1523).

\newcommand{\etalchar}[1]{$^{#1}$}


\begin{thebibliography}{QRWX10}

\bibitem[BG13]{BaakeGrimm_Aperio}
M. Baake and U. Grimm.
\newblock {\em Aperiodic order. {V}olume 1: {A} mathematical invitation},
  volume 149 of {\em Encyclopedia of Mathematics and its Applications}.
\newblock Cambridge University Press, Cambridge, 2013.

\bibitem[BGN03]{BGN_FractalGrSets}
L. Bartholdi, R. Grigorchuk, and V. Nekrashevych.
\newblock From fractal groups to fractal sets.
\newblock In P. Grabner and W. Woess, editors, {\em Fractals in {G}raz 2001},
  Trends Math., pages 25--118. Birkh{\" a}user, Basel, 2003.

\bibitem[BG{\v S}03]{BGS_BranchGr}
L. Bartholdi, R.~I. Grigorchuk, and Z. {\v S}uni{\' k}.
\newblock Branch groups.
\newblock In M. Hazewinkel, editor, {\em Handbook of algebra, {V}ol. 3},
  volume~3 of {\em Handb. Algebr.}, pages 989--1112. Elsevier/North-Holland,
  Amsterdam, 2003.

\bibitem[BIST89]{BellIochScopTest_SpecProp}
J. Bellissard, B. Iochum, E. Scoppola, and D. Testard.
\newblock Spectral properties of one-dimensional quasi-crystals.
\newblock {\em Communications in Mathematical Physics}, 125(3):527--543, 1989.

\bibitem[BJL16]{BJL_ToeplModelSet}
M. Baake, T. J{\"a}ger, and D. Lenz.
\newblock {T}oeplitz flows and model sets.
\newblock {\em Bulletin of the London Mathematical Society}, 48(4):691--698,
  2016.

\bibitem[Bos85]{Boshernitzan_UniErgodic}
M. Boshernitzan.
\newblock A unique ergodicity of minimal symbolic flows with linear block
  growth.
\newblock {\em Journal d'Analyse Math{\'e}matique}, 44:77--96, 1984/85.

\bibitem[BP13]{BeckPogo_SpectrJacobi}
S. Beckus and F. Pogorzelski.
\newblock Spectrum of {L}ebesgue measure zero for {J}acobi matrices of
  quasicrystals.
\newblock {\em Mathematical Physics, Analysis and Geometry}, 16(3):289--308,
  2013.

\bibitem[Cas86]{Cas_SymbDyn}
M. Casdagli.
\newblock Symbolic dynamics for the renormalization map of a quasiperiodic
  {S}chr{\"o}dinger equation.
\newblock {\em Communications in Mathematical Physics}, 107(2):295--318, 1986.

\bibitem[CK97]{CassKar_ToeplWords}
J. Cassaigne and J. Karhum{\"a}ki.
\newblock {T}oeplitz words, generalized periodicity and periodically iterated
  morphisms.
\newblock {\em European Journal of Combinatorics}, 18(5):497--510, 1997.

\bibitem[DKL00]{DamKillipLenz_UniSpectProp3}
D. Damanik, R. Killip, and D. Lenz.
\newblock Uniform spectral properties of one-dimensional quasicrystals. {III}.
  {$\alpha$}-continuity.
\newblock {\em Communications in Mathematical Physics}, 212(1):191--204, 2000.

\bibitem[DKM{\etalchar{+}}17]{DKMSS_Regul-Article}
F. Dreher, M. Kesseb{\"o}hmer, A. Mosbach, T. Samuel, and M. Steffens.
\newblock Regularity of aperiodic minimal subshifts.
\newblock {\em Bulletin of Mathematical Sciences}, 2017.

\bibitem[DL99a]{DamLenz_UniSpectProp1}
D. Damanik and D. Lenz.
\newblock Uniform spectral properties of one-dimensional quasicrystals. {I}.
  {A}bsence of eigenvalues.
\newblock {\em Communications in Mathematical Physics}, 207(3):687--696, 1999.

\bibitem[DL99b]{DamLenz_UniSpectProp2}
D. Damanik and D. Lenz.
\newblock Uniform spectral properties of one-dimensional quasicrystals. {II}.
  {T}he {L}yapunov exponent.
\newblock {\em Letters in Mathematical Physics}, 50(4):245--257, 1999.

\bibitem[DL06]{DamLenz_Boshern}
D. Damanik and D. Lenz.
\newblock A condition of {B}oshernitzan and uniform convergence in the
  multiplicative ergodic theorem.
\newblock {\em Duke Mathematical Journal}, 133(1):95--123, 2006.

\bibitem[DLQ15]{DamLiuQu_PatternSturm}
D. Damanik, Q. Liu, and Y. Qu.
\newblock Spectral properties of {S}chr{\"o}dinger operators with pattern
  {S}turmian potentials.
\newblock preprint, arXiv:1511.03834, 2015.

\bibitem[Dow05]{Downa_OdomToepl}
T. Downarowicz.
\newblock Survey of odometers and {T}oeplitz flows.
\newblock In S. Kolyada, Y. Manin, and T. Ward, editors, {\em Algebraic and
  topological dynamics}, volume 385 of {\em Contemp. Math.}, pages 7--37. Amer.
  Math. Soc., Providence, RI, 2005.

\bibitem[Fur97]{Furman_MultiErgodThm}
A. Furman.
\newblock On the multiplicative ergodic theorem for uniquely ergodic systems.
\newblock {\em Annales de l'Institut Henri Poincar{\'e}. Probabilit{\'e}s et
  Statistiques}, 33(6):797--815, 1997.

\bibitem[GJ16]{GroeJae_PowEntr}
M. Gr{\"o}ger and T. J{\"a}ger.
\newblock Some remarks on modified power entropy.
\newblock In S. Kolyada, M. M{\"o}ller, P. Moree, and T. Ward, editors, {\em
  Dynamics and numbers}, volume 669 of {\em Contemp. Math.}, pages 105--122.
  Amer. Math. Soc., Providence, RI, 2016.

\bibitem[GKBY06]{GKBYM_MaxPatternToepl}
N. Gjini, T. Kamae, T. Bo, and X. Yu{-}Mei.
\newblock Maximal pattern complexity for {T}oeplitz words.
\newblock {\em Ergodic Theory and Dynamical Systems}, 26(4):1073--1086, 2006.

\bibitem[GKM{\etalchar{+}}19]{GKMSS_AperioJarnik}
M. Gr{\" o}ger, M. Kesseb{\" o}hmer, A. Mosbach, T. Samuel, and M. Steffens.
\newblock A classification of aperiodic order via spectral metrics and
  {J}arn{\' i}k sets.
\newblock {\em Ergodic Theory and Dynamical Systems}, 39(11):3031--3065, 2019.

\bibitem[GLN17a]{GLN_Combinat}
R. Grigorchuk, D. Lenz, and T. Nagnibeda.
\newblock Combinatorics of the subshift associated with {G}rigorchuk's group.
\newblock {\em Trudy Matematicheskogo Instituta Imeni V. A. Steklova},
  297(Poryadok i Khaos v Dinamicheskikh Sistemakh):158--164, 2017.
\newblock English translation: \textit{Proceedings of the Steklov Institute of
  Mathematics}, 297(1):138--144, 2017.

\bibitem[GLN17b]{GLN_Survey}
R. Grigorchuk, D. Lenz, and T. Nagnibeda.
\newblock {S}chreier graphs of {G}rigorchuk's group and a subshift associated
  to a nonprimitive substitution.
\newblock In T. Ceccherini-Silberstein, M. Salvatori, and E. Sava-Huss,
  editors, {\em Groups, graphs and random walks}, volume 436 of {\em London
  Math. Soc. Lecture Note Ser.}, pages 250--299. Cambridge Univ. Press,
  Cambridge, 2017.

\bibitem[GLN18]{GLN_SpectraSchreierAndSO}
R. Grigorchuk, D. Lenz, and T. Nagnibeda.
\newblock Spectra of {S}chreier graphs of {G}rigorchuk's group and
  {S}chroedinger operators with aperiodic order.
\newblock {\em Mathematische Annalen}, 370(3-4):1607--1637, 2018.

\bibitem[Gri80]{Grig_BurnsideRuss}
R.~I. Grigor{\v c}uk.
\newblock On {B}urnside's problem on periodic groups.
\newblock {\em Funktsional{$'$}ny{\u \i} Analiz i ego Prilozheniya},
  14(1):53--54, 1980.
\newblock English translation: \textit{Functional Analysis and Its
  Applications}, 14(1):41--43, 1980.

\bibitem[Gri84]{Grig_DegrOfGrowthRuss}
R.~I. Grigorchuk.
\newblock Degrees of growth of finitely generated groups and the theory of
  invariant means.
\newblock {\em Izvestiya Akademii Nauk SSSR. Seriya Matematicheskaya},
  48(5):939--985, 1984.
\newblock English translation: \textit{Mathematics of the USSR-Izvestiya},
  25(2):259--300, 1985.

\bibitem[JK69]{JacobsKeane_01Toeplitz}
K. Jacobs and M. Keane.
\newblock {$0-1$}-sequences of {T}oeplitz type.
\newblock {\em Zeitschrift f{\"u}r Wahrscheinlichkeitstheorie und Verwandte
  Gebiete}, 13:123--131, 1969.

\bibitem[KKT83]{KKT_LocalProbl1D}
M. Kohmoto, L.~P. Kadanoff, and C. Tang.
\newblock Localization problem in one dimension: mapping and escape.
\newblock {\em Physical Review Letters}, 50(23):1870--1872, 1983.

\bibitem[KLS15]{MathAperioOrd}
J. Kellendonk, D. Lenz, and J. Savinien, editors.
\newblock {\em Mathematics of aperiodic order}, volume 309 of {\em Progress in
  Mathematics}.
\newblock Birkh{\"a}user/Springer, Basel, 2015.

\bibitem[Kos98]{Kosk_ComplSuites}
M. Koskas.
\newblock Complexit{\'e}s de suites de {T}oeplitz.
\newblock {\em Discrete Mathematics}, 183(1-3):161--183, 1998.

\bibitem[KZ02a]{KamZamb_MaxPattCompl}
T. Kamae and L. Zamboni.
\newblock Maximal pattern complexity for discrete systems.
\newblock {\em Ergodic Theory and Dynamical Systems}, 22(4):1201--1214, 2002.

\bibitem[KZ02b]{KamZamb_SequEntro}
T. Kamae and L. Zamboni.
\newblock Sequence entropy and the maximal pattern complexity of infinite
  words.
\newblock {\em Ergodic Theory and Dynamical Systems}, 22(4):1191--1199, 2002.

\bibitem[LQ11]{LiuQu_Simple}
Q. Liu and Y. Qu.
\newblock Uniform convergence of {S}chr{\"o}dinger cocycles over simple
  {T}oeplitz subshift.
\newblock {\em Annales Henri Poincar{\'e}}, 12(1):153--172, 2011.

\bibitem[LQ12]{LiuQu_Bounded}
Q. Liu and Y. Qu.
\newblock Uniform convergence of {S}chr{\"o}dinger cocycles over bounded
  {T}oeplitz subshift.
\newblock {\em Annales Henri Poincar{\'e}}, 13(6):1483--1500, 2012.

\bibitem[MB15]{MBon_TopoFullGr}
N. Matte~Bon.
\newblock Topological full groups of minimal subshifts with subgroups of
  intermediate growth.
\newblock {\em Journal of Modern Dynamics}, 9:67--80, 2015.

\bibitem[MH38]{MorseHedl_SymbDyn}
M. Morse and G.~A. Hedlund.
\newblock Symbolic {D}ynamics.
\newblock {\em American Journal of Mathematics}, 60(4):815--866, 1938.

\bibitem[Nek05]{Nekrash_SelfSimGr}
V. Nekrashevych.
\newblock {\em Self-similar groups}, volume 117 of {\em Mathematical Surveys
  and Monographs}.
\newblock American Mathematical Society, Providence, RI, 2005.

\bibitem[NP19]{NagnPerez_SchreierGr}
T. Nagnibeda and A. Perez.
\newblock {S}chreier graphs of spinal groups and associated subshifts.
\newblock private communication/in preparation, 2019.

\bibitem[QRWX10]{QRWX_PatCompDDim}
Y. Qu, H. Rao, Z. Wen, and Y. Xue.
\newblock Maximal pattern complexity of higher dimensional words.
\newblock {\em Journal of Combinatorial Theory. Series A}, 117(5):489--506,
  2010.

\bibitem[SPR{\etalchar{+}}83]{OPRSS_1DSEqn}
O. S., R. Pandit, D. Rand, H.~J. Schellnhuber, and E.~D. Siggia.
\newblock One-dimensional {S}chr{\"o}dinger equation with an almost periodic
  potential.
\newblock {\em Physical Review Letters}, 50(23):1873--1876, 1983.

\bibitem[S{\"u}t87]{Suto_QuasiPerSO}
A. S{\"u}t{\H o}.
\newblock The spectrum of a quasiperiodic {S}chr{\"o}dinger operator.
\newblock {\em Communications in Mathematical Physics}, 111(3):409--415, 1987.

\bibitem[S{\"u}t89]{Suto_SingContSpect}
A. S{\"u}t{\H o}.
\newblock Singular continuous spectrum on a {C}antor set of zero {L}ebesgue
  measure for the {F}ibonacci {H}amiltonian.
\newblock {\em Journal of Statistical Physics}, 56(3-4):525--531, 1989.

\bibitem[Vor10]{Voro_SubstGrigo}
Y. Vorobets.
\newblock On a substitution subshift related to the {G}rigorchuk group.
\newblock {\em Trudy Matematicheskogo Instituta Imeni V. A. Steklova},
  271(Differentsial{$'$}nye Uravneniya i Topologiya. II):319--334, 2010.
\newblock English translation: \textit{Proceedings of the Steklov Institute of
  Mathematics}, 271(1):306--321, 2010.

\bibitem[Wil84]{Wil_ToepNotUniqErgod}
S. Williams.
\newblock {T}oeplitz minimal flows which are not uniquely ergodic.
\newblock {\em Zeitschrift f{\"u}r Wahrscheinlichkeitstheorie und Verwandte
  Gebiete}, 67(1):95--107, 1984.

\end{thebibliography}
\end{document}